\documentclass[11pt]{amsart}
\usepackage{amsmath}
\usepackage{amscd}
\usepackage{amssymb}
\usepackage{amsthm}
\usepackage{dsfont} 
\usepackage{appendix} 
\usepackage{verbatim}
\usepackage[normalem]{ulem} 
\usepackage[all]{xy} 
\usepackage{color}

\begin{document}


\theoremstyle{plain}
\newtheorem{theorem}{Theorem}[subsection]
\newtheorem{introtheorem}{Theorem}

\theoremstyle{plain}
\newtheorem{proposition}[theorem]{Proposition}

\theoremstyle{plain}
\newtheorem{corollary}[theorem]{Corollary}

\theoremstyle{plain}
\newtheorem{lemma}[theorem]{Lemma}

\theoremstyle{plain}
\newtheorem{definition}[theorem]{Definition}

\theoremstyle{remark}
\newtheorem{remind}[theorem]{Reminder}

\theoremstyle{definition}
\newtheorem{condition}[theorem]{Condition}

\theoremstyle{definition}
\newtheorem{construction}[theorem]{Construction}

\theoremstyle{definition}
\newtheorem{example}[theorem]{Example}

\theoremstyle{definition}
\newtheorem{notation}[theorem]{Notation}

\theoremstyle{definition}
\newtheorem{convention}[theorem]{Convention}

\theoremstyle{definition}
\newtheorem{assumption}[theorem]{Assumption}
\newtheorem{isoassumption}[theorem]{Isomorphism Assumption}

\newtheorem{finassumption}[theorem]{Finiteness Assumption}
\newtheorem{indhypothesis}[theorem]{Inductive Hypothesis}

\theoremstyle{remark}
\newtheorem{remark}[theorem]{Remark}




\newcommand{\mathscr}[1]{\mathcal{#1}}
\newcommand{\TODO}{{\color{red} TODO}}
\newcommand{\CHANGE}{{\color{red} CHANGE}}
\newcommand{\annette}[1]{{\color{red}Annette: #1 }}
\newcommand{\simon}[1]{{\color{green}Simon: #1 }}
\newcommand{\giuseppe}[1]{{\color{blue}Giuseppe: #1 }}
\newcommand{\tensor}{\ox}
\newcommand{\Fil}{\tn{Fil}}
\newcommand{\Fh}{\shfF}
\newcommand{\GrVec}{\textnormal{GrVec}}
\newcommand{\isom}{\cong}
\newcommand{\isocan}{\xrightarrow{\hspace{1.85pt}\sim \hspace{1.85pt}}}
\newcommand{\red}{\mathrm{red}}

\newcommand{\ot}{\otimes}
\newcommand{\Gr}{\textnormal{Gr}}
\newcommand{\x}{\times}
\newcommand{\Ga}{\mathbb{G}_{a}}
\newcommand{\Gm}{\mathbb{G}_{m}}
\newcommand{\uGm}{\ul{\Gm}}
\newcommand{\Gal}{\textnormal{Gal}}
\newcommand{\tr}{\textnormal{tr}}
\newcommand{\Lf}{\mathbb{L}}
\newcommand{\id}{\mathrm{id}}
\newcommand{\chr}{\mathrm{char}}
\newcommand{\Aut}{\textnormal{Aut}}
\newcommand{\End}{\textnormal{End}}
\newcommand{\Sym}{\textnormal{Sym}}
\newcommand{\OSym}{\textnormal{coSym}}
\newcommand{\Alt}{\tn{Alt}}
\newcommand{\Altn}{\tn{Alt}^{n}}
\newcommand{\Symn}{\Sym^{n}}
\newcommand{\Altd}{\tn{Alt}^{d}}
\newcommand{\Symd}{\Sym^{d}}
\newcommand{\Symi}{\Sym^{i}}
\newcommand{\Altk}{\Alt^{k}}
\newcommand{\Alti}{\Alt^{i}}
\newcommand{\Symk}{\Sym^{k}}
\newcommand{\Symm}{\Sym^{m}}

\newcommand{\N}{\mathbb{N}}
\newcommand{\Z}{\mathbb{Z}}
\newcommand{\Q}{\mathbb{Q}}
\newcommand{\Ql}{\mathbb{Q}_{\ell}}
\newcommand{\C}{\mathbb{C}}

\newcommand{\A}{\mathbb{A}}
\newcommand{\cata}{\shf{A}}
\newcommand{\shfF}{\shf{F}}

\newcommand{\shfA}{\shf{A}}
\newcommand{\shfB}{\shf{B}}
\newcommand{\shfG}{\shf{G}}

\newcommand{\one}{\mathds{1}}
\newcommand{\tatt}{\one(1)[2]}
\newcommand{\Aff}{\mathbb{A}}
\newcommand{\Hom}{\textnormal{Hom}}
\newcommand{\Ob}{\textnormal{Ob}}
\newcommand{\Mor}{\textnormal{Mor}}
\newcommand{\op}{\tn{op}}
\newcommand{\Ext}{\tn{Ext}}
\newcommand{\Spec}{\textnormal{Spec}}
\newcommand{\Pro}{\textnormal{Pro}}
\newcommand{\lra}{\longrightarrow}
\newcommand{\vp}{\varphi}
\newcommand{\eff}{\tn{eff}}

\newcommand{\GpSch}{\tn{\tbf{GpSch}}}
\newcommand{\LgGpSch}{\tn{\tbf{LgGpSch}}}
\newcommand{\GpSchS}{\tn{\tbf{GpSch}}/S}
\newcommand{\LgGpSchS}{\tn{\tbf{LgGpSch}}/S}
\newcommand{\GpSchk}{\tn{\tbf{GpSch}}/k}
\newcommand{\CGSFTk}{\tn{\tbf{CmGSch}}/k}
\newcommand{\Ab}{\tn{\tbf{Ab}}}
\newcommand{\AbQ}{\Ab_\Q}
\newcommand{\sAb}{\tn{\tbf{sAb}}}
\newcommand{\sAbQ}{{\sAb_\Q}}
\newcommand{\gAb}{\tn{\tbf{gAb}}}
\newcommand{\SAb}{\tn{\tbf{SAb}}}
\newcommand{\SAbk}{\SAb(k)}
\newcommand{\SAbS}{\SAb(S)}
\newcommand{\AbS}{\tn{\tbf{AbSch}}}
\newcommand{\AbV}{\tn{\tbf{AbVar}}}
\newcommand{\AbVk}{\tn{\tbf{AbVar}}/k}
\newcommand{\AbVkQ}{\tn{\tbf{AbVar}}_{\Q}/k}
\newcommand{\AVI}{\tn{\tbf{AbVarIsog}}}
\newcommand{\AVIk}{\tn{\tbf{AbVarIsog}}/k}
\newcommand{\AVIkQ}{(\tn{\tbf{AbVarIsog}}/k)^{\Q}}
\newcommand{\tAbS}{{}^{\tn{t}}\tn{\tbf{AbS}}}
\newcommand{\CGS}{\tn{\tbf{cGrp}}}
\newcommand{\CGSQ}{\CGS_\Q}
\newcommand{\CGSS}{\tn{\tbf{CommGpSch}}_{S}}
\newcommand{\fgAb}{\tn{\tbf{fgAb}}}
\newcommand{\Open}{\tn{\tbf{Open}}}
\newcommand{\Soe}{\tn{\tbf{S}}^{\tn{eff}}_{1}}
\newcommand{\So}{\tn{\tbf{S}}_{1}}
\newcommand{\res}{\tn{res}}
\newcommand{\an}{\tn{an}}
\newcommand{\Prim}{\tn{Prim}}
\newcommand{\rsa}{\rightsquigarrow}
\newcommand{\lbk}{\lbrack}
\newcommand{\rbk}{\rbrack}

\newcommand{\aXYZ}{\alpha_{X, Y}^{Z}}
\newcommand{\aXYU}{\alpha_{X, Y}^{U}}

\newcommand{\tXYZ}{\theta_{X, Y}^{Z}}
\newcommand{\tXYU}{\theta_{X, Y}^{U}}
\newcommand{\tXYV}{\theta_{X, Y}^{V}}
\newcommand{\tXYW}{\theta_{X, Y}^{W}}

\newcommand{\Sets}{\tn{\tbf{Sets}}}
\newcommand{\SFL}{S_{\tn{FL}}}
\newcommand{\kFL}{\Speck_{\tn{FL}}}
\newcommand{\ket}{\Speck_{\et}}
\newcommand{\D}{\tn{D}}
\newcommand{\DD}{\mathbb{D}}
\newcommand{\GoS}{\mathbb{G}/S}
\newcommand{\uG}{\ul{\mathbb{G}}}
\newcommand{\uGS}{\ul{\mathbb{G}}/S}
\newcommand{\Dm}{\tn{D}^{-}}
\newcommand{\Dp}{\tn{D}^{+}}
\newcommand{\Db}{\tn{D}^{\tn{b}}}
\newcommand{\bDm}{\tn{\tbf{D}}^{-}}
\newcommand{\bDmc}{\tn{\tbf{D}}^{-}_{\tn{Cor}}}
\newcommand{\bDms}{\tn{\tbf{D}}^{-}_{\tn{St}}}

\newcommand{\bS}{\tn{\tbf{S}}}
\newcommand{\bSc}{\tn{\tbf{S}}_{\tn{Cor}}}
\newcommand{\bSt}{\tn{\tbf{S}}_{\tn{St}}}

\newcommand{\qcoh}{\tn{qcoh}}
\newcommand{\Dbqcoh}{\Db_{\qcoh}}

\newcommand{\Cb}{\cly{C}^{\tn{b}}}
\newcommand{\Kb}{\cly{K}^{\tn{b}}}

\newcommand*{\longhookrightarrow}{\ensuremath{\lhook\joinrel\relbar\joinrel\rightarrow}}
\newcommand{\lhra}{\longhookrightarrow}

\newcommand{\gr}{\tn{gr}}
\newcommand{\grwn}{\tn{gr}^{W}_{n}}
\newcommand{\grwp}{\tn{gr}^{W}_{p}}
\newcommand{\grwo}{\tn{gr}^{W}_{1}}
\newcommand{\grwi}{\tn{gr}^{W}_{i}}
\newcommand{\grw}{\tn{gr}^{W}}
\newcommand{\et}{\tn{\'{e}t}}
\newcommand{\fl}{\tn{fl}}
\newcommand{\fppf}{\tn{fppf}}
\newcommand{\nis}{\tn{Nis}}
\newcommand{\eh}{\tn{\'{e}h}}
\newcommand{\Dlg}{\tn{Dlg}}
\newcommand{\Tori}{\tn{\tbf{Tori}}}
\newcommand{\fgfrAb}{\tn{\tbf{fgfrAb}}}
\newcommand{\fgf}{\tn{\tbf{fgfrAb}}}
\newcommand{\cch}{\chi_{\bl}}
\newcommand{\Fnlm}{F^{\nu}_{\lambda, \mu}}
\newcommand{\etl}{\'{e}tale}
\newcommand{\Etl}{\'{E}tale}
\newcommand{\Vo}{Voevodsky}
\newcommand{\Vov}{\tn{Vo}}
\newcommand{\Dgl}{D\'{e}glise}
\newcommand{\SeAV}{Serre-Albanese variety}
\newcommand{\SeAT}{Serre-Albanese torsor}
\newcommand{\ASc}{Albanese scheme}
\newcommand{\sesq}{short exact sequence}
\newcommand{\Cech}{\check{C}ech}
\newcommand{\cech}{\check{c}ech}
\newcommand{\sng}{\tn{sing}}
\newcommand{\co}{c_{1}}
\newcommand{\cosng}{c_{1}^{\sng}}
\newcommand{\comot}{c_{1}^{\mot}}
\newcommand{\mot}{\tn{mot}}
\newcommand{\Hmot}{H_{\mot}}

\newcommand{\ner}{N\'{e}ron}
\newcommand{\ners}{N\'{e}ron-Severi}
\newcommand{\nsv}{Nisnevich}
\newcommand{\pnc}{Poincar\'{e}}
\newcommand{\gc}{Grothendieck-Chow}
\newcommand{\gcm}{Grothendieck-Chow motive}
\newcommand{\gcms}{Grothendieck-Chow motives}
\newcommand{\fa}{Fourier analysis}
\newcommand{\rie}{Riemann}
\newcommand{\rien}{Riemannian}
\newcommand{\rienme}{Riemannian metric}
\newcommand{\rienmes}{Riemannian metrics}
\newcommand{\rienma}{Riemannian manifold}
\newcommand{\rienmas}{Riemannian manifolds}

\newcommand{\Coh}{\tn{\tbf{Coh}}}
\newcommand{\Cop}{\tn{\tbf{Coh}}(\Pj^{1})}
\newcommand{\Torp}{\tn{\tbf{Tor}}(\Pj^{1})}
\newcommand{\Tor}{\tn{\tbf{Tor}}}
\newcommand{\slt}{\idl{sl}_{2}}

\newcommand{\De}{\tn{DM}^{\tn{eff}}_{\,\tn{\pcl{X}}}}
\newcommand{\Dg}{\tn{DM}_{\tn{gm}}}
\newcommand{\Deg}{\tn{DM}^{\tn{eff}}_{\tn{gm}}}
\newcommand{\Chow}{\tn{\tbf{Chow}}} 
\newcommand{\Chowk}{\Chow(k)}
\newcommand{\tChowk}{{}^{\tn{t}}\Chow(k)}
\newcommand{\Chowuk}{\tn{\tbf{Chow}}_{\tn{un}}(k)}
\newcommand{\ChowS}{\tn{\tbf{Chow}}(S)}
\newcommand{\Chowe}{\tn{\tbf{Chow}}^{\tn{eff}}}
\newcommand{\ChoweS}{\tn{\tbf{Chow}}^{\tn{eff}}(S)}
\newcommand{\Chowek}{\Chow^{\tn{eff}}(k)}
\newcommand{\ChowZek}{\Chow^{\tn{eff}}_{\Z}(k)}
\newcommand{\tChowek}{{}^{\tn{t}}\Chow^{\tn{eff}}(k)}
\newcommand{\ChowkQ}{\tn{\tbf{Chow}}(k)_{\Q}}
\newcommand{\ChowekQ}{\tn{\tbf{Chow}}^{\tn{eff}}(k)_{\Q}}
\newcommand{\CHp}{\tn{CH}^{p}}
\newcommand{\CHo}{\tn{CH}_{0}}
\newcommand{\tCHo}{\wt{\tn{CH}}_{0}}
\newcommand{\req}{\equiv_{\tn{rat}}}
\newcommand{\DMe}{\DM^{\tn{eff}}} 
\newcommand{\DTM}{\tn{DTM}} 
\newcommand{\DTMn}{\tn{DTM}_{n}} 
\newcommand{\DTMet}{\tn{DTM}_{\et}} 
\newcommand{\DTMeet}{\tn{DTM}^{\tn{eff}}_{\,\tn{\pcl{X}}, \et}}
\newcommand{\DTMetk}{\tn{DTM}_{\et}(k)} 
\newcommand{\DTMeetk}{\tn{DTM}^{\tn{eff}}_{\,\tn{\pcl{X}}, \et}(k)}
\newcommand{\DMeet}{\DMe_{-, \et}}

\newcommand{\DMen}{\DMe_{\,\tn{\pcl{X}}, \nis}}
\newcommand{\DMett}{\DMe_{\,\tn{\pcl{X}}, \tau}}
\newcommand{\DMettk}{\DMett(k)}
\newcommand{\DMenk}{\DMen(k)}
\newcommand{\DMenkQ}{\DMen(k, \Q)}
\newcommand{\DMenkR}{\DMen(k, R)}

\newcommand{\DMettkR}{\DMett(k, R)}
\newcommand{\DMettkZ}{\DMett(k, \Z)}
\newcommand{\DMettkQ}{\DMett(k, \Q)}
\newcommand{\DMeetk}{\DMeet(k)}
\newcommand{\DMeetkQ}{\DMeet(k, \Q)}
\newcommand{\DMeetKQ}{\DMeet(K, \Q)}

\newcommand{\DMeetkZ}{\DMeet(k, \Z)}
\newcommand{\DMeetkR}{\DMeet(k, R)}
\newcommand{\DMeetA}{\DMeet(A)}
\newcommand{\DMeetX}{\DMeet(X)}
\newcommand{\DMX}{\DM(X)}
\newcommand{\DMY}{\DM(Y)}
\newcommand{\DMgmX}{\DMgm(X)}
\newcommand{\DMgmY}{\DMgm(Y)}

\newcommand{\DMetk}{\tn{\tbf{DM}}_{\,\tn{\pcl{X}}, \et}(k)}
\newcommand{\DMetkQ}{\tn{\tbf{DM}}_{\,\tn{\pcl{X}}, \et}(k)_{\Q}}
\newcommand{\DMeget}{\tn{\tbf{DM}}^{\tn{eff}}_{\tn{gm}, \et}}
\newcommand{\DMegetk}{\tn{\tbf{DM}}^{\tn{eff}}_{\tn{gm}, \et}(k)}
\newcommand{\DMgetk}{\tn{\tbf{DM}}_{\tn{gm}, \et}(k)}

\newcommand{\DMentkQ}{\tn{\tbf{DM}}^{\tn{eff}}_{\,\tn{\pcl{X}}, \nis}(k)_{\Q}}
\newcommand{\DMegetQ}{\tn{DM}^{\tn{eff}}_{\tn{gm}, \et} \ox \Q}
\newcommand{\DMegetkQ}{\tn{DM}^{\tn{eff}}_{\tn{gm}, \et}(k) \ox \Q}
\newcommand{\DMg}{\tn{\tbf{DM}}_{\tn{gm}}}
\newcommand{\DMeg}{\tn{\tbf{DM}}^{\tn{eff}}_{\tn{gm}}}
\newcommand{\DMgS}{\tn{\tbf{DM}}_{\tn{gm}}(S)}
\newcommand{\DMeS}{\tn{\tbf{DM}}^{\tn{eff}}(S)}
\newcommand{\DMegS}{\tn{\tbf{DM}}^{\tn{eff}}_{\tn{gm}}(S)}

\newcommand{\DMegQ}{\DMeg \ox \Q}
\newcommand{\dDMeg}{d_{\le 1}\tn{DM}^{\tn{eff}}_{\tn{gm}}}
\newcommand{\dDMeget}{d_{\le 1}\tn{DM}^{\tn{eff}}_{\tn{gm}, \et}}
\newcommand{\dDMegetk}{d_{\le 1}\tn{DM}^{\tn{eff}}_{\tn{gm}, \et}(k)}
\newcommand{\as}{\alpha_{*}}

\newcommand{\Mot}{\tn{Mot}}
\newcommand{\Motk}{\Mot/k}

\newcommand{\DbM}{\Db(\Mok)}
\newcommand{\DbMp}{\Db(\Mok[1/p])}
\newcommand{\DbMQ}{\Db(\Mok \ox \Q)}
\newcommand{\CbMp}{\Cb(\Mok[1/p])}
\newcommand{\CbMQ}{\Cb(\Mok \ox \Q)}
\newcommand{\dlo}{d_{\le 1}} 
\newcommand{\Dlo}{\tn{D}_{\le 1}}
\newcommand{\Deto}{\tn{D}^{\et}_{\le 1}}
\newcommand{\Dnto}{\tn{D}^{\tn{Nis}}_{\le 1}}
\newcommand{\Mo}{\cly{M}_{1}}
\newcommand{\Mof}{M_{1}}
\newcommand{\hof}[1]{\mathfrak{h}_{#1}}
\newcommand{\Mi}[1]{M_{#1}}
\newcommand{\Do}{D_{1}}
\newcommand{\shtG}{\sht{G}}
\newcommand{\shtH}{\sht{H}}
\newcommand{\Mog}{\mathbf{M}_{1}}
\newcommand{\Mg}{\mathbf{M}}
\newcommand{\Mogr}{\mathbf{M}_{1}^{r}}
\newcommand{\Mgr}{\mathbf{M}^{r}}
\newcommand{\Mea}{\cly{M}_{\tn{anc}}^{\tn{eff}}}
\newcommand{\tMeo}{{}^{t}\cly{M}_{1}^{\tn{eff}}}
\newcommand{\ctMeo}{{}_{t}\cly{M}_{1}^{\tn{eff}}}
\newcommand{\tMo}{{}^{t}\cly{M}_{1}}
\newcommand{\Mok}{\cly{M}_{1}(k)}
\newcommand{\MoS}{\cly{M}_{1}(S)}
\newcommand{\Meak}{\cly{M}_{\tn{anc}}^{\tn{eff}}(k)}
\newcommand{\tMeok}{{}^{t}\cly{M}_{1}^{\tn{eff}}(k)}
\newcommand{\tMok}{{}^{t}\cly{M}_{1}(k)}
\newcommand{\ctMo}{{}_{t}\cly{M}_{1}}
\newcommand{\ctMok}{{}_{t}\cly{M}_{1}(k)}

\newcommand{\tH}{{}^{t} H}
\newcommand{\tHn}{{}^{t} H^{n}}
\newcommand{\tHln}{{}^{t} H_{n}}
\newcommand{\Mz}{\cly{M}_{0}}
\newcommand{\Mzk}{\cly{M}_{0}(k)}
\newcommand{\MzS}{\cly{M}_{0}(S)}
\newcommand{\tMz}{{}^{t}\cly{M}_{0}}
\newcommand{\tMzk}{{}^{t}\cly{M}_{0}(k)}

\newcommand{\MS}{\cly{M}(S)}
\newcommand{\MChS}{\cly{M}^{0}(S)}
\newcommand{\MChpS}{\cly{M}^{0}_{+}(S)}

\newcommand{\LAlb}{\tn{LAlb}}
\newcommand{\LAlbc}{\tn{LAlb}^{\tn{c}}}
\newcommand{\LAlbcs}{\tn{LAlb}^{\tn{c}} \;}
\newcommand{\LiAlb}{\tn{L}_{i}\tn{Alb}}
\newcommand{\LoAlb}{\tn{L}_{1}\tn{Alb}}
\newcommand{\LtAlb}{\tn{L}_{2}\tn{Alb}}
\newcommand{\LAlbi}[1]{\tn{L}_{#1}\tn{Alb}}
\newcommand{\LiAlbc}{\tn{L}_{i}\tn{Alb}^{\tn{c}}}
\newcommand{\LzAlb}{\tn{L}_{0}\tn{Alb}}
\newcommand{\LAlbQ}{\tn{LAlb}^{\Q}}
\newcommand{\LiAlbQ}{\LiAlb^{\Q}}

\newcommand{\Pic}{\tn{Pic}}
\newcommand{\RPic}{\tn{RPic}}
\newcommand{\RPicc}{\tn{RPic}^{\tn{c}}}
\newcommand{\RiPic}{\tn{R}_{i}\tn{Pic}}

\newcommand{\RPici}[1]{\tn{R}_{#1}\tn{Pic}}

\newcommand{\RiPicc}{\tn{R}_{i}\tn{Pic}^{\tn{c}}}
\newcommand{\LAlbs}{\tn{LAlb} \;}
\newcommand{\RPics}{\tn{RPic} \;}
\newcommand{\AXk}{\cly{A}_{X/k}}
\newcommand{\AGk}{\cly{A}_{G/k}}
\newcommand{\AUk}{\cly{A}_{U/k}}
\newcommand{\AzXk}{\cly{A}^{0}_{X/k}}
\newcommand{\AsXk}{\cly{A}^{*}_{X/k}}
\newcommand{\uAXk}{\ul{\cly{A}}_{X/k}}
\newcommand{\uAsXk}{\ul{\cly{A}}^{*}_{X/k}}
\newcommand{\PicXk}{\tn{Pic}_{X/k}}
\newcommand{\PicCk}{\tn{Pic}_{C/k}}
\newcommand{\uPicXk}{\ul{\tn{Pic}}_{X/k}}
\newcommand{\uPicCk}{\ul{\tn{Pic}}_{C/k}}
\newcommand{\wPicXk}{\wPic_{X/k}}
\newcommand{\wPicCk}{\wPic_{C/k}}

\newcommand{\PicoXk}{\tn{Pic}^{\circ}_{X/k}}
\newcommand{\uPicoXk}{\ul{\tn{Pic}}^{\circ}_{X/k}}

\newcommand{\NSXk}{\tn{NS}_{X/k}}
\newcommand{\NSZX}{\tn{NS}_{Z}(X)}
\newcommand{\NSsXk}{\tn{NS}_{X/k}^{*}}

\newcommand{\DM}{\tn{\tbf{DM}}} 
\newcommand{\DMgm}{\DM_{\tn{gm}}}
\newcommand{\DT}{\tn{\tbf{DT}}}
\newcommand{\Sm}{\tn{\tbf{Sm}}}
\newcommand{\Smc}{\tn{\tbf{SmCor}}}
\newcommand{\Sch}{\tn{\tbf{Sch}}}
\newcommand{\Schsf}{\Sch_{\tn{sf}}}
\newcommand{\SchsfX}{\Schsf/X}
\newcommand{\SchS}{\tn{\tbf{Sch}}/S}
\newcommand{\Var}{\tn{\tbf{Var}}}
\newcommand{\SmPrVar}{\tn{\tbf{SmPrVar}}}
\newcommand{\SmPrVark}{\tn{\tbf{SmPrVar}}/k}

\newcommand{\GrVekt}{\tn{\tbf{GrVekt}}}
\newcommand{\Vark}{\tn{\tbf{Var}}/k}
\newcommand{\Mgm}{M_{\tn{gm}}}
\newcommand{\tMgm}{\wt{\tn{M}}_{\tn{gm}}}
\newcommand{\Mgmc}{\tn{M}_{\tn{gm}}^{\tn{c}}}
\newcommand{\Mgmi}{\tn{M}_{\tn{gm}}^{i}}
\newcommand{\Mgmt}{\widetilde{\tn{M}}_{\tn{gm}}}

\newcommand{\DMek}{\DMe(k)}

\newcommand{\psul}[2]{{#1}_{\rlap{--}{\phantom{#2}}}}
\newcommand{\pl}{\partial}
\newcommand{\bl}{\bullet}
\newcommand{\sbl}{\bl}
\newcommand{\ssbl}{\bl}
\newcommand{\w}{\wedge}
\newcommand{\Dgk}{\tn{\tbf{DM}}_{\tn{gm}}(k)}
\newcommand{\Degk}{\tn{\tbf{DM}}^{\tn{eff}}_{\tn{gm}}(k)}
\newcommand{\DMk}{\tn{\tbf{DM}}(k)}
\newcommand{\DMS}{\tn{\tbf{DM}}(S)}
\newcommand{\DMgk}{\tn{\tbf{DM}}_{\tn{gm}}(k)}
\newcommand{\DMgkQ}{\tn{\tbf{DM}}_{\tn{gm}}(k)_{\Q}}

\newcommand{\DMegk}{\tn{\tbf{DM}}^{\tn{eff}}_{\tn{gm}}(k)}
\newcommand{\DMegL}{\tn{\tbf{DM}}^{\tn{eff}}_{\tn{gm}}(L)}

\newcommand{\DMegkQ}{\tn{\tbf{DM}}^{\tn{eff}}_{\tn{gm}}(k)_{\Q}}
\newcommand{\DMegLQ}{\tn{\tbf{DM}}^{\tn{eff}}_{\tn{gm}}(L)_{\Q}}

\newcommand{\Dgr}{\tn{Deg}}
\newcommand{\DTegk}{\tn{DT}^{\tn{eff}}_{\tn{gm}}(k)}
\newcommand{\DTk}{\tn{DT}(k)}
\newcommand{\DTS}{\tn{DT}(S)}
\newcommand{\CX}{\cat{C}/X}
\newcommand{\Cx}{\ul{C}_{*}}
\newcommand{\RCx}{\tn{\tbf{R}}\ul{C}_{*}}
\newcommand{\Cnx}{\ul{C}_{n}}
\newcommand{\rat}{\tn{\tiny{rat}}}
\newcommand{\xo}{x_{0}}
\newcommand{\yo}{y_{0}}
\newcommand{\xoT}{\xo^{T}}
\newcommand{\yoT}{\yo^{T}}

\newcommand{\zxoo}{z_{\xo}^{1}}
\newcommand{\zxon}{z_{\xo}^{n}}

\newcommand{\simeqcan}{\st{\tn{\tiny{can.}}}{\simeq}}

\newcommand{\AXz}{A_{X}^{0}}
\newcommand{\AXo}{A_{X}^{1}}
\newcommand{\poC}{p_{1}^{C}}
\newcommand{\fpo}{f_{P_{0}}}
\newcommand{\fxo}{f_{x_{0}}}
\newcommand{\jxo}{j_{x_{0}}}
\newcommand{\jxon}{j_{x_{0}}^{n}}
\newcommand{\jxoo}{j_{x_{0}}^{1}}
\newcommand{\jyo}{j_{y_{0}}}
\newcommand{\ixo}{\iota_{\xo}}
\newcommand{\ixon}{\ixo^{n}}
\newcommand{\iyo}{\iota_{\yo}}
\newcommand{\pxo}{p_{x_{0}}}
\newcommand{\axo}{a_{x_{0}}}
\newcommand{\cxo}{c_{x_{0}}}
\newcommand{\vxo}{v_{x_{0}}}
\newcommand{\upxo}{\ul{p}_{x_{0}}}
\newcommand{\uaxo}{\ul{a}_{x_{0}}}
\newcommand{\qX}{q_{X}}
\newcommand{\qG}{q_{G}}
\newcommand{\Go}{G^{\circ}}

\newcommand{\vnM}{\tn{V}_{n}(M)}
\newcommand{\vmM}{\tn{V}_{m}(M)}
\newcommand{\vnN}{\tn{V}_{n}(N)}
\newcommand{\vmN}{\tn{V}_{m}(N)}
\newcommand{\vn}{\tn{V}_{n}}
\newcommand{\vm}{\tn{V}_{m}}
\newcommand{\Zar}{\tn{Zar}}
\newcommand{\Cov}{\tn{Cov}}
\newcommand{\Et}{\tn{\tbf{Et}}}
\newcommand{\EtX}{\Et/X}
\newcommand{\Etk}{\tn{\tbf{Et}}/k}
\newcommand{\Hetz}{H_{\tn{{\'e}t}}^{0}}
\newcommand{\Heto}{H_{\tn{{\'e}t}}^{1}}
\newcommand{\Chi}{\cly{X}}
\newcommand{\toh}{\textonehalf}
\newcommand{\Lc}{L^{\tn{c}}}
\newcommand{\Smk}{\tn{\tbf{Sm}}/k}
\newcommand{\Smck}{\tn{\tbf{SmCor}}/k}
\newcommand{\SmckQ}{\tn{\tbf{SmCor}}_{\Q}/k}
\newcommand{\SmcQ}{\tn{\tbf{SmCor}}_{\Q}}

\newcommand{\SmcS}{\tn{\tbf{SmCor}}/S}

\newcommand{\scd}{[\, \cdot \, ]_{\tn{sc}}}
\newcommand{\scf}[1]{[#1]_{\tn{sc}}}

\newcommand{\Smcko}{\tn{\tbf{SmCor}}\; k}
\newcommand{\Schk}{\tn{\tbf{Sch}}/k}

\newcommand{\SchX}{\Sch/X}

\newcommand{\pio}{\pi_{0}}
\newcommand{\nat}{\natural}

\newcommand{\pioX}{\pi_{0}(X)}
\newcommand{\pioG}{\pi_{0}(G)}

\newcommand{\cZpioX}{\cly{Z}_{\pi_{0}(X)}}
\newcommand{\cZX}{\cly{Z}_{X}}
\newcommand{\cAX}{\cly{A}_{X}}
\newcommand{\AX}{A_{X}}
\newcommand{\pAX}{\tn{p}A_{X}}
\newcommand{\AoX}{A^{\circ}_{X}}
\newcommand{\cAoX}{\cly{A}^{\circ}_{X}}
\newcommand{\SmProj}{\tn{\tbf{SmProj}}}
\newcommand{\CSmProj}{\tn{\tbf{CSmProj}}}
\newcommand{\SmProjk}{\SmProj/k}
\newcommand{\CSmProjk}{\CSmProj/k}
\newcommand{\CSmProjek}{\CSmProj^{\tn{eff}}k}
\newcommand{\Sh}{\tn{\tbf{Sh}}}
\newcommand{\Shv}{\tn{\tbf{Shv}}}
\newcommand{\Shvs}{\tn{\tbf{Shv}}^{\tn{s}}}
\newcommand{\PSh}{\tn{\tbf{PShv}}}
\newcommand{\Ch}{\tn{\tbf{Ch}}}
\newcommand{\Chp}{\tn{\tbf{Ch}}^{+}}
\newcommand{\Chm}{\tn{\tbf{Ch}}^{-}}
\newcommand{\CoCh}{\tn{\tbf{CoCh}}}
\newcommand{\CoChp}{\tn{\tbf{CoCh}}^{+}}
\newcommand{\CoChm}{\tn{\tbf{CoCh}}^{-}}
\newcommand{\SN}{\tn{\tbf{Sh}}_{\tn{Nis}}}
\newcommand{\SE}{\tn{\tbf{Sh}}_{\tn{\'et}}}
\newcommand{\SNk}{\tn{\tbf{Sh}}_{\tn{Nis}}(\Smck)}
\newcommand{\SEk}{\tn{\tbf{Sh}}_{\et}(k,\Z)}
\newcommand{\SEkQ}{\SE(k, \Q)}
\newcommand{\fgt}{\tn{\tbf{fgt}}}

\newcommand{\SEsmk}{\SE(\Smk)}
\newcommand{\SEsmck}{\SE(\Sm/\cj{k})}

\newcommand{\ST}{\tn{\tbf{ShT}}}
\newcommand{\STt}{\ST_{\tau}}
\newcommand{\STE}{\ST_{\et}}
\newcommand{\STN}{\ST_{\nis}}
\newcommand{\STtk}{\STt(k)}
\newcommand{\STEk}{\STE(k)}
\newcommand{\STNk}{\STN(k)}
\newcommand{\STEkR}{\STE(k, R)}
\newcommand{\STNkR}{\STN(k, R)}
\newcommand{\STEkZ}{\STE(k, \Z)}
\newcommand{\STNkZ}{\STN(k, \Z)}
\newcommand{\STEkQ}{\STE(k, \Q)}
\newcommand{\STNkQ}{\STN(k, \Q)}

\newcommand{\STtkR}{\STt(k, R)}
\newcommand{\STtkZ}{\STt(k, \Z)}
\newcommand{\STtkQ}{\STt(k, \Q)}

\newcommand{\PST}{\tn{\tbf{PShT}}}
\newcommand{\PSTk}{\PST(k)}
\newcommand{\PSTkR}{\PST(k, R)}
\newcommand{\PSTkZ}{\PST(k, \Z)}
\newcommand{\PSTkQ}{\PST(k, \Q)}

\newcommand{\SEsk}{\tn{\tbf{Sh}}_{\et}(\Smk)}
\newcommand{\SNkQ}{\tn{\tbf{Sh}}_{\tn{Nis}}(\SmckQ)}
\newcommand{\SNks}{\tn{\tbf{Sh}}_{\tn{Nis}}(k)}
\newcommand{\DmSNk}{\Dm \; \tn{\tbf{Sh}}_{\tn{Nis}}(\Smcko)}
\newcommand{\DmSN}{\Dm \;  \tn{\tbf{Sh}}_{\tn{Nis}}}
\newcommand{\MM}{\cly{MM}}
\newcommand{\MMk}{\cly{MM}(k)}
\newcommand{\Mnum}{\cly{M}^{\tn{eff}}_{\tn{num}}}
\newcommand{\Mnumk}{\cly{M}^{\tn{eff}}_{\tn{num}}(k)}
\newcommand{\Mh}{\cly{M}^{\tn{eff}}_{\tn{hom}}}
\newcommand{\Mhk}{\cly{M}^{\tn{eff}}_{\tn{hom}}(k)}
\newcommand{\Ma}{\cly{M}^{\tn{eff}}_{\tn{alg}}}
\newcommand{\Mak}{\cly{M}^{\tn{eff}}_{\tn{alg}}(k)}
\newcommand{\Mer}{\cly{M}^{\tn{eff}}_{\tn{rat}}}
\newcommand{\Merk}{\cly{M}^{\tn{eff}}_{\tn{rat}}(k)}
\newcommand{\Mek}{\cly{M}^{\tn{eff}}_{\sim}}
\newcommand{\Meek}{\cly{M}^{\tn{eff}}_{\sim}(k)}
\newcommand{\RMp}{\cly{M}_{\sim}}
\newcommand{\RMkp}{\cly{M}_{\sim}(k)}
\newcommand{\Mnump}{\cly{M}_{\tn{num}}}
\newcommand{\Mnumkp}{\cly{M}_{\tn{num}}(k)}
\newcommand{\HI}{\tn{HI}}
\newcommand{\Hb}{\cly{H}^{\tn{b}}}
\newcommand{\Hmi}{H^{-i}}
\newcommand{\Hpi}{H^{i}}
\newcommand{\Hmz}{H^{0}}
\newcommand{\Hmo}{H^{-1}}

\newcommand{\Heti}{\Hpi_{\et}}
\newcommand{\Hnisi}{\Hpi_{\nis}}
\newcommand{\Hmoti}{\Hpi_{\mot}}

\newcommand{\HIet}{\tn{HI}_{\tn{\'et}}}
\newcommand{\HIset}{\tn{HI}_{\tn{\'et}}^{\tn{s}}}
\newcommand{\cH}{\tn{H}}
\newcommand{\cxH}{\widecheck{H}}
\newcommand{\cxHo}{\widecheck{H}^{1}}
\newcommand{\HIk}{\tn{HI}(k)}
\newcommand{\HIkR}{\tn{HI}(k, R)}
\newcommand{\EA}{\shf{E}_{\A^{1}}}
\newcommand{\Ztr}{\Z_{\tn{tr}}}
\newcommand{\Rtr}{R_{\tn{tr}}}
\newcommand{\ZtrX}{\Z_{\tn{tr}}(X)}
\newcommand{\Cbl}{C_{\bullet}}
\newcommand{\Cobl}{C^{\bullet}}
\newcommand{\CblF}{C_{\bullet}\shf{F}}
\newcommand{\Hot}{\tn{\tbf{Hot}}}
\newcommand{\Hotb}{\tn{Hot}^{\tn{b}}}
\newcommand{\Hotp}{\tn{Hot}^{+}}
\newcommand{\Hotm}{\tn{Hot}^{-}}
\newcommand{\Tot}{\tn{Tot}}
\newcommand{\TotQ}{\tn{Tot}^{\Q}}
\newcommand{\Tots}{\tn{Tot} \;}
\newcommand{\CH}{\tn{CH}}
\newcommand{\CHx}{\tn{CH}^{*}}

\newcommand{\CHMo}{\tn{CH}_{\Mo}}

\newcommand{\uCH}{\ul{\tn{CH}}}
\newcommand{\Rat}{\tn{Rat}}
\newcommand{\kd}{\tn{kd}}
\newcommand{\Ratr}{\tn{Rat}_{r}}
\newcommand{\Cor}{\tn{Cor}}
\newcommand{\Zc}{\tn{Z}}
\newcommand{\Zfin}{\tn{Z}_{\tn{fin}}}
\newcommand{\Cors}{\tn{Cor}_{\sim}}
\newcommand{\Corsk}{\tn{Cor}_{\sim}(k)}
\newcommand{\Cork}{\tn{Cor}_{k}}
\newcommand{\CCork}{\tn{Cor}_{\tn{Ch}/k}}
\newcommand{\CCorkZ}{\tn{Cor}_{\tn{Ch}/k, \Z}}

\newcommand{\VCork}{\tn{Cor}_{\tn{Vo}/k}}
\newcommand{\VCorke}{\tn{Cor}_{\tn{Vo}/k}^{\eff}}
\newcommand{\CCor}{\tn{Cor}_{\tn{Ch}}}
\newcommand{\VCor}{\tn{Cor}_{\tn{Vo}}}
\newcommand{\Sgn}{\Sigma_{n}}
\newcommand{\SmCor}{\tn{\tbf{SmCor}}}
\newcommand{\SchCor}{\tn{\tbf{SchCor}}}
\newcommand{\SchCork}{\tn{\tbf{SchCor}}_{k}}
\newcommand{\SmCork}{\tn{\tbf{SmCor}}_{k}}

\newcommand{\SH}{\tn{\tbf{SH}}}
\newcommand{\SHX}{\tn{\tbf{SH}}(X)}
\newcommand{\SHA}{\tn{\tbf{SH}}(A)}

\newcommand{\pr}{\tn{pr}}
\newcommand{\cone}{\tn{cone}}
\newcommand{\GS}[1]{\Gamma(X, #1)}
\newcommand{\GSX}{\Gamma(X, \pul{U})}
\newcommand{\GammaX}{\Gamma^{X}}
\newcommand{\muonn}{\nu^{\ge n}}
\newcommand{\mlnn}{\nu_{< n}}
\newcommand{\mlenn}{\nu_{\le n}}
\newcommand{\munn}{\nu_{n}}
\newcommand{\cn}{c_{n}}
\newcommand{\ox}{\otimes}
\newcommand{\oxp}{\ox_{\tn{pre}}}
\newcommand{\oxpn}{\oxp^{n}}
\newcommand{\oxtr}{\ox_{\tn{tr}}}
\newcommand{\oxtrD}{\ox_{\tn{tr}}^{\tn{\tiny{doub}}}}
\newcommand{\oxL}{\ox^{\Lf}}
\newcommand{\oxtrLet}{\oxtr^{\Lf, \et}}
\newcommand{\bx}{\bigotimes}
\newcommand{\os}{\oplus}
\newcommand{\od}{\odot}
\newcommand{\bos}{\bigoplus}
\newcommand{\oxo}{\ox_{1}}
\newcommand{\sx}{\boxtimes}
\newcommand{\ost}{\circledast}
\newcommand{\Part}{\tn{Part}}
\newcommand{\supp}{\tn{supp}}
\newcommand{\Tp}{\tn{T}_{p}}
\newcommand{\TpM}{\tn{T}_{p}M}
\newcommand{\TpN}{\tn{T}_{p}N}
\newcommand{\TpV}{\tn{T}_{p}V}
\newcommand{\TLpW}{\tn{T}_{L(p)}W}

\newcommand{\TFpM}{\tn{T}_{F(p)}M}
\newcommand{\TFpN}{\tn{T}_{F(p)}N}
\newcommand{\TM}{\tn{T} M}
\newcommand{\TpX}{\tn{T}_{p}X}
\newcommand{\TX}{\tn{T} X}
\newcommand{\TqM}{\tn{T}_{q}M}

\newcommand{\hinis}{\ul{h}_{i}^{\tn{Nis}}}
\newcommand{\honis}{\ul{h}_{1}^{\tn{Nis}}}
\newcommand{\hznis}{\ul{h}_{0}^{\tn{Nis}}}

\newcommand{\hitau}{\ul{h}_{i}^{\tau}}
\newcommand{\hotau}{\ul{h}_{1}^{\tau}}
\newcommand{\hztau}{\ul{h}_{0}^{\tau}}

\newcommand{\hiet}{\ul{h}_{i}^{\et}}
\newcommand{\hoet}{\ul{h}_{1}^{\et}}
\newcommand{\hzet}{\ul{h}_{0}^{\et}}

\newcommand{\Hoet}{H^{1}_{\et}}

\newcommand{\cHoet}{\widecheck{H}^{1}_{\et}}

\newcommand{\hz}{h^{0}}
\newcommand{\ho}{h^{1}}
\newcommand{\hn}{h^{n}}
\newcommand{\htw}{h^{2}}
\newcommand{\htd}{h^{2d}}
\newcommand{\hgo}{h^{\ge 1}}
\newcommand{\hi}{h^{i}}

\newcommand{\chinis}{\ul{h}^{i}_{\nis}}
\newcommand{\chonis}{\ul{h}^{1}_{\nis}}
\newcommand{\chznis}{\ul{h}^{0}_{\nis}}
\newcommand{\chmnis}{\ul{h}^{-1}_{\nis}}
\newcommand{\chnis}{\ul{h}_{\nis}}

\newcommand{\cs}[1]{{#1}^{\tn{cs}}}

\newcommand{\Slm}{S_{\lambda}}
\newcommand{\Tk}{\tn{T}^{k}}
\newcommand{\Ak}{\Lambda^{k}}
\newcommand{\Sk}{\Sigma^{k}}
\newcommand{\TkM}{\tn{T}^{k} M}
\newcommand{\TkN}{\tn{T}^{k} N}
\newcommand{\AkM}{\Lambda^{k} M}
\newcommand{\SkM}{\Sigma^{k} M}
\newcommand{\TkV}{\tn{T}^{k} (V)}
\newcommand{\AkV}{\Lambda^{k} (V)}
\newcommand{\SkV}{\Sigma^{k} (V)}
\newcommand{\TkW}{\tn{T}^{k} (W)}
\newcommand{\AkW}{\Lambda^{k} (W)}
\newcommand{\SkW}{\Sigma^{k} (W)}

\newcommand{\TaukM}{\tau^{k}(M)}
\newcommand{\TaukN}{\tau^{k}(N)}

\newcommand{\Td}{\tn{T}^{\bl}}
\newcommand{\Ad}{\Lambda^{\bl}}
\newcommand{\TdM}{\tn{T}^{\bl} M}
\newcommand{\AdM}{\Lambda^{\bl} M}
\newcommand{\SdM}{\Sigma^{\bl} M}
\newcommand{\TdV}{\tn{T}^{\bl} (V)}
\newcommand{\AdV}{\Lambda^{\bl} (V)}
\newcommand{\SdV}{\Sigma^{\bl} (V)}
\newcommand{\TdW}{\tn{T}^{\bl} (W)}
\newcommand{\AdW}{\Lambda^{\bl} (W)}
\newcommand{\SdW}{\Sigma^{\bl} (W)}

\newcommand{\Tuk}{\tn{T}_{k}}
\newcommand{\Auk}{\Lambda_{k}}
\newcommand{\Suk}{\Sigma_{k}}
\newcommand{\TukM}{\tn{T}_{k} M}
\newcommand{\AukM}{\Lambda_{k} M}
\newcommand{\SukM}{\Sigma_{k} M}
\newcommand{\TukV}{\tn{T}_{k} (V)}
\newcommand{\AukV}{\Lambda_{k} (V)}
\newcommand{\SukV}{\Sigma_{k} (V)}
\newcommand{\TukW}{\tn{T}_{k} (W)}
\newcommand{\AukW}{\Lambda_{k} (W)}
\newcommand{\SukW}{\Sigma_{k} (W)}

\newcommand{\TenkV}{\tn{Ten}^{k}(V)}
\newcommand{\Tenk}{\tn{Ten}^{k}}
\newcommand{\TenukV}{\tn{Ten}_{k}(V)}
\newcommand{\Tenuk}{\tn{Ten}_{k}}

\newcommand{\TenV}{\tn{Ten}(V)}
\newcommand{\Ten}{\tn{Ten}}
\newcommand{\TendV}{\tn{Ten}^{\bl}(V)}
\newcommand{\Tend}{\tn{Ten}^{\bl}}
\newcommand{\TenuV}{\tn{Ten}(V)}
\newcommand{\Tenu}{\tn{Ten}}
\newcommand{\shF}{\shf{F}}
\newcommand{\shV}{\shf{V}}
\newcommand{\shG}{\shf{G}}
\newcommand{\shH}{\shf{H}}
\newcommand{\shT}{\shf{T}}
\newcommand{\shO}{\shf{O}}
\newcommand{\shOX}{\shf{O}_{X}}
\newcommand{\shOn}{\shf{O}(n)}
\newcommand{\shOm}{\shf{O}(m)}
\newcommand{\shOd}{\shf{O}(d)}
\newcommand{\shOdi}{\shf{O}(d_{i})}
\newcommand{\shL}{\shf{L}}
\newcommand{\ZX}{\Z^{X}}
\newcommand{\ZY}{\Z^{Y}}
\newcommand{\ZC}{\Z^{C}}
\newcommand{\DX}{D_{X}}
\newcommand{\ZpioX}{\Z^{\pioX}}
\newcommand{\ZqX}{\Z^{\qX}}
\newcommand{\rhoX}{\rho_{X}}
\newcommand{\etaX}{\eta_{X}}
\newcommand{\thetaX}{\theta_{X}}
\newcommand{\psiX}{\psi_{X}}
\newcommand{\muX}{\mu_{X}}
\newcommand{\nuX}{\nu_{X}}
\newcommand{\muXz}{\mu_{X}^{0}}
\newcommand{\nuXz}{\nu_{X}^{0}}
\newcommand{\muXo}{\mu_{X}^{1}}
\newcommand{\muC}{\mu_{C}}
\newcommand{\muCz}{\mu_{C}^{0}}
\newcommand{\muCo}{\mu_{C}^{1}}
\newcommand{\muY}{\mu_{Y}}
\newcommand{\muYo}{\mu_{Y}^{0}}
\newcommand{\unuX}{\ul{\nu}_{X}}
\newcommand{\gammaX}{\gamma_{X}}
\newcommand{\veX}{\ve_{X}}
\newcommand{\deltaX}{\delta_{X}}
\newcommand{\degX}{\deg_{X}}
\newcommand{\alphaX}{\alpha_{X}}
\newcommand{\aX}{a_{X}}
\newcommand{\betaX}{\beta_{X}}
\newcommand{\iotaX}{\iota_{X}}

\newcommand{\utri}{\bigtriangleup}
\newcommand{\TNm}{\tn{T} N}
\newcommand{\cTp}{\tn{T}_{p}^{\vee}}
\newcommand{\cTpM}{\tn{T}_{p}^{\vee}M}
\newcommand{\cTM}{\tn{T}^{\vee} M}
\newcommand{\cTpN}{\tn{T}_{p}^{\vee}N}
\newcommand{\cTN}{\tn{T}^{\vee} N}
\newcommand{\CiM}{C^{\infty}(M)}
\newcommand{\CiN}{C^{\infty}(N)}
\newcommand{\Ci}{C^{\infty}}
\newcommand{\ev}{\tn{ev}}
\newcommand{\trf}{\st{\sim}{\hra}}

\newcommand{\lrai}{\st{\sim}{\lra}}
\newcommand{\rai}{\st{\sim}{\ra}}
\newcommand{\trc}{\st{\sim}{\twoheadrightarrow}}

\newcommand{\dlV}{V^{\vee}}
\newcommand{\dlW}{W^{\vee}}
\newcommand{\dlf}{f^{\vee}}
\newcommand{\dl}[1]{{#1}^{\vee}}
\newcommand{\dldl}[1]{{#1}^{\vee \vee}}
\newcommand{\dldlW}{\dldl{W}}
\newcommand{\dldlV}{\dldl{V}}
\newcommand{\dldlf}{\dldl{f}}
\newcommand{\dlA}{{A}^{\vee}}
\newcommand{\wt}[1]{\widetilde{#1}}

\newcommand{\muon}[1]{\nu^{\ge {#1}}}
\newcommand{\mln}[1]{\nu_{< {#1}}}
\newcommand{\mlen}[1]{\nu_{\le {#1}}}
\newcommand{\mun}[1]{\nu_{{#1}}}
\newcommand{\dpqr}{d^{p,q}_{r}}
\newcommand{\dpqro}{d^{p+r,q-r+1}_{r}}
\newcommand{\dpqrm}{d^{p-r,q+r-1}_{r}}
\newcommand{\Epq}{E^{p,q}}
\newcommand{\Epqr}{E^{p,q}_{r}}
\newcommand{\Epqro}{E^{p+r,q-r+1}_{r}}
\newcommand{\Zpqr}{Z^{p,q}_{r}}
\newcommand{\Bpqr}{B^{p,q}_{r}}
\newcommand{\apqr}{\alpha^{p,q}_{r}}
\newcommand{\bpqr}{\beta^{p,q}_{r}}
\newcommand{\Epqi}{E^{p,q}_{\infty}}

\newcommand{\Hr}{H^{\textnormal{r}}}
\newcommand{\Hl}{H^{\textnormal{l}}}
\newcommand{\Autr}{\textnormal{Aut}_{\textnormal{r}}}
\newcommand{\Autl}{\textnormal{Aut}_{\textnormal{l}}}
\newcommand{\CMt}{K/K_{0}/\Q}

\newcommand{\phn}[1]{\phantom{#1}}
\newcommand{\pul}[1]{\ul{\phantom{#1}}}
\newcommand{\pcl}[1]{\sout{\phantom{#1}}}
\newcommand{\sds}{\, \cdot \,}
\newcommand{\tc}[2]{\textcolor{#1}{#2}}
\newcommand{\pol}[1]{\ol{\phantom{#1}}}
\newcommand{\st}[2]{\stackrel{#1}{#2}}

\makeatletter
\def\equalsfill{$\m@th\mathord=\mkern-7mu
\cleaders\hbox{$\!\mathord=\!$}\hfill
\mkern-7mu\mathord=$}
\makeatother

\newcommand{\lgeqt}[1]{\st{\tn{#1}}{\hbox{\equalsfill}}}
\newcommand{\lgeq}[1]{\st{#1}{\hbox{\equalsfill}}}
\newcommand{\defeq}{\lgeq{\tn{\tiny{def}}}}

\newcommand{\xra}{\xrightarrow}
\newcommand{\xla}{\xleftarrow}
\newcommand{\sxra}[1]{\xra{#1}}
\newcommand{\sxla}[1]{\xla{#1}}

\newcommand{\sra}[1]{\stackrel{#1}{\ra}}
\newcommand{\sla}[1]{\stackrel{#1}{\la}}
\newcommand{\slra}[1]{\stackrel{#1}{\lra}}
\newcommand{\sllra}[1]{\stackrel{#1}{\llra}}
\newcommand{\slla}[1]{\stackrel{#1}{\lla}}

\newcommand{\ira}{\stackrel{\simeq}{\ra}}
\newcommand{\ila}{\stackrel{\simeq}{\la}}
\newcommand{\ilra}{\stackrel{\simeq}{\lra}}
\newcommand{\illa}{\stackrel{\simeq}{\lla}}

\newcommand{\wh}[1]{\widehat{#1}}
\newcommand{\fn}[1]{\footnote{#1}}
\newcommand{\mbf}[1]{\mathbf{#1}}
\newcommand{\modulo}[1]{\; \left( \textnormal{mod} \; {#1} \right)}
\newcommand{\ArtSym}[3]{\left( \frac{{#1} / {#2}}{{#3}} \right)}
\newcommand{\SqrArt}[3]{\left[ \frac{{#1} / {#2}}{{#3}} \right]}
\newcommand{\ArtMap}[2]{\left( \frac{{#1} / {#2}}{\cdot} \right)}
\newcommand{\recmap}[2]{[\, \cdot \,, {#1}/{#2}]}
\newcommand{\rec}[3]{[{#1}, {#2}/{#3}]}
\newcommand{\conj}[1]{\overline{#1}}
\newcommand{\cj}[1]{\overline{#1}}
\newcommand{\leg}[2]{\left( \frac{#1}{#2} \right)}
\newcommand{\sep}[1]{{#1}^{\textnormal{sep}}}
\newcommand{\alg}{\tn{alg}}
\newcommand{\ab}[1]{{#1}^{\textnormal{ab}}}
\newcommand{\abs}[1]{\left|{#1} \right|}
\newcommand{\er}[2]{\textnormal{End}_{#1}(#2)}
\newcommand{\ea}[2]{\textnormal{End}^{0}_{#1}(#2)}
\newcommand{\plh}{( \cdot )}
\newcommand{\absmap}{| \cdot |}
\newcommand{\paar}{\langle \cdot, \cdot \rangle}
\newcommand{\nrmmap}{\rVert \, . \, \rVert}
\newcommand{\nmap}[1]{\rVert {#1} \rVert}
\newcommand{\bnmap}[2]{\langle {#1}, {#2} \rangle}
\newcommand{\sqmap}{[ \, \cdot \, ]}
\newcommand{\smprod}{\; \Pi \;}
\newcommand{\car}{\curvearrowright}
\newcommand{\limm}{\lim\limits}
\newcommand{\pair}{\langle \, , \, \rangle}
\newcommand{\Top}{\tn{\tbf{Top}}}
\newcommand{\TopX}{\tn{\tbf{Top}}(X)}
\newcommand{\nfn}[1]{\rVert #1 \rVert}
\newcommand{\sheaf}[1]{\mathscr{#1}}
\newcommand{\shf}[1]{\mathscr{#1}}
\newcommand{\sh}[1]{\wt{G}}
\newcommand{\sht}[1]{\wt{#1}^{\tr}}
\newcommand{\goth}[1]{\mathfrak{#1}}
\newcommand{\cat}[1]{\textnormal{\textbf{#1}}}
\newcommand{\ideal}[1]{\mathfrak{#1}}
\newcommand{\idl}[1]{\mathfrak{#1}}
\newcommand{\functor}[1]{\mathcal{#1}}
\newcommand{\curly}[1]{\mathcal{#1}}
\newcommand{\cly}[1]{\mathcal{#1}}
\newcommand{\und}[1]{\underline{#1}}
\newcommand{\floor}[1]{\lfloor #1 \rfloor}
\newcommand{\ceil}[1]{\lceil #1 \rceil}
\newcommand{\dirlim}[1]{\lim_{\stackrel{\longrightarrow}{#1}}}
\newcommand{\homol}[3]{H_{#1}(#2, #3)}
\newcommand{\cohom}[3]{H^{#1}(#2, #3)}
\newcommand{\cohomzero}[2]{H^{0}(#1, #2)}
\newcommand{\cohomone}[2]{H^{1}(#1, #2)}
\newcommand{\labeleq}[1]{\stackrel{\textnormal{\tiny{#1}}}{=}}
\newcommand{\diff}[2]{\frac{d{#1}}{d{#2}}}
\newcommand{\mdiff}[3]{\frac{d^{#3}{#1}}{d{#2^{#3}}}}
\newcommand{\pderiv}[2]{\frac{\partial{#1}}{\partial{#2}}}
\newcommand{\pd}{\partial}
\newcommand{\nf}[2]{\nicefrac{#1}{#2}}
\newcommand{\nfh}{\nicefrac{1}{2}}
\newcommand{\nfq}{\nicefrac{1}{4}}
\newcommand{\sth}{{\small\textonehalf}}
\newcommand{\stq}{{\small\textonequarter}}

\newcommand{\pdx}{\frac{\pd}{\pd x}}
\newcommand{\pdxi}{\frac{\pd}{\pd x_{i}}}
\newcommand{\pdxj}{\frac{\pd}{\pd x_{j}}}
\newcommand{\pdxn}[1]{\frac{\pd}{\pd x_{#1}}}

\newcommand{\ord}[2]{\textnormal{ord}_{#2}(#1)}
\newcommand{\mpderiv}[3]{\frac{\partial^{#3}{#1}}{\partial{#2^{#3}}}}
\newcommand{\mpderivb}[3]{\frac{\partial^{#3}{#1}}{\partial{#2}}}
\newcommand{\dotp}[2]{\langle #1, #2 \rangle}
\newcommand{\dotpa}[2]{\langle #1, #2 \rangle_{\tn{a}}}
\newcommand{\dotpm}[2]{\langle #1, #2 \rangle_{\tn{m}}}
\newcommand{\euls}[2]{(#1,  #2)_{\tn{a}}}
\newcommand{\eulsym}{(\pul{x},  \pul{x})}
\newcommand{\eulsyma}{(\pul{x},  \pul{x})_{\tn{a}}}
\newcommand{\dotpsym}{\langle \, \cdot \, \rangle}
\newcommand{\norm}[1]{\Vert #1 \Vert}
\newcommand{\bfm}{(\, \cdot \, , \, \cdot \,)}
\newcommand{\bfmr}[1]{(\, \cdot \, , #1)}
\newcommand{\bfml}[1]{(#1, \, \cdot \,)}
\newcommand{\bfma}{\langle \, \cdot \, , \, \cdot \, \rangle}
\newcommand{\normm}[1]{\rVert #1 \rVert}
\newcommand{\refl}[1]{{#1}^{\textnormal{r}}}
\newcommand{\unit}[1]{{#1}^{\times}}
\newcommand{\tcr}[1]{\textcolor{red}{#1}}
\newcommand{\tcg}[1]{\textcolor{grey}{#1}}
\newcommand{\tbr}[1]{\textcolor{red}{\tbf{#1}{}}}
\newcommand{\tbb}[1]{\textcolor{blue}{\tbf{#1}{}}}
\newcommand{\tbg}[1]{\textcolor{dgreen}{\tbf{#1}{}}}
\newcommand{\remph}[1]{\emph{\tcr{#1}}}
\newcommand{\tp}[1]{{#1}^{\tn{t}}}
\newcommand{\pt}[1]{{}^{\tn{t}}{#1}}
\newcommand{\fns}[1]{\footnotesize{#1}}
\newcommand{\scs}[1]{\scriptsize{#1}}
\newcommand{\tpt}[1]{{}^{\tn{t}}{#1}}
\newcommand{\gbx}[1]{\tcg{\fbox{\tn{\small{#1}}}}}
\newcommand{\gbl}{\tcg{\fbox{\tn{\small{\phn{blah}}}}}}

\newcommand{\tcb}[1]{\textcolor{blue}{#1}}
\newcommand{\NrmF}[2]{\textnormal{N}^{#1}_{#2}}
\newcommand{\prm}{\prime}
\newcommand{\p}{\prime}
\newcommand{\bs}{\backslash}
\newcommand{\ve}{\varepsilon}
\newcommand{\vth}{\vartheta}
\newcommand{\vsig}{\varsigma}
\newcommand{\tn}[1]{\textnormal{#1}}
\newcommand{\tbf}[1]{\textbf{#1}}
\newcommand{\ul}[1]{\underline{#1}}
\newcommand{\ol}[1]{\overline{#1}}
\newcommand{\ub}[1]{\underbrace{#1}}
\newcommand{\ubu}[2]{\underbrace{#1}_{#2}}
\newcommand{\ob}[1]{\overbrace{#1}}
\newcommand{\obo}[2]{\overbrace{#1}^{#2}}
\newcommand{\csf}[1]{\stackrel{\smallfrown}{#1}}
\newcommand{\tcl}[1]{\textcolor{#1}}
\newcommand{\ora}[1]{\overrightarrow{#1}}

\newcommand{\explicitbijection}[4] 
{\begin{array}{ccccccc} 
\rnode{left}{#1} & & & & & & \rnode{right}{#2}  \\\\\\
\end{array} 
\psset{nodesep=3pt, offsetA=3pt, offsetB=3pt} 
\everypsbox{\scriptstyle} 
\ncline{->}{left}{right}\Aput{#3} 
\ncline{->}{right}{left}\Aput{#4}}

\newcommand{\tricommdiag}[6]
{{\xymatrix{
{#1} \ar@{->}[d]_{#6} \ar@{->}[r]^{#4} & {#2}\\
{#3} \ar@{->}[ur]_{#5} &}}}

\newcommand{\hightricommdiag}[6]
{{\xymatrix{
{#1} \ar@{->}[dd]_{#6} \ar@{->}[rr]^{#4} & & {#2}\\
& & \\
{#3} \ar@{->}[uurr]_{#5} & & }}}

\newcommand{\tricommdiagtwo}[6]
{\begin{array}{ccccc} 
\rnode{left}{#1} & & & & \rnode{right}{#2}  \\\\\\
& & \rnode{bottom}{#3} & &
\end{array} 
\psset{nodesep=3pt} 
\everypsbox{\scriptstyle} 
\ncline{<-}{left}{right}\Aput{#4} 
\ncline{<-}{left}{bottom}\Aput{#5}
\ncline{<-}{bottom}{right}\Bput{#6}}

\newcommand{\tricommdiagthree}[7]
{\begin{array}{ccccc} 
& & & & \rnode{top}{#1}  \\\\\\\\
\rnode{left}{#2} & & & & \rnode{right}{#3}
\end{array} 
\psset{nodesep=3pt} 
\everypsbox{\scriptstyle} 
\ncline{->}{top}{left}\Bput{#4} 
\ncline{->}{right}{left}\Bput{#5}
\ncarc[linestyle=dashed,ncurv=0.7,arcangleB=10,arcangleA=10,nodesep=2pt]{<-}{top}{right} \Aput{{#6}}
\ncarc[linestyle=dashed,ncurv=0.7,arcangleB=10,arcangleA=10,nodesep=2pt]{<-}{right}{top} \Aput{{#7}}}

\newcommand{\amultdiag}[6]
{\begin{array}{ccccccc} 
& & & & & & \rnode{top}{#1^{#2}}  \\\\\\
\rnode{midleft}{#1} & & & & & & \rnode{midright}{#1} \\\\\\
& & & & & & \rnode{bottom}{#3}
\end{array} 
\psset{nodesep=3pt} 
\everypsbox{\scriptstyle} 
\ncline{->}{midleft}{top}\Aput{#4^{#2}} 
\ncline{->}{midleft}{midright}\Aput{\iota(#5)}
\ncline{->}{midleft}{bottom}\Bput{#6}
\ncline[linestyle=dashed]{->}{top}{midright}\Bput{f_{#5}}
\ncline[linestyle=dashed]{->}{bottom}{midright}\Aput{g_{#5}}
\ncarc[linestyle=dashed,ncurv=0.7,arcangleB=30,arcangleA=30,nodesep=2pt]{<-}{top}{bottom}\Aput{h_{#5}}}

\newcommand{\hightricommdiagtwo}[6]
{\begin{array}{ccccc} 
\rnode{left}{#1} & & & & \rnode{right}{#2}  \\\\\\\\
& & \rnode{bottom}{#3} & &
\end{array} 
\psset{nodesep=3pt} 
\everypsbox{\scriptstyle} 
\ncline{<-}{left}{right}\Aput{#4} 
\ncline{<-}{left}{bottom}\Aput{#5}
\ncline{<-}{bottom}{right}\Bput{#6}}

\newcommand{\diamcommdiag}[8]
{\begin{array}{ccccc} 
& & \rnode{t}{#1} & & \\\\\\
\rnode{l}{#4} & & & & \rnode{r}{#2}  \\\\\\
& & \rnode{b}{#3} & &
\end{array} 
\psset{nodesep=3pt} 
\everypsbox{\scriptstyle} 
\ncline{->}{t}{r}\Aput{#5} 
\ncline{->}{r}{b}\Aput{#6}
\ncline{->}{l}{b}\Bput{#7}
\ncline{->}{t}{l}\Bput{#8}}

\newcommand{\fibrediag}[7] 
{\begin{array}{ccccc} 
& & \rnode{t}{#3} & & \\\\\\\\
& & \rnode{m}{#1 \times_{#4} #2} & & \\\\\\
\rnode{l}{#1} & & & & \rnode{r}{#2}  \\\\\\
& & \rnode{b}{#4} & &
\end{array} 
\psset{nodesep=3pt} 
\everypsbox{\scriptstyle} 
\ncline[linestyle=dashed]{->}{t}{m}\Aput{#7} 
\ncline{->}{r}{b}
\ncline{->}{l}{b}
\ncline{->}{m}{r}\Bput{p_{2}}
\ncline{->}{m}{l}\Aput{p_{1}}
\ncarc{<-}{l}{t}\Aput{#5}
\ncarc{->}{t}{r}\Aput{#6}}

\newcommand{\newfibrediag}[8] 
{\begin{array}{ccccc} 
& & \rnode{t}{#3} & & \\\\\\\\
& & \rnode{m}{#4} & & \\\\\\
\rnode{l}{#1} & & & & \rnode{r}{#2}  \\\\\\
& & \rnode{b}{#5} & &
\end{array} 
\psset{nodesep=3pt} 
\everypsbox{\scriptstyle} 
\ncline[linestyle=dashed]{->}{t}{m}\Aput{#8} 
\ncline{->}{r}{b}
\ncline{->}{l}{b}
\ncline{->}{m}{r}\Bput{p_{2}}
\ncline{->}{m}{l}\Aput{p_{1}}
\ncarc{<-}{l}{t}\Aput{#6}
\ncarc{->}{t}{r}\Aput{#7}}

\newcommand{\specfibrediag}[7] 
{\begin{array}{ccccc} 
& & \rnode{t}{#4} & & \\\\\\\\
& & \rnode{m}{\Spec(#1 \otimes_{#3} #2)} & & \\\\\\
\rnode{l}{\Spec(#1)} & & & & \rnode{r}{\Spec(#2)}  \\\\\\
& & \rnode{b}{\Spec(#3)} & &
\end{array} 
\psset{nodesep=3pt} 
\everypsbox{\scriptstyle} 
\ncline[linestyle=dashed]{->}{t}{m}\Aput{#7} 
\ncline{->}{r}{b}
\ncline{->}{l}{b}
\ncline{->}{m}{r}\Bput{p_{2}}
\ncline{->}{m}{l}\Aput{p_{1}}
\ncarc{<-}{l}{t}\Aput{#5}
\ncarc{->}{t}{r}\Aput{#6}}

\newcommand{\contrafibrediag}[7] 
{\begin{array}{ccccc} 
& & \rnode{t}{\Gamma(#4, \sheaf{O}_{#4})} & & \\\\\\\\
& & \rnode{m}{#1 \otimes_{#3} #2} & & \\\\\\
\rnode{l}{#1} & & & & \rnode{r}{#2}  \\\\\\
& & \rnode{b}{#3} & &
\end{array} 
\psset{nodesep=3pt} 
\everypsbox{\scriptstyle} 
\ncline[linestyle=dashed]{<-}{t}{m}\Aput{#7} 
\ncline{<-}{r}{b}
\ncline{<-}{l}{b}
\ncline{<-}{m}{r}\Bput{p_{2}^{*}}
\ncline{<-}{m}{l}\Aput{p_{1}^{*}}
\ncarc{->}{l}{t}\Aput{#5^{*}}
\ncarc{<-}{t}{r}\Aput{#6^{*}}}

\newcommand{\revtricommdiag}[6]
{\begin{array}{ccccc} 
\rnode{i}{#1} & & & & \rnode{j}{#2}  \\\\\\
& & \rnode{g}{#3} & &
\end{array} 
\psset{nodesep=3pt} 
\everypsbox{\scriptstyle} 
\ncline{<-}{i}{j}\Aput{#4} 
\ncline{<-}{j}{g}\Aput{#5}
\ncline{<-}{i}{g}\Bput{#6}}

\newcommand{\catequivcommdiag}[6] 
{\begin{array}{cccccc} 
& & & \rnode{tu}{#1(#3)} & & \\\\\\
\rnode{tl}{#3} & & & & & \rnode{tr}{(#2 \circ #1)(#3)} \\\\\\\\
& & & \rnode{bu}{#1(#4)} & & \\\\\\
\rnode{bl}{#4} & & & & & \rnode{br}{(#2 \circ #1)(#4)}  
\end{array} 
\psset{nodesep=3pt} 
\everypsbox{\scriptstyle} 
\ncline{->}{tl}{tu}\Aput{#1} 
\ncline{->}{tu}{tr}\Aput{#2}
\ncline{->}{tl}{tr}\Aput{#6(#3)}
\ncline{->}{bl}{bu}\Aput{#1} 
\ncline{->}{bu}{br}\Aput{#2}
\ncline{->}{bl}{br}\Bput{#6(#4)}
\ncline{->}{tl}{bl}\Bput{#5} 
\ncline[linestyle=dashed]{->}{tu}{bu}\Aput{#1(#5)}
\ncline{->}{tr}{br}\Aput{(#2 \circ #1)(#5)}}

\newcommand{\rectcommdiag}[8]
{\xymatrix{
{#1} \ar@{->}[d]_{#8} \ar@{->}[r]^{#5} & {#2} \ar@{->}[d]^{#6} \\
{#3} \ar@{->}[r]_{#7} & {#4}}}

\newcommand{\highrectcommdiag}[8]
{\xymatrix{
{#1} \ar@{->}[dd]_{#8} \ar@{->}[rr]^{#5} & & {#2} \ar@{->}[dd]^{#6} \\
& & \\
{#3} \ar@{->}[rr]_{#7} & & {#4}}}

\newcommand{\sheafrectcommdiag}[5] 
{\begin{array}{ccccc} 
\rnode{tl}{#1(#3)} & & & & \rnode{tr}{#2(#3)}  \\\\\\
\rnode{bl}{#1(#4)} & & & & \rnode{br}{#2(#4)}
\end{array} 
\psset{nodesep=3pt} 
\everypsbox{\scriptstyle} 
\ncline{->}{tl}{tr}\Aput{#5(#3)} 
\ncline{->}{tr}{br}\Aput{\rho^{#2}_{#3 #4}}
\ncline{->}{bl}{br}\Aput{#5(#4)}
\ncline{->}{tl}{bl}\Aput{\rho^{#1}_{#3 #4}}}

\newcommand{\rectcommdashdiag}[9]
{\begin{array}{ccccc} 
\rnode{tl}{#1} & & & & \rnode{tr}{#2}  \\\\\\
\rnode{bl}{#4} & & & & \rnode{br}{#3}
\end{array} 
\psset{nodesep=3pt} 
\everypsbox{\scriptstyle} 
\ncline{->}{tl}{tr}\Aput{#5} 
\ncline{->}{tr}{br}\Aput{#6}
\ncline{->}{bl}{br}\Aput{#7}
\ncline{->}{tl}{bl}\Bput{#8}
\ncline[linestyle=dashed]{->}{bl}{tr}\Aput{#9}}

\newcommand{\contrahighrectcommdiag}[8]
{\begin{array}{ccccc} 
\rnode{tl}{#1} & & & & \rnode{tr}{#2}  \\\\\\\\
\rnode{bl}{#4} & & & & \rnode{br}{#3}
\end{array} 
\psset{nodesep=3pt} 
\everypsbox{\scriptstyle} 
\ncline{->}{tl}{tr}\Aput{#5} 
\ncline{<-}{tr}{br}\Aput{#6}
\ncline{->}{bl}{br}\Aput{#7}
\ncline{<-}{tl}{bl}\Aput{#8}}

\newcommand{\sheafhighrectcommdiag}[5] 
{\begin{array}{ccccc} 
\rnode{tl}{#1(#3)} & & & & \rnode{tr}{#2(#3)}  \\\\\\\\
\rnode{bl}{#1(#4)} & & & & \rnode{br}{#2(#4)}
\end{array} 
\psset{nodesep=3pt} 
\everypsbox{\scriptstyle} 
\ncline{->}{tl}{tr}\Aput{#5(#3)} 
\ncline{->}{tr}{br}\Aput{\rho^{#2}_{#3 #4}}
\ncline{->}{bl}{br}\Aput{#5(#4)}
\ncline{->}{tl}{bl}\Aput{\rho^{#1}_{#3 #4}}}

\newcommand{\dirlimdiag}[6] 
{\begin{array}{ccccccc} 
\rnode{tl}{#1_{#4}} & & & & & & \rnode{tr}{#1_{#5}}  \\\\\\\\
& & & \rnode{m}{#1} & & & \\\\\\\\
& & & \rnode{b}{#2} & & &
\end{array} 
\psset{nodesep=3pt} 
\everypsbox{\scriptstyle} 
\ncline{->}{tl}{tr}\Aput{#3_{#4 #5}} 
\ncline{->}{tr}{m}\Bput{#3^{#1}_{#5}}
\ncline{->}{tl}{m}\Aput{#3^{#1}_{#4}}
\ncarc{<-}{b}{tl}\Aput{#3^{#2}_{#4}}
\ncline{->}{m}{b}\Aput{#6}
\ncarc{->}{tr}{b}\Aput{#3^{#2}_{#5}}}

\newcommand{\dirlimdiagdash}[6] 
{\begin{array}{ccccccc} 
\rnode{tl}{#1_{#4}} & & & & & & \rnode{tr}{#1_{#5}}  \\\\\\\\
& & & \rnode{m}{#1} & & & \\\\\\\\
& & & \rnode{b}{#2} & & &
\end{array} 
\psset{nodesep=3pt} 
\everypsbox{\scriptstyle} 
\ncline{->}{tl}{tr}\Aput{#3_{#4 #5}} 
\ncline{->}{tr}{m}\Bput{#3^{#1}_{#5}}
\ncline{->}{tl}{m}\Aput{#3^{#1}_{#4}}
\ncarc{<-}{b}{tl}\Aput{#3^{#2}_{#4}}
\ncline[linestyle=dashed]{->}{m}{b}\Aput{#6}
\ncarc{->}{tr}{b}\Aput{#3^{#2}_{#5}}}

\newcommand{\twomatrix}[4]
{\left(  \begin{array}{cc} 
\rnode{tl}{#1} & \rnode{tr}{#2}  \\
\rnode{bl}{#3} & \rnode{br}{#4}
\end{array} \right)}

\newcommand{\bigmatrix}[9]
{\left(  \begin{array}{cccc} 
\rnode{tl}{#1} & \rnode{tm}{#2}  & \cdots & \rnode{tr}{#3} \\
\rnode{ml}{#4} & \rnode{mm}{#5}  & \cdots & \rnode{mr}{#6} \\
\vdots & \vdots  & \ddots & \vdots \\
\rnode{bl}{#7} & \rnode{bm}{#8}  & \cdots & \rnode{br}{#9}
\end{array} \right)}

\newcommand{\bigmatrixb}[9]
{\left(  \begin{array}{cccc} 
\rnode{tl}{#1} & \rnode{tm}{#2}  & \cdots & \rnode{tr}{#3} \\\\
\rnode{ml}{#4} & \rnode{mm}{#5}  & \cdots & \rnode{mr}{#6} \\\\
\vdots & \vdots  & \ddots & \vdots \\\\
\rnode{bl}{#7} & \rnode{bm}{#8}  & \cdots & \rnode{br}{#9}
\end{array} \right)}

\newcommand{\mathsmap}[5] 
{\begin{array}{cccc}
#1 \, : & #2 & \longrightarrow & #3\\
& #4 & \longmapsto & #5
\end{array}}

\newcommand{\mathsmapbi}[7] 
{\begin{array}{cccc}
#1 \, : & #2 & \longrightarrow & #3 \\
& #4 & \longmapsto & #5 \\
& #6 & \longmapsfrom & #7 \\
\end{array}}

\newcommand{\mathsmapxy}[5] 
{\xymatrix{
{#1}: & {#2} \ar@{->}[r] & {#3} \\
& {#4} \ar@{|->}[r] & {#5}}}

\newcommand{\matrixtwo}[4] 
{\left( \begin{array}{ccc}
{#1} & {#2} \\
{#3} & {#4} 
\end{array}\right)}

\newcommand{\mathsmaptwo}[4] 
{\begin{array}{ccc}
#1 & \longrightarrow & #2\\
#3 & \longmapsto & #4
\end{array}}

\newcommand{\mathsmaptwox}[4] 
{\begin{eqnarray*}
\xymatrix{
#1 \ar@{->}[rr] && #2 \\
#3 \ar@{|->}[rr] && #4}
\end{eqnarray*}}

\newcommand{\mathsmapvert}[5] 
{\xymatrix{
{#2} \ar@{->}[dd]_{{#1}} & \ni & {#4} \ar@{|->}[dd]\\
& & \\
{#3} & \ni & {#5}}}

\newcommand{\doublemathsmap}[7] 
{\begin{array}{cccc}
#1: & #2 & \longrightarrow & #3\\
& #4 & \longmapsto & #5 \\
& #6 & \longmapsto & #7
\end{array}}

\newcommand{\ses}[7]
{\xymatrix{#1 \ar[r] & #2 \ar[r]^{#6} & #3 \ar[r]^{#7} & #4 \ar[r] & #5}}

\newcommand{\sestwo}[5]
{\xymatrix{#1 \ar[r] & #2 \ar[r] & #3 \ar[r] & #4 \ar[r] & #5}}

\newcommand{\sesthree}[5]
{\xymatrix{0 \ar[r] & #1 \ar[r]^{#4} & #2 \ar[r]^{#5} & #3 \ar[r] & 0}}

\newcommand{\splitses}[8]
{\xymatrix{#1 \ar[r] & #2 \ar[r]^{#6} & #3 \ar@/^/[r]^{#7} & #4 \ar@/^/[l]^{#8} \ar[r] & #5}}

\setcounter{tocdepth}{1}

\title{On the motive of a commutative algebraic group}

\author{Giuseppe Ancona}
\address{Universit\"at Duisburg-Essen, Thea-Leymann-Stra\ss e 9 ,
45127 Essen, Germany}

\email{monsieur.beppe@gmail.com}
\author{Stephen Enright-Ward}
\address{Mathematisches Institut, Albert-Ludwigs-Universit\"at Freiburg, Eckerstra\ss e~1,  
79104 Freiburg im Breisgau, Germany}
\email{steveleward@gmail.com}
\author{Annette Huber}
\address{Mathematisches Institut, Albert-Ludwigs-Universit\"at Freiburg, Eckerstra\ss e~1,  
79104 Freiburg im Breisgau, Germany}
\email{annette.huber@math.uni-freiburg.de}

\begin{abstract}
We prove a canonical K\"unneth decomposition for the motive of a commutative group scheme over a field. Moreover, we show that this decomposition behaves under the group law just as in cohomology. 
\end{abstract}


%


\maketitle
\begin{center}
\today
\end{center}
\tableofcontents
\section*{Introduction}
The aim of this article is to prove a canonical K\"unneth decomposition for the motive of a commutative group scheme over a field. Moreover, we show that this decomposition behaves under the group law just as in cohomology.

\

Let us start with the concrete example of a connected commutative algebraic group $G$ over $\C$, the field of complex numbers. Then the complex manifold associated with $G$ is homotopy equivalent to a product of unit circles $S^1$. So, by the K\"unneth formula for singular cohomology, one has 
\[H^i_\mathrm{sing}(G(\C),\Q)= \bigwedge^i H^1_\mathrm{sing}(G(\C),\Q).\]
This phenomenon exists also for $\ell$-adic cohomology. Let $G$ be a connected commutative group scheme over a field $k$ and let $\ell$ a prime number different from the characteristic of $k$. Then 
$H^*(G_{\bar{k}},\Q_{\ell})$ 
is a graded finite-dimensional
$\Q_{\ell}$-vector space. Moreover, it has a canonical structure of graded commutative and cocommutative connected Hopf algebra coming from the group structure and the diagonal immersion $\Delta:G \to G \times G$ (the latter induces the cup product).
The classification of graded Hopf algebras implies that $H^*(G_{\bar{k}},\Q_{\ell})$ has to be the exterior algebra of its \textsl{primitive part}. It turns out
to be $H^1(G_{\bar{k}},\Q_{\ell})$, the dual of the Tate module tensor $\Q_{\ell}$. 
So one deduces an isomorphism of graded Hopf algebras
\[H^*(G_{\bar{k}},\Q_{\ell}) = \bigwedge^* H^1(G_{\bar{k}},\Q_{\ell}) \ .\]

Our main theorem shows that this formula is motivic, up to two minor subtleties. First, in odd degree symmetric powers realize to the exterior powers (Koszul rule of signs), so one has to consider the former. Second, the realization functors from motives to cohomology is contravariant so one has to consider the opposite Hopf algebra structure.

\begin{introtheorem}\label{IntroKunn}
Let $G$ be a connected commutative algebraic group over a field $k$. 
Then there exists a canonical decomposition in $\DMegkQ$
\[M(G)=\bigoplus_{i=0}\Mi{i}(G),\]
 such that, for any mixed Weil cohomology theory $H^*$, the motive $\Mi{i}(G)$ realizes to $H^i(G)$. Moreover,
\begin{enumerate}
\item for $i$ big enough the motives $ \Sym ^i( \Mi{1}(G))$  and $\Mi{i}(G)$ vanish,
\item there is a canonical isomorphism 
\[\Mi{i}(G) \cong \Sym^i\Mof(G)\ ,\]
\item
The canonical isomorphism 
 \[M(G) \cong \OSym  (\Mi{1}(G))= \bigoplus_i\Sym^i\Mof(G)\]
 is an isomorphism of Hopf algebras between the motive of $G$ and the \textit{symmetric coalgebra} over the motive $\Mof(G).$
 \end{enumerate}

\end{introtheorem}
For more 
refined statements see Theorems \ref{Thmconsequences} and
\ref{propweil} 

There are a number of useful applications of our motivic decomposition. For instance, we are able to describe the weight filtration on $\Mi{i}(G)$ explicitly (see Section~\ref{sec_weights}). We get a new proof of Kimura finiteness of $1$-motives (which implies Kimura finiteness of the motive of a curve). We also show that  $H^*(M)$ is concentrated in one degree for any $1$-motive $M$ and any Weil cohomology theory.

Recently, Sugiyama \cite{Sug} has applied Theorem \ref{IntroKunn} to generalize  to semiabelian varieties results of Beauville \cite{Beau} and Bloch 
\cite[Theorem 0.1]{Bl1} on Chow groups of abelian varieties.

In a sequel of this paper \cite{relative}, we will show that all the results in this introduction hold for the relative motive of a general smooth commutative group scheme $G$ over a base scheme $S$ (supposed noetherian and finite dimensional). Note that Theorem \ref{IntroKunn} was known for abelian schemes over a regular base \cite{DeMu,Ku1}. Our generalization will be interesting especially for N\'eron models of abelian varieties and for the group scheme over  compactifications of Shimura varieties.


\

Theorem \ref{IntroKunn} was already known for tori and abelian varieties. The straightforward case of tori was discussed partially by Huber-Kahn \cite{HK}. The abelian case was proven by K\"unnemann \cite{Ku1} following earlier partial results by Shermenev \cite{Sher} and Deninger-Murre \cite{DeMu}. The main tool in his proof is the Fourier transform for cycles over an abelian variety, introduced by Beauville \cite{Beau}. This is not available for more general groups. 

Instead, our starting point is to use the more flexible category 
\textit{motivic complexes}. 
Indeed, we write down the component $\Mof(G)$ completely explicitly. This is non-trivial, even for abelian varieties.


First, following Barbieri-Viale and Kahn \cite{BVK}, consider $\ul{G}$ the sheaf  on the big \'etale site on $\Smk$ represented by $G$. It inherits from the group structure of $G$ a canonical structure of sheaf of abelian groups. 
 By work of Spie\ss-Szamuely (\cite{SS}) $\ul{G}$ has transfers, hence it induces a motive (namely the singular complex $\ul{C}_*\ul{G}$ that we will note $\Mof(G).$  

Second, consider the map of (pre)sheaves with transfers
\[ \Cor(\cdot,G)\to \ul{G}\]
which maps a multivalued map $S\to G$ (a correspondence) to the sum of its values in the commutative group $G$. It induces a canonical map
\[ \alpha_G:M(G)\to \Mof(G)\ .\]
Using the comultiplication on $M(G)$ this extends to a natural map
\[ \vp_G^n: M(G)\to \Sym^n\Mof(G)\ .\]
Most of the effort of this paper goes into proving the following (Theorem \ref{MainThm} in the paper):



\begin{introtheorem}\label{IntroMain}
Let $G$ be a semiabelian variety over a perfect field $k$. Then the motive $\Mof(G)$ is odd (i.e. $\Symn \Mof(G)$ vanishes for $n$ big enough) and, moreover, the map 
\[ \vp_G=\bigoplus_n \vp_{G}^n: M(G) \lra \bigoplus_{n = 0}\Symn\Mof(G).\]
is an isomorphism. 
\end{introtheorem}
We will easily deduce Theorem \ref{IntroKunn} from Theorem \ref{IntroMain}: the construction of the morphism $\vp_G$ has been done so that it is natural and it respects the Hopf structures (and the unipotent part of a general commutative group will be easy to treat, as well as the non-perfect case).  

A difficulty in showing this theorem is that we do not have for free a natural maps in the opposite direction
(the reader can think of the embedding $\Mor_{\Smk}(\cdot,G) \subset\Cor(\cdot,G)$, but this is just a morphism of sheaves of sets).

We prove Theorem \ref{IntroMain}  by induction on the torus rank with the case of abelian varieties as a starting point.
Consider a short exact sequence of semiabelian varieties 
\[ 1\to\Gm\to G\to H\to 1\]
and suppose to know Theorem~\ref{IntroMain} for $H$. 
On the one hand, the filtration on $\ul{G}$ defined by the above short exact sequence 
induces triangles
\[  \Sym^{n}\Mof(G) \to \Sym^{n}\Mof(H) \to \Sym^{n-1}\Mof(H)(1)[2]\ .\]
On the other hand, the localization sequence for the $\Gm$-torsor $G\to H$ reads
\[ M(G)\to M(H)\to M(H)(1)[2]\ .\]
By comparing these triangles, the hypothesis that $\vp_H$ is an isomorphism 
implies that there is a non-canonical isomorphism
\[\psi:M(G)\to \bigoplus_n\Symn(\Mof(G))\ .\]
This fact has two essential consequences. First, the motive $M(G)$ is \textit{Kimura finite}. Second, there is a non-trivial morphism from $\bigoplus_n\Symn(\Mof(G))$ to $M(G)$. We then show that one can modify this morphism in order to obtain a map which is the inverse of $\vp_G$ after $\ell$-adic realization. Then Kimura finiteness allows us to conclude that $\vp_G$ is an isomorphism.

\ 

We now briefly discuss the structure of the paper. Section \ref{Notations} settles notations and recalls some facts from the literature.
In Section \ref{sect_constr} we define the motive $\Mof(G)$ and the morphism $\alpha_G$ and establish basic properties. We will pretty quickly specialize to the case of a semiabelian variety.
In Section \ref{sect main theorem} we construct the morphism $\vp_G$ and show that it respects the Hopf structures, under the hypothesis that $\Symn\Mof(G)$ vanishes for $n$ big enough.
Section \ref{sect special} deals with the two special cases: tori and abelian varieties.  
In Section \ref{sect real} we show that the motive $\Mof(G)$ of a semiabelian variety $G$ is odd and geometric. We then compute its $\ell$-adic realization. 
Section \ref{proofMainThm} finally gives the proof of Theorem \ref{IntroMain} for any semiabelian variety $G$. 

Section \ref{secconseq} shows how to deduce Theorem \ref{IntroKunn} from Theorem \ref{IntroMain}. 
We also explain how the semiabelian case implies the case of a general commutative scheme, possibly with  a unipotent radical or several connected components. Some other properties are studied, namely the uniqueness of the K\"unneth decomposition  and that the weight filtration in cohomology lifts canonically to a filtration of each K\"unneth component. 

Some technical points of the main proof are left to the appendices. 
Appendix \ref{sectionchern} 
is the essential input to relate the two triangles above. This is done
comparing two obvious definitions of the first Chern class of the line
bundle defined by an extension of a semiabelian variety by $\Gm$.
Appendix~\ref{appsym} rewiews the definition of the symmetric coalgebra, its universal properties and a comparison with the standard  symmetric algebra in the setting of $\Q$-linear pseudo-abelian additive categories. 
Appendix \ref{appfilt} constructs a natural filtration on $\Sym^n(V)$ given
a subobject $U\subset V$ in the setting of general $\Q$-linear abelian categories.

\subsection*{Acknowledgments}
Part of these results have already been obtained in the PhD thesis \cite{ward} of the second author, written under the supervision of the third. The first author would like to thank them for letting him participate to conclude this work and for the warm hospitality in Freiburg. 

In more detail: the definition of $\vp_G$ was given in \cite{ward} and  the results on tori and abelian varieties were established there. It also contained the inductive argument of Section \ref{sec_triangles} but not the proof of the main Theorem \ref{IntroMain}.

We would like to thank J.~Ayoub,  F.~D\'eglise, B.~Drew, J.~Fres\'an, M.~Huruguen, P.~Jossen, B.~Kahn, S.~Kelly,  K.~K\"unnemann, S.~Meagher, S.~Pepin Lehalleur, J.~Scholbach, R.~Sugiyama, and C.~Vial for useful discussions.

R.~Sugiyama pointed out a gap in an earlier version of our preprint. The 
problem is resolved by his \cite[Appendix B]{Sug}.

We are most thankful to M.~Wendt, who has been involved in the project from the start. He came to act as a coadvisor of the thesis and contributed generously by discussions, advice on references and careful proof reading. The (as it turns out decisive) remark that the very formulation of the main theorem needs a finiteness assumption, was his.

\section{Notations and Generalities}\label{Notations}
Throughout this paper $k$ denotes a fixed perfect field of any characteristic. We fix   an algebraic closure of $k$ and we call it $\bar{k}$. We write $\overline{(\cdot)}$ for the pull-back to $\bar{k}$ on varieties, motives and sheaves. 
\subsection{Categories of varieties}
We denote:
\begin{itemize}
\item $\Schk$ the category of schemes separated and of finite type over $k$; 
\item $\Smk$ the full subcategory of smooth schemes over $k$;
\item $\SmPrVark$ the full subcategory of smooth and projective schemes over $k$;
\item $\Smck$ the category of smooth correspondences. Objects are smooth schemes over $k$  and morphisms are finite correspondences in the sense of Voevodsky \cite[Section 2]{TMF}. The free abelian group generated by primitive finite correspondences from $X$ to $Y$ will be denoted $\Cor(X,Y)$; 
\item $\CGS/k$ the category of commutative group schemes of finite type over $k$ with morphisms  of $k$-group schemes; 
\item $\sAb/k$ the full subcategory of $\CGS$ of semiabelian varieties. By definition an object of $\sAb/k$ is a connected smooth commutative group scheme $G$ over $k$ such that its pullback $G_{\bar{k}}$  does not have any subgroups which are isomorphic to an additive group.
\item $\Ab/k$ the full subcategory of $\CGS$ of abelian varieties.
\end{itemize}
We will usually abbreviate $\Sch$ instead of $\Sch/k$ etc.

\begin{remark}
Notice that by a theorem of Barsotti \cite{Bars} and Chevalley \cite{Chev} (see also \cite{Con} for a modern presentation) any semiabelian variety can be uniquely decomposed as  an extension of an abelian variety by a torus. 


Let us also recall that there are no non-constant homomorphism from a torus to an abelian variety, see for example \cite[Lemma 2.3]{Con}. In particular any homomorphism between two semiabelian varieties induces morphisms between the tori and between the abelian varieties of their decomposition.
\end{remark}
\begin{remark}
Notice that a short exact sequence in $\CGS$ of smooth groups induces a short exact sequence of \'etale sheaves, as any smooth morphism has a section locally for the \'etale topology.
\end{remark}
\subsection{Rational coefficients}\label{sect_rational}
When $\cata$ is an additive category, then we will write 
$ \cata_{\Q}$ for the pseudo-abelian hull of 
the category having the same objects as $\cata$ and such that for any $X,Y \in \cata$ the set of the homomorphisms from $X$ to $Y$ is the $\Q$-vector space 
\[ \Hom_{\cata}(X,Y)\otimes_{\Z}\Q\ .\] 
This applies, in particular, to the additive categories $\Smc$, $\Ab$ and $\sAb$.
\begin{remark}
\begin{enumerate}

\item Notice that morphisms  in $\SmcQ$ between two varieties are $\Q$-linear combinations of primitive  finite correspondences.
\item The categories $\AbQ$ and $\sAbQ$ are the categories of abelian and semiabelian varieties up to isogeny, respectively. These categories are abelian.
\end{enumerate}
\end{remark}

\subsection{Symmetric powers, exterior powers and Kimura finiteness}\label{notsym}
Let $\cata$ be a pseudo-abelian $\Q$-linear symmetric tensor category with unit $\one$. There are canonical functors
\[\Sym^n:\cata\to\cata \ .\]
Indeed, let $X \in \cata$ be an object of $\cata$. As the category is symmetric the group of permutations $\mathcal{S}_n$ acts on $X^{\otimes n}$. The endomorphism 
\[ \frac{1}{n!}\sum_{\sigma\in \mathcal{S}_n}\sigma:X^{\tensor n}\to X^{\tensor n}\]
is a projector. Then one defines $\Sym^n(X)$
to be its image (notice that it is fonctorial in $X \in \cata$).  We define similarly the functor $\bigwedge^n :\cata\to\cata$. By convention, we will write $\Sym^{0}X=\bigwedge^0 X= \one$ for all non-zero objects $X$.
Following \cite{Kim} and \cite{OS} we will say that an object $X$ is 
\begin{itemize}
\item {\em odd of dimension $N$} if $\Sym^N X\neq 0$ and $\Sym^{N+1} X= 0$. In this case we will write $\det X= \Sym^N X$;
\item  {\em even  of dimension $N$}  if $\bigwedge^N X\neq 0$ and $\bigwedge^{N+1} X= 0$. In this case we will write $\det X= \bigwedge^N X$;
\item {\em odd or even finite-dimensional} if it is even or odd finite-dimensional for some $N$;
\item {\em Kimura finite} if there exists a decomposition (in general not unique) $X= X_+ \oplus X_-$ such that $X_+$ is even finite-dimensional and $X_-$ is odd finite-dimensional.
\end{itemize}
 
 The result we will need about finite-dimensional motives is the following theorem of Andr\'e and Kahn.
\begin{theorem}[Andr\'e, Kahn]\label{conservativity}
Let $K$ be a field of characteristic zero,  $\mathcal{C}$ be a $K$-linear pseudoabelian symmetric tensor category, and let $f:X \to Y$ be a map in $\mathcal{C}$ between two finite-dimensional objects. Suppose that there exist another $K$-linear pseudo-abelian symmetric tensor category $\mathcal{D}$, a non-zero $K$-linear symmetric tensor functor $F: \mathcal{C} \to  \mathcal{D}$ (covariant or contravarient) and a map $g:Y \to X$ in $\mathcal{C}$ such that $F(f)$ and $F(g)$ are inverses of the other. Then $f$ is an isomorphism.
\end{theorem}
\begin{proof}This is a consequence of the work \cite{AK}. As this theorem is never stated in this form let us explain how to deduce it from loc. cit. We will suppose $F$ to be covariant (otherwise one can compose with the canonical antiequivalence of $\mathcal{D}$ with its opposite category).

 First of all we can replace $\mathcal{C}$ by the $K$-linear symmetric tensor category generated by finite-dimensional objects, hence suppose that $\mathcal{C}$ is a \textit{Kimura category}. The kernel of the functor $F$ is a tensor ideal of $\mathcal{C}$, which is by hypothesis non-trivial. Then by \cite[Theorem. 9.2.2]{AK} the kernel is contained inside a tensor ideal $\mathcal{R}$ called the \textit{radical}  and $\mathcal{C}$ is a \textit{Wedderburn category}. Write  $\pi: \mathcal{C} \to \mathcal{C} / \mathcal{R}$ for the canonical functor. Our hypothesis implies that $\pi(f)$ and $\pi(g)$ are  inverses of each other. We can now conclude because by \cite[Proposition 1.4.4(b)]{AK} the functor $\pi$ detects the isomorphism, i.e., $\pi(f)$ is an isomorphism if and only if $f$ is.
 
Notice that the erratum of \cite{AK} does not concern the statements we are using.
\end{proof}

\subsection{Symmetric coalgebra}
Let $X$ be an odd object of dimension $N$ in a pseudo-abelian $\Q$-linear symmetric tensor category with unit $\one$.
%
We define the {\em symmetric coalgebra}
\[\OSym(X)=\bigoplus_{n=0}^N\Sym^n(X)\ .\]


The coalgebra $\OSym(X)$ has a canonical structure of Hopf algebra. Multiplication and comultiplication are defined so that $\OSym(X)$ becomes the standard symmetric algebra in the opposite category. For details see Appendix 
\ref{appsym}.


\subsection{Graded super-vector spaces} 
Let $F$ be a field. The category of {\em graded super-$F$-vector spaces} will be denoted by $\GrVec^{\pm}_{F}$. As an $F$-linear tensor category it is equivalent to the category of graded $F$-vector spaces. In particular if $V=\bigoplus_{i \in \Z}V_i$ and $W=\bigoplus_{i \in \Z}W_i$, then
\[(V\otimes W)_n=\bigoplus_{p+q=n}V_p\otimes W_q.\]
The difference is in the convention of the symmetry. If $\sigma_{i,j}: V_i\otimes W_j \to W_j \otimes V_i$ is the isomorphism of symmetry in the category of $F$-vector spaces, then the isomorphism of symmetry in $\GrVec^{\pm}_{F}$ is given by $(-1)^{i\cdot j} \sigma_{i, j}$.
\subsection{Chow and Voevodsky motives}\label{sect_motives}

Let $k$ be a perfect field (for the non-perfect case see Section \ref{reduce perfect}).
Our conventions on motives will follow notations and constructions from Voevodsky \cite{TMF} and especially from \cite[Section 14]{MVW}, where the case of rational coefficients is explicitly treated. We will study the following categories:
\[\STEkQ\ ,\ \DMeetkQ\ ,\ \DMegk_\Q\ ,\ \DMgk_\Q\ ,\ \ChowkQ\]
of {\em \'etale sheaves with transfers}, {\em motivic complexes}, {\em effective geometric motives}, {\em geometric motives}, and {\em Chow motives}, respectively. We are going to need the functors
\[L:\Sm_k \to D^-(\STEkQ) \hspace{0.5cm} \textrm{and} \hspace{0.5cm} q: D^-(\STEkQ) \to  \DMeetkQ\ \ .\]
We will write 
\[M= q \circ L \ ,\] 
and for any variety $X\in \Sm_k$ we will call $M(X)$ the {\em motive} of $X$.

Let us give more details. The category of sheaves of $\Q$-vector spaces with transfers on the big \'etale site on $\Smk$ is denoted \[\STEkQ \ . \] By \cite[Lemma 14.21]{MVW} a  rational presheaf with transfers is a sheaf for the Nisnevich topology if and only if it is a sheaf for \'etale topology. We decide to work with the latter topology.  Tensor product is exact on this category by \cite[Proposition B.1]{Sug}. Moreover, by \cite[Lemma 14.21]{MVW} the categories $\DMeetkQ$ and $\DMen(k,\Q)$ are equivalent.

By \cite[Theorem 14.28]{MVW} and \cite[Remark 14.29]{MVW} there are two adjoint $\Q$-linear exact functors 
\[ q:D^-\STEkQ \rightleftarrows \DMeetkQ :i\ . \]
The functor $q$ is compatible with tensor products. Moreover, the functor $i$ is an embedding and the functor $q$ is a localization. We may identify $\DMeetkQ$ with its image in 
$D^-\STEkQ$ and avoid to write $i$.

 The category of {\em effective geometric motives} will be denoted by 
\[\DMegk_{\Q} \ . \] 
Several equivalent definitions are possible, here we define it as the full triangulated $\Q$-linear  pseudo-abelian tensor sub-category of $\DMeetkQ$ generated by motives of smooth schemes over $k$. 
\begin{remark}
\begin{enumerate}
\item Let $\DMeg(k,\Z)$ be the category of effective geometric motives with integer coefficients defined by Voevodsky \cite[Definition 2.1.1]{TMF}, then 
\[\DMegk_{\Q}= \DMeg(k,\Z)_{\Q}\ , \] 
where the tensor product on the right hand side is in the sense of Section~\ref{sect_rational}. 
This fact can be deduced from \cite[Theorem 11.1.13]{CD}, \cite[Remark 9.1.3(3)]{CD} and \cite[Corollary 16.1.6]{CD}.

 \item The reader should be warned that $\STEkQ$ and $\DMeetkQ$ are not $\STEkZ_{\Q}$ and $\DMeetkZ_{\Q}$. We will always work with $\STEkQ$ and $\DMeetkQ$.
 \end{enumerate}
 \end{remark}

After tensor-inverting the Lefschetz motive inside $\DMegk_{\Q}$ we obtain a $\Q$-linear pseudo-abelian triangulated rigid tensor category, denoted by $\DMgk_{\Q}$ and called the category of {\em geometric  motives}. The category $\DMegk_{\Q}$ is canonically a full  subcategory of $\DMgk_{\Q}$ by Voevodsky's Cancellation Theorem \cite{Canc}.

The $\Q$-linear pseudo-abelian rigid tensor sub-category of $\DMgk_{\Q}$ generated by motives of smooth and proper schemes over $k$ will be denoted by $\ChowkQ$ and called the category of Chow (or pure) motives (with $\Q$ coefficients). 
Note that by \cite[Proposition 2.1.4]{TMF} together with \cite{VoeChow}, this category is equivalent to the \textit{opposite} of the classical category of Chow motives.

Although we will use properties of classical Chow motives transposed in this context, we will always keep the notations coming from the conventions of Voevodsky, for instance the motive of the projective line
 decomposes as $M(\mathbb{P}^1_k)=\mathds{1} \oplus \mathds{1}(1)[2].$

We are going to need the following facts. 
\begin{proposition}\label{stalk}
Let $f$ be a morphism in $\DMeetkQ$. Then $f$ is an
isomorphism if and only if its pull-back $\overline{f}$ to $\bar{k}$ is. In particular a motive $M$ vanishes if and only if $\overline{M}$ does.
 \end{proposition}
\begin{proof}
Let $f$ be a morphism in $\DMeetkQ$. Then $f$ is an
isomorphism if and only if it is an isomorphism in the derived category of \'etale sheaves, i.e., if it induces isomorphism on cohomology sheaves without transfers. A morphism of \'etale sheaves is an isomorphism if it is an isomorphism on geometric stalks. It suffices to consider geometric points over closed points of objects in $\Sm/k$.
Hence we can check it after pull back to the algebraic closure. 
\end{proof}

\begin{remark}\label{tensorexact}
Voevodsky's tensor product on $\STEkQ$ is exact 
because by \cite[\S 5 and \S 6]{SV1} the category can be seen as a full tensor 
subcategory of the category of $\mathrm{qfh}$-sheaves with the standard tensor product of sheaves. This argument is due to Sugiyama, see \cite[Appendix B]{Sug} for full details.
\end{remark}

\subsection{Motives over non-perfect fields}\label{reduce perfect}

To avoid the hypothesis of perfectness one can consider Beilinson motives with rational coefficients following Cisinski and D\'eglise \cite{CD}. These motives are defined over general bases and when the base is a perfect field their definition is equivalent to the category of geometric Voevodsky motives with $\Q$-coefficients $\DMegkQ$, see \cite[Corollary 16.1.6]{CD}. 

Suppose now that $k$ is not perfect, and let $k^i$ be its perfect closure.
Then by \cite[Proposition 2.1.9]{CD} and \cite[Theorem 14.3.3]{CD} the pull-back from $\DMegkQ$ to $\DMeg(k^i)_{\Q}$ is an equivalence of category.

\subsection{Realization functor}\label{ladic}
Fix a prime number  $\ell$ different from the characteristic of $k$.
We denote
\[H^* =\bigoplus_{i\in \Z} H^i: \Sm_k\to\Sm_{\overline{k}}\longrightarrow \GrVec^{\pm}_{\Ql}\]
the contravariant functor of {\em $\ell$-adic cohomology}.
By work of Ivorra \cite[Theorem 4.3]{Ivo} it extends uniquely 
to a controvariant functor  called {\em realization} 
\[H^*: \DMgk \longrightarrow \GrVec^{\pm}_{\Ql}\]
which is a $\Q$-linear symmetric tensor functor, sending the unit to the unit and which moreover verifies $H^*(M(n))=H^*(M)(-n)$ and $H^m(M[n])=H^{m-n}(M)$ for all integers $m$ and $n$ and all motives $M$.

\ 

The structure of $\ell$-adic cohomology of semiabelian varieties is well known.

\begin{lemma}\label{H semiabelian}
Let $G$ be a semiabelian variety. 
Then
\[ H^*(G)=\Sym(H^1(G)) \ . \]
Moreover, let $V_\ell(G)$ be the $\ell$-adic Tate module tensored by $\Q_{\ell}$, then we have
\[ H^1(G)=V_\ell(G)^*\ . \]
\end{lemma}
\begin{proof}The abelian or torus case are classical. For the semiabelian
case see e.g. \cite[Lemma 4.1 and 4.2]{BSz}. The rational case follows from their stronger assertion with torsion coefficients.
\end{proof}




\begin{lemma}\label{H1 faithful}      
The functor $H^1$ is exact and faithful on $\sAbQ$.

\end{lemma}
\begin{proof}
From the explicit  formula for $H^1$ we get both the exact sequence and a formula for the dimension of $H^1(G)$. In particular, it does not vanish, if $G$ is non-trivial.
\end{proof}

\section{The $1$-motive $\Mof(G)$ of $G$}\label{sect_constr}
Let $k$ be a perfect field and $G$ be a smooth commutative group scheme over $k$ (often a semiabelian variety). In this section 
we are going to construct
 a natural morphism
\begin{align*}
\alpha_{G}: M(G) \lra \Mof(G)
\end{align*}
in the category $\DMeetkQ$ of triangulated motives, where $\Mof(G)$ is the
the $1$-motive of $G$ to be defined below. 
This is based on the work of
Barbieri-Viale/Kahn \cite{BVK}, Spie{\ss}/Szamuely \cite{SS}, and Suslin/Voevodsky \cite{SV1}.

There is no need to work with rational coefficients in this section, but we prefer to put them in the setting (they will become essential later). 
\subsection{Adding transfers}
\begin{definition}\label{defnulG}Let $\ul{G}$ be the presheaf of abelian groups $\Smk$ defined by $G$,
i.e.,
\[ \ul{G}(S)=\Mor_{\Smk}(S,G)\]
for $S\in \Sm/k$. We denote $\ul{G}_\Q=\ul{G}\tensor_\Z\Q$ the presheaf tensor product, i.e.,
\[ \ul{G}_\Q(S)=\Mor_{\Smk}(S,G)\tensor_{\Z}\Q\]
for $S\in\Sm/k$.
\end{definition}

\begin{lemma}[Spie{\ss}-Szamuely, Orgogozo]\label{lemwelldef}
Let $G$ be a smooth commutative group scheme over $k$, then $\ul{G}_\Q$ is an \'etale sheaf with transfers of $\Q$-vector spaces. When $G$ is a semiabelian variety, $\ul{G}_\Q$ is homotopy invariant.
\end{lemma}
\begin{proof}
Spie{\ss} and Szamuely showed in \cite[Proof of Lemma 3.2]{SS} that $\ul{G}$ is
a presheaf with transfers (for more details we refer to second author's thesis \cite[Section 2.3]{ward}).

Homotopy invariance follows as in the abelian or torus case, see \cite[Lemma 3.3.1]{Org}.

We have to check the sheaf condition for $\ul{G}_\Q$. It is enough to check this for covers with finitely many objects, as a general cover can be refined to such a cover. On the other hand the sheaf condition in this case is implied by the sheaf condition of 
$S \mapsto\Mor_{\Smk}(S,G)$, which is true by \'etale descent. 
\end{proof}

Let us recall the relation between finite correspondences $\Cor(S,X)$ and multivalued maps. Let $S$ be a connected smooth $k$-scheme and $W\subset S\times X$ be an irreducible subvariety finite and surjective over $S$, i.e., a primitive finite correspondence. Let $d$ be the degree of $W/S$. By work of Suslin and Voevodsky  \cite[Preamble to Theorem 6.8]{SV1} there is an associated morphism
\[ S\to S^d(X) \ ,\]
where $S^d(X)=X^d/\mathcal{S}_d$ is the symmetric power of $X$.

From this description, the following is straightforward.
\begin{proposition}\label{propspsz}
Let $G$ be a commutative group scheme over $k$. Then there is a canonical map of sheaves with transfers
\[\gamma_G:\Cor(\cdot,G)\otimes_\Z\Q\to \ul{G}_\Q\ ,\]
characterized by the fact that it maps a morphism $S\stackrel{f}{\rightarrow} S^d(G)$ to
\[S\stackrel{f}{\rightarrow} S^d(G)\stackrel{\mu}{\rightarrow}G\ ,\]
where $ \mu:S^d(G)\to G$ is the summation map. Moreover, $\gamma_G$ is natural in $G\in \CGS_k$.
\end{proposition}

\begin{definition}\label{defnalpha}
Let $G$ be a semiabelian variety over $k$. 
We denote
\[ \Mof(G)\in \DMeetkQ \]
the image of the complex given by $\ul{G}_\Q$ concentrated in degree in $0$
under the funtor $q$.
Let
\[ \alpha_G:M(G)\to \Mof(G)\]
be the morphism in $\DMeetkQ$ induced from $\gamma_G:\Cor(\cdot,G)\otimes_\Z\Q\to\ul{G}_\Q$  (by the adjunction between
$D^-(\STEkQ)$ and $\DMeetkQ$, see Section \ref{sect_motives}).
\end{definition}

\begin{remark}
We call $\Mof(G)$ the $1$-motive defined by $G$. Indeed, the definition is a special case of the embedding of the category of $1$-motives into triangulated motives constructed by Barbieri-Viale/Kahn \cite{BVK}.

The notation $\Mof(G)$ should suggest the first K\"unneth component of $G$.
This intuition is justified by  Proposition \ref{Yoneda} and Lemma \ref{natural} (and of course also by the main Theorem \ref{MainThm}).
\end{remark}

\subsection{Elementary properties}\label{fromnowsemiabelian}


\begin{proposition}[Orgogozo]\label{exactMof}\label{Yoneda}
The assignment 
\[  \Mof:\sAb\to\DMeetkQ\]
is an exact functor, i.e., it maps short exact sequences to exact triangles. 
It is isogeny invariant and factors via an exact functor
\[ \Mof:\sAbQ\to\DMeetkQ\ .\]
In particular, it maps direct products in $\sAbQ$ to direct sums and multiplication
by $n$ on a semiabelian variety $G$ to multiplication by $n$ on $\Mof(G)$.

Moreover, $\Mof$ is a full embedding, i.e., for any two semiabelian varieties
$G$ and $H$ we have a natural isomorphism
\[ \Hom_{\sAbQ}(G,H)=\Hom_{\DMeetkQ}(\Mof(G),\Mof(H))\ .\]
\end{proposition}

\begin{proof}
The fact that $\Mof$ is a full embedding is shown more generally for the derived category of $1$-motives in \cite[Proposition 3.3.3]{Org}. Exactness is explained in the preamble of loc. cit.
\end{proof}

\begin{lemma}\label{natural}
The morphism $\alpha_G: M(G) \to \Mof(G)$ is natural in ${G \in \sAb/k}$ and it is always non-zero for $G\neq 0 \in \sAbQ$.
\end{lemma}
\begin{proof}
The statement is clear if we replace $\alpha_G$ by the map $\gamma_G$ of Proposition \ref{propspsz} (we have $\gamma_G \neq 0$ because it does not vanish on $\id_G\in \Cor(G,G)\otimes_\Z\Q$). By adjunction this implies the statements for $\alpha_G$.
\end{proof}

\begin{corollary}\label{formulaMof}
Let $G$ be a semiabelian variety. Let $\mu$ be the multiplication, $-1$ the inverse, $0$ the unit and $\epsilon$ the structural map. Then
\begin{align*}
\Mof(\mu)=+:&\Mof(G)\oplus\Mof(G) \to\Mof(G)\ ,\\
\Mof(-1)=-1:&\Mof(G)\to \Mof(G)\ ,\\
\Mof(0)=0:&\Mof(\Spec k)\to \Mof(G)\ ,\\
\Mof(\epsilon)=0:&\Mof(G)\to \Mof(\Spec k)\ .
\end{align*}
\end{corollary}
\begin{proof} The statement on $\mu$ holds by definition. The second follows from Proposition \ref{exactMof} with $n=-1$.
The last two are trivial
from $\Mof(\Spec k)=0$.
\end{proof}




Recall that  $\bar{(\cdot)}$ denotes the pull-back to $\bar{k}$.

\begin{lemma}\label{basechange}
Let $G$ be a semiabelian variety over $k$. 
\begin{align*}
   \overline{M(G)}=M(\overline{G})\ ,\\
  \overline{\Mof(G)}=  \Mof(\overline{G})\ ,\\
  \overline{\alpha_G}=\alpha_{\overline{G}}\ .
\end{align*}
\end{lemma}
\begin{proof}
The analogous statement in $\STE(\bar{k},\Q)$ for $\Cor(\cdot,G)\otimes_\Z\Q$, $\ul{G}_\Q$ and $\gamma_G$ can be proved just by checking the definitions. Then one deduces the statement by applying the functor $q$.
\end{proof}

\section{Main Theorem}\label{sect main theorem}
In this section we state the main theorem. 
The point is to construct a morphisms $\vp_G$ and the theorem will essentially state that it is an isomorphism. We also establish some basic properties of $\vp_G$.
\subsection{The morphism $\vp_G$}

\begin{definition}
Let $G$ be a semiabelian variety over $k$, $n \geq 0$ be an integer and $\Delta^n_G$ be the $n$-fold diagonal. We define $\vp_G^n$ to be the morphism
\[ \vp_G^n:  M(G) \sxra{M(\Delta_{G}^{n})} M(G)^{\ox n} \sxra{\alpha_{G}^{\ox n}} \Mof(G)^{\ox n}.\]
As $\Delta^n_G$ is invariant under permutations, this factors uniquely 
\[\xymatrix{
M(G)\ar[rr]^{\alpha_G^{\ox n}M(\Delta_G^n)}\ar[rd]_{\vp^n_G}&&   \Mof(G)^{\ox n}\\
   &  \Symn(\Mof(G))\ar[ur]
}\]
\end{definition}
\begin{remark} Equivalently, we have
\[\vp_G^n: M(G) \sxra{M(\Delta_{G}^{n})} M(G)^{\ox n} \sxra{\alpha_{G}^{\ox n}} \Mof(G)^{\ox n}\lra \Symn(\Mof(G)),\]
which was the original definition in \cite{ward}.
\end{remark}

\begin{definition}\label{defnall}
Let $G$ be a semiabelian variety over $k$ which is an extension of an abelian variety of dimension $g$ by a torus of rank $r$. 
Define the morphism $\vp_{G}$ as 
\[ \vp_{G}= \bigoplus_{n = 0}^{2g+r} \vp_G^n: M(G) \lra \bigoplus_{n = 0}^{2g+r}\Symn(\Mof(G)).\]
\end{definition}

Our main theorem is the following; it will be proven in Section \ref{proofMainThm}.

\begin{theorem}\label{MainThm}
Let $k$ be a perfect field and $G$ be a semiabelian variety which is an extension of an abelian variety of dimension $g$ by a torus of rank $r$. Then the motive $\Mof(G)$ is odd of dimension $2g+r$ and the map 
\[ \vp_{G}: M(G) \lra \bigoplus_{n = 0}^{2g+r}\Symn(\Mof(G)).\]
is an isomorphism of motives. 
\end{theorem}
More refined statements will be deduced in Section \ref{secconseq}.

\subsection{First properties and reductions}

Let $G$ be a semiabelian variety.  The motive $M(G)$ has a canonical Hopf algebra structure
induced by morphisms of varieties:
\begin{itemize}
\item multiplication by multiplication on $G$;
\item comultiplication by the diagonal $\Delta:G\to G\times G$;
\item the antipodal map by the inverse on $G$;
\item the unit by the neutral element;
\item the counit by the structure map to the base field.
\end{itemize}

 The aim here is to show that the morphism $\vp_G$ of Definition \ref{defnall} is a natural morphism of Hopf algebras from $M(G)$ to the
 symmetric coalgebra 
\[ \OSym(\Mof(G))=\prod_{n\geq 0}\Symn(\Mof(G))\ .\]
To consider such an object one needs to work under the following assumption:
\begin{finassumption}\label{assumptionfindim}
The motive $\Mof(G)$ is odd of dimension $2g+r$ (here $G$ is a semiabelian variety which is an extension of an abelian variety of dimension $g$ by a torus rank $r$). 
\end{finassumption}

\begin{remark}The finiteness assumption is needed to make $\OSym(\Mof(G))$ an
object of $\DMeetkQ$.
We are going to establish later (see Proposition \ref{propodd}) that this assumption is always satisfied.
The reader who does not want to work under this assumption can simply work unconditionally in the procategory $\Pro$-$\DMeetkQ$.
\end{remark}

\begin{lemma}\label{coalg}
Under the finiteness assumption \ref{assumptionfindim}, the map $\vp_G$ is the unique morphism of commutative coalgebras extending $\alpha_G$ (Definition \ref{defnalpha}).
\end{lemma}
\begin{proof}
From the definitions it is clear that it is a morphism of coalgebras.
Uniqueness then comes by universal property of $\OSym(\Mof(G))$.
\end{proof}
\begin{lemma}\label{naturalvp}
The maps $\vp^n_G$ are natural in ${G\in\sAb}$. 
\end{lemma}
\begin{proof}
Clear by construction and naturality of $\alpha_G$ 
(Lemma \ref{natural}).
\end{proof}

\begin{lemma}\label{basechangevp}
Let $\bar{k}$ be an algebraic closure of $k$ and denote $\bar{(\cdot)}$ the base change to $\bar{k}$. Then 
\[\overline{\vp^n_G}=\vp^n_{\overline{G}}\ .\]
\end{lemma}
\begin{proof}
By construction this comes from the case of $\alpha_G$  (Lemma \ref{basechange}).
\end{proof}

\begin{corollary}\label{reduction1}
Let $G$ be a semiabelian variety over $k$. Then
Theorem~\ref{MainThm} holds for $G$ if it holds for $\overline{G}$.
\end{corollary}
\begin{proof}
By 
Lemma \ref{basechange}, we have $\Mof(\overline{G})=  \overline{\Mof(G)}$. 
We apply Proposition~\ref{stalk} to the motives $\Symn \overline{\Mof(G)}$  and we obtain that $\Mof(G)$ is odd of the same dimension as
$\Mof(\overline{G})$. In particular, the finiteness assumption \ref{assumptionfindim} is verified for $G$. By Lemma \ref{basechangevp} we have $\overline{\vp_G}=\vp_{\overline{G}}$, we then conclude applying Proposition \ref{stalk} to $f=  \vp_G$.
\end{proof}
 
We want to study compatibility of $\vp_G$ with products.
Let $G$ and $H$ be two semiabelian varieties.
Recall that we have
\begin{gather*}
\Mof(G\times H)=\Mof(G)\oplus \Mof(H)\\
M(G\times H)=M(G)\otimes M(H)\\
\OSym(\Mof(G)\oplus\Mof(H))=\OSym(\Mof(G))\otimes \OSym(\Mof(H))
\end{gather*}
by additivity of $\Mof$ (Lemma \ref{exactMof}), K\"unneth formula and
Corollary \ref{opsymsum}.

\begin{proposition}\label{product}
Let $G$ and $H$ be two semiabelian varieties over $k$ satisfying finiteness assumption \ref{assumptionfindim}. Then, under the above identifications, we have
 \[\vp_{G\times H} = \vp_{G} \otimes \vp_{ H} \ ,\]
 i.e., the diagram
\begin{align*}
\xymatrix{
M(G \x H) \ar[d]^{\isom} \ar[r]^{\vp_{G \x H}} & 
\OSym(\Mof(G \x H)) \ar[r]^{\isom} & 
\OSym(\Mof(G) \os \Mof(H)) \\
M(G) \ox M(H) \ar[rr]_{\vp_{G} \ox \vp_{H}}  & & \OSym(\Mof(G)) \ox \OSym(\Mof(H))\ar[u]^{\isom} }
\end{align*}
commutes.
\end{proposition}
\begin{proof}
First, notice that $G\times H$ verifies the finiteness assumption \ref{assumptionfindim}. Now, by the universal property of the coalgebra $\OSym$
it is enough to check that the diagram
\begin{align*}
\xymatrix{
M(G \x H) \ar[d]_{\isom} \ar[r]^{\alpha_{G \x H}} &  
\Mof(G \x H) \ar[r]^{\isom}  & 
\Mof(G) \oplus \Mof(H)  \\
M(G) \ox M(H) \ar[rr]_{(\alpha_{G} \ox \epsilon_G) \os (\epsilon_H \otimes\alpha_{H})} &  & \Mof(G) \ox \one \os \one \ox \Mof(H)\ar[u]^{\isom}
}
\end{align*}
commutes,  which is the case by Corollary \ref{formulaMof}.
\end{proof}
\begin{corollary}\label{correduction2}
Let $G$ and $H$ be connected semiabelian varieties satisfying finiteness assumption \ref{assumptionfindim}. Then
$\vp_{G\times H}$ is an isomorphism if and only if $\vp_G$ and $\vp_H$ are isomorphisms.
\end{corollary}
\begin{proof}If $\vp_G$ and $\vp_H$ are isomorphisms, so is $\vp_G\tensor \vp_H$. For the converse note that $M(G)$ is a direct factor of $M(G\times H)$
via $G\to G\times H$.
\end{proof}


\begin{proposition}\label{Hopf}
Under the finiteness assumption \ref{assumptionfindim}, the morphism $\vp_G$ is a morphism of Hopf algebras.
\end{proposition}
\begin{proof}
Comultiplication is part of the Lemma \ref{coalg}. Antipode, unit and counit are special cases of naturality Lemma \ref{naturalvp}. Multiplication is also reduced to naturality using Proposition \ref{product}.
\end{proof}




\section{Special cases}\label{sect special}
Before giving a full proof we need to address the cases of tori and abelian varieties, the two building blocks of the category of semiabelian varieties.
The case of tori is simple from the properties that we have established so far.
In the case of abelian varieties, there are  two key ingredients: some properties of the Chow motive of $A$ (after Deninger, Murre, K\"unnemann and Kings) and a computation of Voevodsky of the motive of a curve. 
We also draw some consequences from these partial results.

\subsection{The case of tori}
Recall that $(\CGS/k)_\Q$ is the category of commutative groups schemes of finite type  over $k$ up to isogeny.

\begin{definition}\label{M1torus}Let $T$ be a torus over $k$. 
The {\em cocharacter sheaf} $\Xi(T)$ of $T$ is
defined as the sheaf of $\Q$-vector spaces on the small \'etale site of $k$ given by
\[ K\mapsto\Hom_{(\CGS/K)_\Q}((\Gm)_K,T_K)\]
for all finite field extensions $K/k$ 
By abuse of notation, we also write $\Xi(T)$ for the pull-back to the category $\STEkQ$ of \'etale sheaves with transfers and for its image in $\DMeetk_{\Q}$.
\end{definition}
\begin{remark}\label{artintorus} 
The motive $\Xi(T)$ is an Artin motive.
\end{remark}
\begin{proposition}\label{toruscase}
Let $T$ be a torus over $k$ of rank $r$. Then the main Theorem \ref{MainThm} holds for $G=T$, i.e., $\Mof(T)$ is odd of dimension $r$ and 
\[\vp_T:M(T)\to\OSym(\Mof(T))\]
is an isomorphism.
 Moreover, the natural pairing
\[ \Hom_{\CGSQ}(\Gm,T)\times\ul{\Gm}_\Q\to\ul{T}_\Q\]
defines a map
\[ \Xi(T)\tensor \one(1)[1]\to \Mof(T)\ \]
which is an isomorphism.
\end{proposition}
\begin{proof}Let us start with the first part of the statement.
Let first $T=\Gm$. In this case it is well-known that 
\[ M(\Gm)=\one\oplus\one(1)[1]\]
with $\one(1)[1]=\ul{\Gm}_\Q$ and the factor $\one$ is the image of the projector induced by the constant automorphism of $\Gm$. We claim that $\alpha$ induces an isomorphism
\[\one(1)[1]\to\Mof(\Gm)\ .\]
The proper way of showing this would be to analyze the constructions carefully. However, as $\Hom_{\DMeetkQ}(\one,\one)=\Q$ we can use a 
quicker argument instead.
By the naturality of Lemma \ref{natural}, the morphism $\alpha|_{\Gm}$ vanishes when restricted to the unit motive $\one$, hence it suffices
to check that 
\[\alpha_{\Gm}:M(\Gm)\to\Mof(\Gm)\]
is non-zero. This was pointed out in Lemma \ref{natural}. Note that 
$\Mof(\Gm)$ is odd of dimension $1$ because $\one$ is even of dimension $1$.

Now let $T$ be general.
By Corollary \ref{reduction1} it suffices to consider the case
$k$ algebraically closed. Hence $T\isom\Gm^r$. 
The assertion now follows from
Corollary
\ref{correduction2} and $T=\Gm$.

We now turn to the identification of $\Mof(T)$. 
In order to check that it is an isomorphism, we use the same reduction steps as before. Without loss of generality, $k$ is algebraically closed (Proposition \ref{stalk}) and hence $T\isom\Gm^r$. By compatibility with products it suffices to consider the case
$r=1$ which is tautological. 
\end{proof}

\begin{remark}The analogous computation for the associated graded of the slice filtration is shown in \cite[Proposition 7.2]{HK}, even with integral coefficients. Note that by loc. cit. Corollary 7.9 we do not expect the integral version of Proposition \ref{toruscase}.
\end{remark}
\begin{remark}\label{det torus} Let $T$ be a torus of rank $r$. Then 
\[ \det(\Mof(T))=\Sym^r(\Mof(T))=\left(\bigwedge^r\Xi(T)\right)(r)[r]=\det \Xi(T)(r)[r]\] is
a finite-dimensional motive. It is odd of dimension $1$ when $r$ is odd and it is even of dimension $1$ when $r$ is even. However, it is not always isomorphic
to $\one(r)[r]$ as the example of a non-split torus of rank $1$ shows.

The Artin motive $\det\Xi(T)$ is even of dimension $1$. It is given by a one-dimensional continuous representation of the absolute Galois group of $k$ in the category of $\Q$-vector spaces. It is either trivial or of order $2$.
\end{remark}

\subsection{The Chow motive of an abelian variety}
We recall here some classical results on the Chow motive of an abelian variety.

Let us recall some notations and conventions: $\ChowkQ$ is the  pseudo-abelian $\Q$-linear rigid symmetric tensor category of Chow motives over $k$ with rational coefficients, endowed with the {\em covariant} $\Q$-linear symmetric tensor functor called \emph{motive}  
\[M :\SmPrVar/k \rightarrow \ChowkQ\ .\]
For a detailed description of  this category see \cite{DeMu} or \cite{Ku1}. 
\begin{remark}
Recall that the standard convention in the classical literature on Chow motives is to take
the functor to Chow motives to be controvariant. By replacing a cycle by
its transpose we can pass to the covariant version. Note that this
operation interchanges the notions of symmetric algebra and symmetric coalgebra.
Note also that \cite{DeMu} and \cite{Ku1} use the notation
$\bigwedge^i$ instead of $\Sym^i$. 
\end{remark}

\begin{theorem}[{\cite[Thm. 3.1]{DeMu}}]\label{DM}
For any abelian variety over $k$ of dimension $g$  there is a unique decomposition in $\ChowkQ$
\[M(A)=\bigoplus_{i=0}^{i=2g}\hof{i}(A),\]
 which is natural in $A \in \Ab/k$, and such that for all $n\in\Z$
 \[M(n_A)= \bigoplus_{i=0}^{i=2g} n^i \cdot\id_{\hof{i}(A)}.\]
Moreover, one has $\hof{0}(A)=\one$.
\end{theorem}

\begin{theorem}[{\cite[Thm. 3.1.1 et 3.3.1]{Ku1}}]\label{Ku}
For any abelian variety $A$ over $k$ of dimension $g$ the following holds:
\begin{enumerate}
\item\label{abel poincare} The Poincar\'e duality holds \[\hof{2g-i}(A)^{\vee}=\hof{i}(A)(-g)[-2g]\ .\] 
In particular  one has $\hof{2g}(A)=\one(g)[2g]$.
\item For $i> 2g,$ the motive $ \Sym ^i \hof{1}(A)$ vanishes. 
\item The canonical morphism of coalgebras
 \[M(A) \longrightarrow \OSym(\hof{1}(A))\]
 induced by the projection $M(A) \to  \hof{1}(A)$ is an isomorphism. It respects the grading, i.e. it induces isomorphisms
 \[\hof{i}(A) \longrightarrow \Sym^i(\hof{1}(A))\ .\]
\end{enumerate}
\end{theorem}

\begin{proposition}[{\cite[Prop. 2.2.1]{Ki}}] \label{Ki}
For all pairs of abelian varieties $A$ and $B$, the functor $\hof{1}$ induces an isomorphism of $\Q$-vector spaces
\[\Hom_{\AbQ}(A,B) \isocan \Hom_{\ChowkQ}(\hof{1}(A), \hof{1}(B)).\]
\end{proposition}

\begin{proposition}\label{Chow Jacobian}
Let $C$ be a smooth and projective curve over $k$ and $J(C)$ its Jacobian. 
Suppose that $C(k)\neq \emptyset$ and let $x_0: C \to C$ be a constant map.
Then the motive of the curve can be decomposed as
\[M(C)=\one  \oplus \hof{1}(J(C))  \oplus  \one(1)[2]\]
such that the projector to $\one$ is given by $M(x_0)$ and the projector to $\one(1)[2]$ by
its Poincar\'e dual.
\end{proposition}
\begin{proof}
This is classical, see for example \cite[Proposition 4.5]{Scholl}. Notice that by Theorem 5.3 of loc. cit. the different notions of motivic $\hof{1}$ for an abelian variety coincide. 
\end{proof}

\subsection{The Voevodsky motive of an abelian variety}
We consider $A$ an abelian variety of dimension $g$ over the base field $k$. 

In order to prove the main Theorem \ref{MainThm} in this special case, the key point is to show that $\alpha_A$ induces an isomorphism between $\hof{1}(A)$ and  $\Mof(A)$.  We will reduce this to the case of Jacobians and then use Proposition \ref{Chow Jacobian}
and a parallel result of Voevodsy for geometric motives.

\begin{lemma}\label{non zero}
Consider the decomposition $M(A)=\bigoplus_{i=0}^{i=2g}\hof{i}(A)$ of  Theorem \ref{DM} and the map 
$\alpha_A :  M(A) \rightarrow   \Mof(A)$  of definition \ref{defnall}. 
Then the restriction of the map to each factor $\hof{i}(A)$ is zero for all $i\neq 1$. Moreover, the induced map 
\[\alpha_A :  \hof{1}(A) \rightarrow   \Mof(A)\]
is non-zero.
\end{lemma}

\begin{proof}
By Lemma, \ref{natural} the map $\alpha_A$ is natural in $\sAb_{\Q}$. On the other hand the action of the multiplication $n_A$ is equal to  $n^i \cdot\id$ on $\hof{i}(A)$ (Theorem \ref{DM}) and to $n\cdot\id$ on $\Mof(A)$ (Proposition \ref{Yoneda}). This implies that $\alpha_A $ is zero on $\hof{i}(A)$ for $i\neq 1$. 

To conclude, notice that the restriction of $\alpha_A$ to $\hof{1}(A)$ has to be non-zero, otherwise the whole $\alpha_A$ would be zero, which contradicts Lemma \ref{natural}.
\end{proof} 
\begin{lemma}\label{Voe Jacobian}
For any smooth and  proper curve $C$ with a rational point, the motives $\hof{1}(J(C))$ and $\Mof(J(C))$ are isomorphic.
\end{lemma}
\begin{proof}
Consider $M(C)\in D^-(\STEkQ)$. It is cohomologically concentrated in degrees
$0$ and $-1$. The cohomology in degree $0$ is by \cite[Theorem 3.4.2]{TMF}  
given by $\Pic(C)$. In degree $-1$ it is equal to $\sheaf{O}^*$. 

We use the projector given by the rational point $x$ and its Poincar\'e dual to split off 
$\one\oplus \one(1)[2]=\Z\oplus\sheaf{O}^*[1]$. The remaining object is cohomologically concentrated in degree $0$ and given by the kernel of the degree map $\Pic(C)\to \Z$, hence
isomorphic to $\Mof(J(C))$.

By comparing with the decomposition
in Proposition \ref{Chow Jacobian} we get the result.
\end{proof}

\begin{proposition}\label{iso mot H1 abel}
For any abelian variety $A$, the map
\[\alpha_A|_{\hof{1}(A)} :  \hof{1}(A) \rightarrow   \Mof(A)\]
is an isomorphism.
\end{proposition}
\begin{remark}
In \cite[Section 4.3]{ward} this is established by going through the definitions carefully. 
The proof given here is different.
\end{remark}
\begin{proof}
We can assume that $k$ is algebraically closed by Proposition \ref{stalk} and Lemma \ref{basechange}.
We may decompose $A$ up to isogeny into simple factors. 
The map $\alpha_A$ is natural in $\AbQ$ and compatible with direct products.
Hence it suffices to consider the case of a simple abelian variety. 
We can choose a curve $C$ such that $A$ is a factor of $J(C)$ up to isogeny.
As $k$ is algebraically closed, we can apply Lemma \ref{Voe Jacobian}
and deduce that there is some isomorphism
\[ \hof{1}(J(C))\to\Mof(J(C))\ .\]
Hence we may view $\hof{1}(A)$ and $\Mof(A)$ as direct factors of the same
object $X$ and $\alpha|_{\hof{1}(A)}$ as an endomorphism of $X$. 
They are both simple factors by Proposition \ref{Yoneda} and Proposition \ref{Ki}. Hence any non-zero map between them is an isomorphism. 
By Lemma \ref{non zero} this is the case for $\alpha_A$.
\end{proof}

\begin{proposition}\label{abelian case}
Let $A$ be an abelian variety over $k$ of dimension $g$. Then the main Theorem~\ref{MainThm} holds for $G=A$, i.e., $\Mof(A)$ is odd of dimension $2g$ and
\[\vp_A:M(A)\to\OSym(\Mof(A))\]
is an isomorphism.
\end{proposition}

\begin{proof}
First, $\hof{1}(A)$ is odd of dimension $2g$ by Theorem \ref{Ku}. So by Proposition \ref{iso mot H1 abel} the same holds for $\Mof(A)$.
Let us consider the following commutative diagram:
 \[\xymatrix{
  M(A) \ar[d]  \ar[rd]^{\alpha_A} & {} \\
  \hof{1}(A) \ar[r]_{{\alpha_A}_{|\hof{1}(A)}} & \Mof(A) }
\]
where the vertical arrow is the projection. By the universal property \ref{universalalgebra} it induces a unique commutative diagram 
\[\xymatrix{
  M(A) \ar[d]  \ar[rd] & {} \\
 \OSym( \hof{1}(A)) \ar[r] & \OSym(\Mof(A)) }
\]
of morphisms of coalgebras. The diagonal morphism is $\vp_A$ by Lemma \ref{coalg}. By Proposition \ref{iso mot H1 abel}  $\alpha_A :  \hof{1}(A) \rightarrow   \Mof(A)$ is an isomorphism, so the horizontal arrow $\OSym(\hof{1}(A)) \to \OSym(\Mof(A))$ is an isomorphism. The vertical arrow is an isomorphism by Theorem \ref{Ku}. We deduce that $\vp_A$ is an isomorphism.
\end{proof}

\section{Properties of $\Mof(G)$}\label{sect real}
In all the section, $G$ is a semiabelian variety over $k$. We consider the basic exact sequence
\[1\to T\to G\to A\to 1\ , \]
with $T$ a torus of rank $r$ and $A$ an abelian variety of dimension $g$. We establish properties for $\Mof(G)$ which we already know for $\Mof(T)$ and $\Mof(A)$.

\subsection{The motive $\Mof(G)$ is Kimura finite}
\begin{proposition}\label{propodd}Let $G$ be 
a semiabelian variety over $k$ which is an 
extension of an abelian variety of dimension $g$ by a torus $T$ of rank $r$.
Then the motive  $\Mof(G)$
is odd of dimension $2g+r$, i.e., $ \Sym ^n( \Mof(G))$ vanishes for $n>2g+r$, and the motive
\[ \det(G):=  \det(\Mof(G))=\Sym^{2g+r}\Mof(G)\]
is of the form
\[ \Lambda(g+r)[2g+r]\]
where $\Lambda$ is the tensor-invertible Artin motive $\det \Xi(T)$ of Remark \ref{det torus}. In particular, if the torus part of
$G$ is split, then $\Lambda=\one$.
\end{proposition}
\begin{proof} Let $1\to T\to G\to A\to 1 $ be the basic sequence and consider the associated filtration of Appendix \ref{sectfilttensor}
\[ \Fil_i^{\Mof(T)}\Sym^{n}\Mof(G)\ .\]
By Proposition \ref{qniisomo2} its associated graded pieces have the form 
\[ \Sym^i\Mof(T)\tensor\Sym^{n-i}\Mof(A)\ .\]
Hence, they vanish in $\DMeetkQ$ for $i> r$ (see Proposition \ref{toruscase}) or $n-i>2g$ (see Theorem \ref{Ku}). This implies vanishing for $n>2g+r$. For $n=2g+r$
we get a canonical isomorphism
\[\det\Mof(G)=\det\Mof(T)\tensor\det\Mof(A)\ .\]
Hence the formula follows  from the Proposition \ref{toruscase} for tori and Theorem \ref{Ku} for abelian varieties. Note that $\Lambda$ is indeed invertible.
\end{proof}
The following is not needed in the sequel.
\begin{corollary}\begin{enumerate}
\item Let $C$ be a curve (not necessarily smooth and projective). Then $M(C)$ is Kimura finite. 
\item\label{findim1motives} Fix the embedding of the category of $1$-motives \cite{DeH3} in $\DMeetkQ$ to be the one constructed in \cite{BVK,Org}. Then all $1$-motives  are even objects in $\DMeetkQ$.
\end{enumerate}
\end{corollary}
\begin{proof}
By \cite[Theorem 11.2.1]{BVK}  the motive $M(C)$ decomposes into the sum of an Artin-motive and a $1$-motive (shifted by $1$). Then, it is enough to show (\ref{findim1motives}).

Consider a $1$-motive $[F\to G]$ as a complex where  $F$ (in degree $0$) is a $k$-group scheme such that $F_{\bar{k}}$ is a free abelian group of finite rank  and $G$ (in degree $1$) is a semiabelian variety. 

 By Proposition \ref{qniisomo2}, the filtration (in the abelian category $C^b(\STEkQ)$)
\[ 0\to [0\to \ul{G}_\Q]\to [\ul{F}_\Q\to \ul{G}_\Q]\to [\ul{F}_\Q\to 0]\to 0\]
 induces a filtration on 
$\Sym^n[\ul{F}_\Q\to \ul{G}_\Q]$ with associated graded pieces isomorphic to
\[ \Sym^a[0\to \ul{G}_\Q]\tensor\Sym^{n-a}[\ul{F}_\Q\to 0]\ .\]
These exact sequences induce triangles in $\DMeetkQ$ which reduce to the case of  $[F\to 0]$ and $[0\to G]$. The first is just an Artin motive, hence even. The second, by definition, equals to $\Mof(G)[-1]$, which is even by Proposition \ref{propodd}.
\end{proof}

\begin{remark}
\begin{enumerate}
\item
Kimura finiteness of motives of curves was known
by the work of Guletskii \cite{Guletskii} and Mazza \cite{Mazza} (using different methods).
\item The fact that $1$-motives are Kimura
finite is pointed out in \cite[Remark 5.11]{Mazza} (attributed to O'Sullivan) as a consequence of Kimura finiteness of motives of curves. The above is more precise and also more direct.
\end{enumerate}
\end{remark}

\subsection{The motive $\Mof(G)$ is geometric}
\begin{proposition}
The motive $\Mof(G)\in \DMeetkQ$ belongs to the full subcategory $\DMegk_{\Q}$.
\end{proposition}
\begin{remark}The fact that all $1$-motives are geometric is already shown
in \cite{Org}. We give a straight-forward argument in our case.
\end{remark}

\begin{proof}
If $G$ is a torus or an abelian variety, we have established the isomorphism
\[\vp_G:M(G)\to\OSym(\Mof(G))\]
in Proposition \ref{toruscase} and Proposition \ref{abelian case}.
In particular, $\Mof(G)$ is a direct summand of a geometric motive, hence geometric. (Alternatively, we have given an explicit description of $\Mof(G)$
in Proposition \ref{toruscase} and Proposition~\ref{iso mot H1 abel}, which 
is in both cases geometric.)

In general, consider a basic exact sequence fixed in the beginning of the section
\[ 1\to T \to G\to A\to 1 \ .\]
It induces an exact triangle 
\[ \Mof(T) \to \Mof(G)\to \Mof(A)\]
in $\DMeetkQ$. The claim follows because the category of geometric motives is triangulated.
\end{proof}
\subsection{Computation of realization}
\begin{proposition}\label{realization M1}

The realization of the map $\alpha_G: M(G) \to \Mof(G)$ of Definition \ref{defnalpha}
is zero in all degrees except in degree one where it induces an isomorphism
\[H^*(\alpha_G):  H^*(\Mof(G)) \to H^1(G).\]
\end{proposition}
\begin{proof}
Let us start by showing that the statement holds for all $G$ which satisfy the main Theorem \ref{MainThm}.
Indeed, applying $H^*$
one has an isomorphism of Hopf algbras
\[ H^*(G)\isom \Sym(H^*(\Mof(G)))\ .\]
Hence, their primitive parts are isomorphic.
By the structure theory of connected graded Hopf algebras (see for example \cite[Appendix A]{Loday}) the primitive part of on the right hand side is $H^*(\Mof(G))$. On the other hand the primitive part of $ H^*(G)$ is $H^1(G)$ by Lemma \ref{H semiabelian}.

Note that, in particular, we have shown the statement in the toric case and in the abelian  case, by Propositions \ref{toruscase} and \ref{abelian case}. In the general case, write
\[ 1\stackrel{}{\longrightarrow} T \stackrel{f}{\longrightarrow} G \stackrel{g}{\longrightarrow} A \stackrel{}{\longrightarrow} 1 \ .\] 
 Then one has the following commutative diagram of complexes
 \[\xymatrix{
 H^*(\Mof(A))  \ar[d]^{H^*(\alpha_A)} \ar[r]^{H^*(\Mof(g))}  & H^*(\Mof(G))  \ar[d]^{H^*(\alpha_G)} \ar[r]^{H^*(\Mof(f))} & H^*(\Mof(T))  \ar[d]^{H^*(\alpha_T)}\\
  H^1(A) \ar[r]^{H^1(g)} & H^1(G) \ar[r]^{H^1(f)} & H^1(T).}
\]
The two squares are commutative by Lemma \ref{natural}. 
The second line is a short exact sequence by Lemma \ref{H1 faithful} and the first one is a priori just a complex. 
We have just shown that the first and the third vertical arrows are isomorphisms. 
This implies that $H^*(\Mof(g))$ is injective and $H^*(\Mof(f))$ is surjective.
By Proposition \ref{propodd}, the object $\Mof(G)$ is odd 
of dimension $2g+r$. Hence $H^*(\Mof(G))$ has $\Q_{\ell}$-dimension $2g+r$. 
So the first line is actually a short exact sequence. Then we can conclude that also the second vertical column is an isomorphism.
\end{proof}

\section{Proof of the Main Theorem}\label{proofMainThm}
The proof is by induction on the torus rank.
By comparing two triangles, we establish that there is some isomorphism between $M(G)$ and $\OSym(\Mof(G))$ and deduce that these two motives are finite-dimensional. In the next section \ref{end} we show that
$\vp_G$ is an isomorphism studying its behaviour in the realization and using Kimura finiteness.


\subsection{Comparing exact triangles}\label{sec_triangles}

Throughout this section, we consider a short  exact sequence of semiabelian varieties
\[ 1\to\Gm\to G\to H\to 1\ .\]


\begin{lemma}\label{sequence2}
\begin{enumerate}
\item \label{sequence1} Let $n\geq 0$.  We denote
\[[\Mof(G)]:\Mof(H)\to\one(1)[2]\] 
the connecting morphism of the exact triangle
\[ \one(1)[1]\to\Mof(G)\to\Mof(H)\ .\]
Then there  is an exact triangle
\[ \Symn(\Mof(G))\to\Symn(\Mof(H))\xrightarrow{\cup [\Mof(G)]}\Sym^{n-1}(\Mof(H))(1)[2]\ .\]

\item
Let
\[ c_1([G])\in H^1_\et(H,\Gm)
	 =\Mor_{\DMeetkQ}(M(H), \tatt)\]
be the first Chern class of $G$ viewed as a $\Gm$-torsor over $H$.
Then there is an exact triangle
\[ M(G)\to M(H)\xrightarrow{\cdot\cup c_1(G)} M(H)(1)[2]\ .\]
\item The diagram
\begin{align*} 
\xymatrix{    
M(H) \ar[r]^{\cdot\cup c_{1}([G])} \ar[d]^{\vp_{H}}& 
M(H)(1)[2] \ar@{->}[d]^{\vp_{H}(1)[2]} \\
\OSym(\Mof(H)) \ar[r]_{\cdot \cup [\Mof(G)]} & 
\OSym(\Mof(H))(1)[2] 
} 
\end{align*} 
commutes.
\end{enumerate}
\end{lemma}
\begin{proof}
For (\ref{sequence1}),
we apply Theorem \ref{corcup} to the exact sequence of sheaves
with transfers
\[ 0\to\ul{\Gm}_\Q\to\ul{G}_\Q\to\ul{H}_\Q\to 0 \]
and the localization functor
\[ q:\STEkQ\to\DMeetkQ\ .\]
Note that $\STEkQ$ is $\Q$-linear abelian symmetric tensor category
with an exact tensor product, see Remark \ref{tensorexact}, and that
$\Mof(\Gm)=\one(1)[1]$ and $\Sym^2(\one(1)[1])=0$.

For the second triangle,
let $\Aff(G)\to H$ be the line bundle associated to the $\Gm$-torsor
$G$. The zero section of $\Aff(G)$ identifies $H$ with $0(H)$, a smooth subvariety of $\Aff(G)$ of codimension $1$. Its complement is $G$. Hence the localization sequence reads
\[ M(G)\to M(\Aff(G))\to M(0(H))(1)[2]\ .\]
By homotopy invariance $M(\Aff(G))=M(H)$. The identification
of the boundary map with the first Chern class is carried out in \cite[\S 7]{HK}.

By compatibility with comultiplication (which holds by definition of the maps), it suffices to check  commutativity in degree $1$. This is precisely the comparison of Chern classes in Proposition~\ref{propc1}. 
\end{proof}

\begin{corollary}\label{iso psi}
Assume that there exists a short exact sequence of semiabelian varieties
\[ 1\to\Gm\to G\to H\to 1,\]
such that the main Theorem \ref{MainThm} holds for $H$.
Then there exists an isomorphism \[\psi: \OSym( \Mof(G)) \rightarrow M(G)\ .\] 
In particular, $M(G)$ is Kimura finite.
\end{corollary}

\begin{proof}
We consider 
\begin{align*} 
\xymatrix@C=14pt{   
M(G) \ar@{->}[rr]^(0.5){M(g)} \ar@{-->}[dd]_(0.5){\psi} & & 
M(H) \ar@{->}[rr]^(0.45){1_{M(H)} \cup c_{1}(\Aff(G))} \ar@{->}[dd]^(0.5){\vp_{H}}_{\simeq} & & 
M(H)(1)[2] \ar@{->}[dd]^(0.5){\vp_{H}(1)[2]}_{\simeq} \ar@{->}[rr]^(0.52){\pd_{\Aff(G), H} \circ M(s_{0})} & & 
M(G)[1] \ar@{-->}[dd]^(0.5){\psi[1]} \\
& & & & & & \\
\OSym(\Mof(G)) \ar@{->}[rr] & & 
\OSym(\Mof(H)) \ar@{->}[rr] & & 
\OSym(\Mof(H))(1)[2] \ar@{->}[rr] & & \OSym(\Mof(G))[1].
} 
\end{align*} 
Both triangles are constructed in Lemma \ref{sequence2}.
The central square commutes also by Lemma \ref{sequence2}. By assumption $\vp_H$ is an isomorphism. By the axioms of a triangulated category we obtain an isomorphism $\psi$ as indicated.
$M(G)$ is Kimura finite because $\Mof(G)$ is Kimura finite by Proposition \ref{propodd} and the notion is stable under tensor products and direct summands.
\end{proof}

\begin{remark}The above corollary was the main result of \cite{ward}. We expect
$\vp_?$ to define a morphism of triangles, i.e., $\psi=\vp_G$.
This would immediately show that $\vp_G$ is an isomorphism. We were not
able to establish this morphism of triangles and use a completely different argument instead.
\end{remark}

\subsection{$\vp_G$ is an isomorphism}\label{end}
We modify the non-canonical isomorphism $\psi$ (Corollary \ref{iso psi}) such that its  realization is the inverse of the realization of $\vp_G.$ To conclude we use conservativity of the realization functor on finite-dimensional motives.
\begin{isoassumption}\label{assume_psi}
We assume and fix an isomorphism $\psi: \OSym( \Mof(G)) \rightarrow M(G)$ and write
 $\psi_1:  \Mof(G) \rightarrow M(G)$ for its restriction to $\Mof(G)$.
\end{isoassumption}

\begin{lemma}\label{realization psi}
Under the above isomorphism assumption \ref{assume_psi}, the realization  of $\psi_1$ induces an isomorphism
\[H^1(\psi_1): H^1(G)\isocan H^1(\Mof(G)) \]
\end{lemma}

\begin{proof} 
Recall from Proposition \ref{realization M1} that $H^*(\Mof(G))$ is concentrated in degree one.
Hence the realization of $\psi$ gives isomorphisms
\[ H^n(G)\xrightarrow{H^n(\psi)} \Symn(H^1(\Mof))\ .\]
Moreover, $H^1(\psi_1)=H^1(\psi)$.
\end{proof}

\begin{lemma}
Under the isomorphism assumption \ref{assume_psi}, the endomorphism $\alpha_G \circ \psi_1$ of the motive $\Mof(G)$ is an isomorphism. In particular, there exists a morphism $\beta_1:  \Mof(G) \rightarrow M(G)$ such that 
\[ \alpha_G \circ \beta_1 = \id_{\Mof(G)}\ .\] 
\end{lemma}
\begin{proof} We write 
\[ \alpha_G \circ \psi_1= \Mof(f_0)\ .\]
This is possible because
by Proposition \ref{Yoneda} any endomorphism of $\Mof(G)$ is of the form $\Mof(f)$ where $f$ is in $\End_{\sAbQ}(G)$.

It is enough to show that $f_0$ is an automorphism of $G\in \sAbQ$.
As $H^1$ is exact and faithful on $\sAbQ$ (see Lemma \ref{H1 faithful}), we can
test this after applying $H^1$. Moreover,
\[ H^1(f_0)=H^*(\Mof(f_0))=H^*(\psi_1)\circ H^*(\alpha_G)\ .\]
By Lemma \ref{realization M1},  the map $H^*(\alpha_G)$ is an isomorphism onto its image $H^1(G)$. By Lemma  \ref{realization psi}, the map $H^*(\psi_1)$ is an isomorphism when restricted to $H^1(G)$. Hence the composition is an isomorphism.
\end{proof}

\begin{lemma}\label{iso real}
Under the isomorphism assumption \ref{assume_psi}, let us fix  a morphism $\beta_1:  \Mof(G) \rightarrow M(G)$ such that ${\alpha_G \circ \beta_1 = \id_{\Mof(G)}}$ as in the previous lemma. Let
\[\beta:  \OSym\Mof(G) \rightarrow M(G)\]
be the induced morphism of algebras.

Then $H^*(\varphi_G)$ and $H^*(\beta)$ are inverse to each other.
\end{lemma}

\begin{proof}
By Assumption \ref{assume_psi} the vector spaces $H^*(\OSym(\Mof(G)))$ and $ H^*(M(G))$ have the same dimension, in particular it is enough to check that the composition in one direction is the identity.

By  Proposition \ref{Hopf}, $\vp_G$ is not only a morphism of coalgebras but also  a morphism of {\em algebras}. Hence  
\[\vp_G \circ \beta:  \OSym(\Mof(G)) \rightarrow \OSym(\Mof(G))\]
is also a morphism of algebras and so 
\[H^*(\vp_G \circ \beta)= H^*(\beta)  \circ H^*(\vp_G) :  \Sym(H^*(\Mof(G))) \rightarrow \Sym(H^*(\Mof(G)))\]
is a morphism of {\em coalgebras}. 

By Corollary \ref{universalboth}, the bialgebra $\Sym(H^1(\Mof(G))$ also has 
the universal property with respect to comultiplication. We are going to
exploit it in order to establish that $H^*(\vp_G \circ \beta)$ is the identity.

By Proposition \ref{realization M1}, $H^*(\Mof(G)) = H^1(\Mof(G))$ is concentrated in degree one. 
In degree one our morphism is equal to $H^*(\alpha_G \circ \beta_1)$ so it is the identity by assumption.
\end{proof}

\begin{proposition}\label{vp_G iso}
Under the isomorphism assumption \ref{assume_psi}, the morphism
$\vp_G$ is an isomorphism.
\end{proposition}
\begin{proof}
 By Proposition \ref{propodd} and Corollary \ref{iso psi} the objects $\OSym( \Mof(G))$ and $M(G)$ are finite-dimensional. So we can use Theorem \ref{conservativity}, applied to the realization and conclude by Lemma \ref{iso real}.
\end{proof}
\subsection{Conclusion}
\begin{proof}[Proof of the Main Theorem \ref{MainThm}.]
Let $k$ be a perfect field and $G$ be a semiabelian variety over $k$ which is an extension of an abelian variety $A$ of dimension $g$ and a torus $T$ of rank $r$. 

By Proposition \ref{propodd}, the motive $\Mof(G)$ is odd of dimension $2g+r$.
It remains to establish that $\vp_G$ is an isomorphism. 
By Corollary \ref{reduction1} we can suppose that $k$ is algebraically closed. 

We now argue by induction on $r$. When $r=0$ (and hence $G=A$) the theorem is proved by Proposition \ref{abelian case}. 

Let us now consider the case $r\geq 1$. As the ground field $k$ is algebraically closed we have $T\isom\Gm^r$. We fix such a splitting and let $\Gm\to\Gm^r$ be the inclusion as the first coordinate. This defines a short exact sequence
\[ 1\to\Gm\to G\to H\to 1,\]
with $H$ a semiabelian variety of torus rank $r-1$. By inductive hypothesis the theorem holds for
$H$.  
By Corollary \ref{iso psi}, this implies the existence of some isomorphism
\[\psi:M(G)\to\OSym(\Mof(G))\ .\]
This is the isomorphism assumption \ref{assume_psi} for $G$. 
Then Proposition \ref{vp_G iso} shows that  the morphism
$\vp_G$ is an isomorphism.
\end{proof}

\section{Consequences}\label{secconseq}
Let $k$ be a field (not necessarily perfect). We deduce from our main Theorem \ref{MainThm} a K\"unneth decomposition for the motive of a semiabelian variety, the behaviour under Weil cohomology theories and the existence of a weight filtration. 
Finally, we also compute the motives of arbitrary commutative group schemes.

\subsection{K\"unneth components}
In this section we fix a prime number $\ell$ and write $H^*: \DMegkQ \longrightarrow \GrVec^{\pm}_{\Ql}$ for the $\ell$-adic realization, see Section \ref{ladic}.

\begin{theorem}\label{Thmconsequences}
Let $G$ be a semiabelian variety over $k$ which is an 
extension of an abelian variety of dimension $g$ and a torus of rank $r$. 
Then there exists a unique decomposition in $\DMegkQ$
\[M(G)=\bigoplus_{i=0}^{2g+r}\Mi{i}(G)\]
 which is natural in $G \in \sAb/k$ and such that 
\[H^*(\Mi{i}(G))=H^i(G_{\bar{k}},\Q_{\ell}).\]
 Moreover:
 \begin{enumerate}

 \item\label{thmconmultn} The multiplication by $n_G$ acts as
 \[M(n_G)= \bigoplus_{i=0}^{i=2g+r} n^i \cdot\id_{\Mi{i}(G)}.\]
 In particular, for any non-zero integer $n$, the morphism $M(n_G)$ is an isomorphism and hence the K\"unneth decomposition is natural in $G \in \sAbQ$.

 \item\label{thmconext} If $L$ is a field extension of $k$, then we have
 \[\Mi{i}(G)_L= \Mi{i}(G_L)\]
 in $\DMegLQ$.
 
\item \label{thmm1}The image of the motive $\Mi{1}(G)$ in
$\DMeetkQ$ is given by the homotopy invariant sheaf with transfers
\[ S\mapsto \Mor_{\Sch/k}(S,G)\tensor\Q\ .\]

\item\label{thmconYoneda} For all pairs $G_1, G_2 \in \sAb/k_{\Q}$, the functor $\Mi{1}$ induces an isomorphism of $\Q$-vector spaces
\[\Hom_{\sAbQ}(G_1,G_2) \isocan \Hom_{\DMegkQ}(\Mi{1}(G_1), \Mi{1}(G_2)).\]

\item\label{thmconh1exact} Any exact sequence $1 \to G_1 \to G_2 \to G_3 \to 1$ in  $\sAbQ$ induces an exact triangle  
\[ \Mi{1}(G_1) \to \Mi{1}(G_2) \to \Mi{1}(G_3)  \]
in $\DMegkQ$

\item\label{thmconodd} For $i> 2g+r,$ the motive $ \Sym ^i( \Mi{1}(G))$ vanishes.
\item\label{thmconmain}\label{thmcongrad}
The canonical morphism of coalgebras
 \[M(G)=\bigoplus_i\Mi{i}(G) \longrightarrow \OSym  (\Mi{1}(G))=\bigoplus_i\Sym^i(\Mof(G))\]
 induced by the projection $M(G) \to  \Mi{1}(G)$ is a graded isomorphism of 
Hopf algebras. 

 \item\label{thmcontensor} The motive
\[ \det(G):=\Mi{2g+r}(G)=\det(\Mof(G))\]
is of the form
\[ \Lambda(g+r)[2g+r]\]
with a tensor-invertible Artin motive $\Lambda$. If the torus part of
$G$ is split, then $\Lambda=\one$.
 \item\label{thmcondual} There is an isomorphism 
\begin{align*}
\Mi{i}(G)^{\vee} &\cong \Mi{2g+r-i}(G) \otimes \det(G)^{-1}\\
 &=\Mi{2g+r-i}(G)\tensor\Lambda^{-1}(-g-r)[-2g-r]\ 
 \end{align*}
 natural in $G \in \sAb/k_{\Q}$. In particular there are isomorphisms 
\begin{align*}
M(G)^{\vee} &\cong M(G) \otimes \det(G)^{-1}\ ,\\ 
M_c(G) &\cong M(G) \otimes \det(G)^{-1}(g+r)[2g+2r]\\
   &=M(G)\tensor\Lambda^{-1}[r]
\end{align*}
natural in $G \in \sAb/k_{\Q}$ (where $M_c(G)$ is the motive with compact support of $G$).
\end{enumerate}
\end{theorem}

\begin{proof}
By Section \ref{reduce perfect}, we may assume that $k$ is perfect.
We use the main Theorem \ref{MainThm} and choose
\[ \Mi{i}(G)=\Sym^i(\Mof(G))\ .\]
By Proposition \ref{realization M1} it has
the correct behaviour for the realization of $\Mof(G)$ and by
Lemma \ref{H semiabelian}  also for all $\Mi{i}(G)$. 
By Proposition \ref{exactMof} it also satisfies (\ref{thmconmultn}).
Suppose there is another natural decomposition 
\[ M(G)=\bigoplus_{i=0}^{2g+r}\Mi{i}'(G)\ .\]
By naturality, this decomposition is stable under the action of the {$\Q$-algebra} generated by $M(n_G)$. Notice now that, for all $i$, the projector $p_i$ defining $\Mi{i}(G)$ is in this algebra, indeed
\[ p_i=\frac{\prod_{i\neq j} M(n_G) - n^j\id}{\prod_{i\neq j}(n^i - n^j)} \ . \] So we get 
 a decomposition 
\[ \Mi{i}'(G)=\bigoplus_{j=0}^{2g+r}\Mi{i}^j(G)\]
where the motive $\Mi{i}^j(G)$ is a direct factors of $\Mi{i}'(G)$ and of $\Mi{j}(G)$. Hence by hypothesis, its realization is concentrated, on one hand in degree $i$ and on the other hand in degree $j$, which implies that its realization is zero when $i \neq j$. 

Moreover, $\Mi{i}^j(G)$ is a finite-dimensional motive (as $M(G)$ is by Theorem \ref{MainThm}) so we can apply Theorem \ref{conservativity} to deduce that $\Mi{i}^j(G)$ vanishes for $i \neq j$. This gives the uniqueness.  


For (\ref{thmconext}) one can argue in two different ways: using uniqueness of the decomposition or using the multiplication $n_G$ of (\ref{thmconmultn}). Properties (\ref{thmm1}) and (\ref{thmconmain}) hold  by definition and main Theorem \ref{MainThm}.  
The properties (\ref{thmconYoneda}) and (\ref{thmconh1exact}) come  from Proposition \ref{Yoneda}.  

The properties (\ref{thmconodd}) and  (\ref{thmcontensor}) were established in
Proposition \ref{propodd}.

 Part (\ref{thmcondual}) comes from a more general statement:
 we claim that if $X$ is an odd object of dimension $d$ in a $\Q$-linear pseudo-abelian symmetric tensor category, then there is a canonical isomorphism 
 \[(\Sym^i X)^{\vee} \cong \Sym^{d-i}X \otimes (\Sym^{d}X)^{\vee} \ . \]
Let us prove the claim. In \cite[Lemma 3.2]{OS} it is proven that for any even object
 $X$ of dimension $d$ there is a canonical isomorphism $(\wedge^i X)^{\vee} \cong \wedge^{d-i}X \otimes (\wedge^{d}X)^{\vee}.$ 
 One can change the sign of all symmetries so that one gets a new category which is 
 equivalent as $\Q$-linear tensor category to the previous one (but not as $\Q$-linear symmetric tensor category). This transformation sends even objects of dimension $d$ to odd objects of dimension $d$ so that 
 \cite[Lemma 3.2]{OS} implies our claim.
\end{proof}

\subsection{Weil cohomology of semiabelian varieties}
 Let $H^*$ be any mixed Weil cohomology with coefficients in $F$, in the sense of \cite{CD2} (we do not ask anymore this to be the $\ell$-adic cohomology). 
Recall that, by definition, this means that $F$ is a field of characteristic zero, the K\"unneth formula and excision hold, and that, moreover, the cohomology of the point, the affine line and the multiplicative group have the standard dimensions.

We also require that the cohomology of any scheme is concentrated in non-negative degrees. In \cite[Remark 2.7.15]{CD2} it is conjectured that this should be deduced  from the other axioms (see the comments in loc. cit. after Theorem 2.7.14). We also write \[H^*: \DMegk \longrightarrow \GrVec^{\pm}_{F}\] for the realization functor induced by the Weil cohomology \cite[Thm 3]{CD2}. 
 
\begin{lemma}\label{lemweilartin}
Let $M$ be an Artin motive. Then $H^*(M)$ is concentrated in degree $0$.
\end{lemma}
\begin{proof}
By \cite[Thm 1(4)]{CD2}, the theory $H^*$ satisfies Poincar\'e duality. By assumption it is concentrated in non-negative degrees. Hence $H^*(M(\Spec L))$ is concentrated in degree $0$.
\end{proof}
 \begin{proposition}\label{propweil}
 Let $G$ be a semiabelian variety over $k$ and $\Mi{i}(G)$  be the  K\"unneth components constructed  in Theorem \ref{Thmconsequences}. Then
 \[{H}^*(\Mi{i}(G))={H}^i(G)\]
 for any mixed Weil cohomology $H^*$.
\end{proposition}

\begin{proof}
 We apply $H^*$ to the isomorphism of Hopf algebras
\[\vp_G:M(G)\to\OSym(\Mof(G))\ .\]
Hence $H^*(\Mof(G))$ is the primitive part of $H^*(G)$ and it suffices to show
that it is concentrated in degree $1$. We know that $\Mof(G)$ is odd of dimension $2g+r$ where $r$ is the torus rank of $G$ and $g$ the dimension of the abelian part of $A$. Hence $H^*(\Mof(G))$ is of dimension $2g+r$. Moreover,
\[ \det(H^*(\Mof(G))=H^*(\det(G))=H^*(\Lambda(g+r)[2g+r])\]
with $\Lambda$ as in Theorem \ref{Thmconsequences} an invertible Artin motive.
By Lemma \ref{lemweilartin} the cohomology of $\Lambda$ is concentrated in
degree $0$. Hence the cohomology of
$\det(G)$ is concentrated in degree $2g+r$. 
By assumption $H^*(\Mof(G))$ is odd and $H^*$ is concentrated in non-negative degrees.
This is only possible
if all of $H^*(\Mof(G))$ is concentrated in degree $1$.
\end{proof}

Consider now the category of $1$-motives over $k$ (introduced by Deligne  \cite{DeH3}) and its embedding in $\DMeetkQ$ \cite{BVK,Org}.

\begin{proposition}\label{prop1-mot}
Let $[F\to G]\in\DMeetkQ$  be a $1$-motive, with $F$ (in degree $0$) a $k$-group scheme such that $F_{\bar{k}}$ is a free abelian group of rank $d$  and $G$ (in degree $1$) an extension of an abelian variety of dimension $g$ by a torus of rank $r$. Let $H^*$ be a Weil cohomology theory such that $H^*(X)$ is concentrated in non-negative degrees. Then $H^*([F\to G])$ is concentrated in degree $0$ and of dimension $2g+d+r$.
\end{proposition}
\begin{proof}
This holds for $[F\to 0]$ by Lemma \ref{lemweilartin} and for $[0\to G]$ by Proposition \ref{propweil} (compare with the Definition \ref{defnalpha} and note the shift on the degree here). By considering the long exact cohomology sequence, it follows for $[F\to G]$. 
\end{proof}

\begin{remark}
To our knowledge this result is new. There is a discussion of realizations of $1$-motives in \cite[Section 15.4]{BVK}  which is going into a different direction.
\end{remark}

\subsection{Weight filtration}\label{sec_weights}
Recall that Bondarko defines in \cite[Def 1.1.1]{Bon} categories
$_{\leq a}\DMgm$ of motives of weight at most $a$ and $_{\geq a}\DMgm$ of motives of weight at least $a$ such that
\[ _{\leq a}\DMgm\cap _{\geq a}\DMgm =\ChowkQ[-a]\ .\]

Our structural knowledge also gives us a weight filtration in the sense of Bondarko on
$M(G)$. Let $G$ be semiabelian with torus part $T$ of rank $r$ and abelian part $A$ of dimension $g$. Recall that there is a natural exact triangle
\[ \Mof(T)\to\Mof(G)\to\Mof(A)\]
with $\Mof(A)$ a Chow motive and $\Mof(T)=\Xi(1)[1]$ (with $\Xi$ the Artin motive of Proposition \ref{toruscase}) a Chow motive shifted by $[-1]$. This means
that $\Mof(A)$ is pure of weight $0$ and $\Mof(T)$ is pure of weight $1$. Hence $\Mof(G)$ has weights between $0$ and $1$ and the above sequence is a weight decomposition. Using the filtration of Appendix \ref{appfilt} we extend this to all K\"unneth components.

\begin{proposition}\label{weight}
Fix an integer $i$.
For every choice of $-\infty \leq a\leq b\leq \infty$ in $\Z\cup\{-\infty,\infty\}$ there is a functor
\[_{a\leq w\leq b}\Mi{i}:\sAbQ\to _{\leq b}\DMgm\cap _{\geq a}\DMgm\]
together with an exact triangle of functors for every choice of $a\leq b<c$
\[ _{b+1\leq w\leq c}\Mi{i}\to _{a\leq w\leq c}\Mi{i}\to _{a\leq w\leq b}\Mi{i}\]
such that
\[ _{a\leq w\leq b}\Mi{i}=\Mi{i}\hspace{2ex}\text{for $a\leq i$ and $b\geq 0$.} \]
Moreover, for every semiabelian variety $G$ with torus part $T$ of rank $r$  and abelian part $A$ of dimension $g$, we have naturally
\[ _{a\leq w\leq a}\Mi{i}(G)=\Mi{a}(T)\tensor \Mi{i-a}(A)=\Sym^a\Mof(T)\tensor\Sym^{i-a}\Mof(A)\ .\]
The weights of $\Mi{i}(G)$ are concentrated in the range
\[ \Mi{i}(G)= _{\max(0,i-2g)\leq w\leq \min(i,r)}\Mi{i}(G)\ .\]
\end{proposition}
\begin{remark}
Suppose that $G=T \times A$, then
\[\Mi{i}(G)= \Sym^i(\Mi{1}(T)\oplus\Mi{1}(A))= \bigoplus_{j=0} \Mi{j}(T)\otimes \Mi{i-j}(A)\]
and the definitions
\[ _{a \leq w \leq b}\Mi{i}(G)= \bigoplus_{j=a}^{b}  \Mi{j}(T)\otimes \Mi{i-j}(A)\]
verify all the properties of the statement. In general $G$ will just be an extension of $A$ by $T$ and one needs to replace this description by a filtered one, using Proposition \ref{qniisomo2}.
\end{remark}

\begin{proof}[Proof of Proposition \ref{weight}.]
By Section \ref{reduce perfect}, we may assume that $k$ is perfect.
Recall that $\Mof(G)=q\underline{G}$ where $q:D^-(\STEkQ) \to \DMeetkQ$
is the localization functor and $\underline{G}\in\STEkQ$ is a sheaf.
We put
\[ _{a\leq w\leq b}\Mi{i}(G)=
q\left( \Fil_{a}^{\underline{T}}\Sym^{i}\underline{G}/\Fil_{b+1}^{\underline{T}}\Sym^{i}\underline{G}\right)\]
with $\Fil^{\underline{T}}$ as defined in Definition \ref{symninot}. 
The exact triangles are immediate from this construction. 
The computation of $ _{a\leq w\leq a}\Mi{i}(G)$ follows from
Proposition \ref{qniisomo2}. In particular, it is pure of weight $a$ and geometric. All other statements follow from this case by induction on $|b-a|$. 
\end{proof}

\begin{corollary}
Suppose we are in one of the following two cases:
\begin{enumerate}
\item $k$ is embedded into $\C$ and $H^*$ is the Betti realization with its
natural mixed Hodge structure;
\item $k$ is a field of finite type over its prime field, $\ell$ a prime different from the characteristic of $k$  and $H^*$ the $\ell$-adic cohomology;
\end{enumerate}
In both cases let $(W_nH^*)_{n\in\Z}$ be the weight filtration. Then
\[ H^*\left(_{-\infty\leq w\leq a}\Mi{i}(G)\right)=H^i\left( _{-\infty\leq w\leq a}\Mi{i}(G)\right)=
   W_{i+a}H^i(G)\ .\]
\end{corollary}
\begin{proof}
Let again be $T$ the torus part of $G$ and $A$ the abelian part.
 Note that by Proposition \ref{propweil}
\[ H^*(\Mi{a}(T)\tensor\Mi{i-a}(A))=H^a(T)\tensor H^{i-a}(A)\]
is concentrated in degree $i$.
As $A$ is smooth projective, $H^{i-a}(A)$ is pure of weight $i-a$. 
On the other hand, $H^a(T)$ is pure of weight $2a$. Hence
$H^*( _{a\leq w\leq a}\Mi{i}(G))$ is concentrated in degree $i$ and pure of
weight $i+a$.

The corollary follows by induction. 
\end{proof}
\begin{remark}
In general one does not have to expect that the weight filtration of a cohomology and its graded components lift \emph{canonically} in $\DMegkQ$. This has been studied in \cite{WildChow} and a sufficient criterion,  called of \emph{avoiding weights}, has been given. This criterion does not apply here, but our situation allows anyway to define such a canonical lifting.
\end{remark}

\subsection{General Commutative Group Schemes}\label{sec_cgroup}
Let $k$ be a field and $G/k$ an arbitrary commutative group
scheme of finite type over $k$.
Our aim is to extend the previous results to this case.

Let us initially consider $k$ to be perfect. 
Let $G^0$ be the connected component of the neutral element and $\pi_0(G)$ the group of connected components of $G$. The latter is finite. We have
a natural short exact sequence
\[ 1\to G^0\to G\to \pi_0(G)\to 1\ .\]

We are going to express the motive of $G$ in terms of $G^0$ and $\pi_0(G)$.
Note that we always have
\[ M(G)=M(G^\red)\]
where $G^\red$ is the maximal reduced subscheme of $G$. 
As $k$ is perfect, this scheme is a group and it is, moreover, smooth.
Recall that by 
 the Barsotti-Chevalley structure theorem  (\cite{Bars} and \cite{Chev}), there is a short exact sequence
\[ 1\to G^u\to G^\red\to G^{sa}\to 1\]
with $G^u$ unipotent and $G^{sa}$ semiabelian. 

\begin{remark}Let $F$ be a finite group scheme over $k$. Then
the Artin motive $M(F)$ has the explicit description as representation of the absolute Galois group:
\[ M(F)\isom\Q[F(\bar{k})]\]
with operation of $\Gal(\bar{k}/k)$ given by the operation on $\bar{k}$-valued
points. The Hopf algebra structure is the one of the group ring. Note
that we have naturally $F(\bar{k})=F^\red(\bar{k})$.
\end{remark}

\begin{definition}[cf. Definition \ref{defnulG}]Let $G$ be a commutative group scheme over a perfect field $k$.
Define
$\ul{G}$ to be the \'etale sheaf on $\Sm/k$ given by
\[ \ul{G}(S)=\Mor_{\Sch/k}(S,G)\]
and $\ul{G}_\Q=\ul{G}\tensor_\Z\Q$.
\end{definition}

\begin{lemma}\label{group_reduced}Let $G$ be a commutative group scheme. Then $\ul{G}$ is
an \'etale sheaf with transfers and we have
\[ \ul{G}=\ul{G}^\red=\ul{G}^0=(\ul{G}^0)^\red\ .\]
\end{lemma}
\begin{proof}Note that we consider $\ul{G}$ as a sheaf only on $\Sm$. For
$S$ smooth, we have
\[ \Mor_{\Sch/k}(S,G)=\Mor_{\Sch/k}(S,G^\red)\ .\]
Moreover, $G^\red$ is smooth, hence the sheaf has transfers by  Lemma \ref{lemwelldef}. By isogeny invariance (Proposition \ref{exactMof}) the finite group of
components does not contribute.
\end{proof}
\begin{definition}[cf. Definition \ref{defnalpha}]
Let $G$ be a commutative group scheme over $k$. Consider $\ul{G}_\Q \in D^-\STEkQ$ as an object concentrated in degree zero and define 
\[ \Mof(G)\in \DMeetkQ \]
as the image of $\ul{G}_\Q$ by the functor $q$ (see Section \ref{sect_motives}).
Similarly, consider
\[\gamma_{G^\red}:\Cor(\cdot,G^\red)\otimes_\Z\Q\to \ul{G^\red}_\Q\ ,\]
of Proposition \ref{propspsz}, as a map in $D^-\STEkQ$, and define
\[ \alpha_G:M(G)=M(G^\red)\to \Mof(G^\red)=\Mof(G)\]
as the image of $\gamma_{G^\red}$ by the functor $q$.
Moreover, let
\[\vp_G:M(G)\to \OSym(\Mof(G))\]
be the unique extension of $\alpha_G$ compatible with comultiplication.
\end{definition}
If $G$ is semiabelian, then this definition agrees with the old one.
\begin{lemma}\label{unip}Let $k$ be perfect, $G$ a smooth connected commutative groups scheme over $k$. Let $G^{sa}$ be its semi-abelian part.
Then
\[ M(G)\to M(G^{sa}),\hspace{2ex}\Mof(G)\to \Mof(G^{sa})\]
are isomorphisms. In particular, $M(G)$ and $\Mof(G)$ vanish if
$G$ is unipotent.
\end{lemma}
\begin{proof}
It suffices to consider the case where $k$ is algebraically closed. We
consider the sequence
 \[ 1\to G^u\to G\to G^{sa}\to 1\]
with $G^u$ unipotent and $G^{sa}$ semiabelian. As $\Mof$ is exact, we need
to show that $\Mof(G^u)$ vanishes.
The unipotent group $G$ has a filtration with associated graded components isomorphic to $\Ga$. 
As $\Mof$ is exact, it suffices to consider $G=\Ga$. By definition
\[ \ul{\Ga}=\sheaf{O}\]
is the structure sheaf. We take its image under 
$\STEkQ\xrightarrow{q}\DMeetkQ\xrightarrow{i}D^-(\STEkQ)$ and get
the Suslin complex 
\[ iq\sheaf{O}=\ul{C}_*(\sheaf{O})=\sheaf{O}\tensor_k \sheaf{O}(\Delta^*)\]
where $\Delta^*$ is the standard cosimplicial object with $\Delta^n$ the
algebraic $n$-simplex. It suffices to show that $\sheaf{O}(\Delta^*)$ is exact.
This is the case $q=0$ of \cite[Prop. 1.1]{BG}  Note that they assume $\chr(k)=0$, but the assumption is not used at this point.

We now turn to $M(G)$. We view $G\to G^{sa}$ as a $G^u$-bundle. Note that $G^u$ is $\A^1$-homotopy
equivalent to a point. 
The surjectivity of the map of \'etale sheaves $\ul{G}\to\ul{G}^{sa}$ means that
$G^{sa}\to G$ \'etale locally has a section. This implies that the
bundle is \'etale locally trivial. Hence the $M(G)\to M(G^{sa})$ is
an isomorphism.
\end{proof}

\begin{theorem}\label{propcommgroup}
Let $G$ be a commutative group scheme of finite type over a field $k$, which is not necessarily perfect. Let
$\pi:G\to\pi_0(G)$ be the projection to its group of components.
Then the natural map
\[ \psi_G:M(G)\to\OSym(\Mof(G))\tensor M(\pi_0(G))\]
given by the composition
\[ \psi_G:M(G)\xrightarrow{\Delta}M(G)\tensor M(G)\xrightarrow{\vp_G\tensor \pi}
\OSym(\Mof(G))\tensor M(\pi_0(G))\]
is an isomorphism of Hopf algebras and it is natural in $G\in\CGS$.
\end{theorem}
\begin{remark}All other results, e.g., K\"unneth components, computation of
Weil cohomology and weights immediately extend to this case.
\end{remark}
\begin{proof} By Section \ref{reduce perfect} we may assume that $k$ is perfect and the above constructions apply.

By Lemma \ref{group_reduced} we may assume that $G$ is reduced and hence
smooth. By
Proposition \ref{stalk}, we may assume that $k$ is algebraically closed.
By Lemma \ref{groupsplit} we may assume that $G=G^0\times F$ with $F$ finite. The morphism
$\psi_G$ is a morphism of coalgebras by construction. It is compatible with direct products as in Proposition \ref{product}. Hence it suffices to
consider the cases $G=G^0$ and $G=F$ separately. The latter case is trivial.
From now on let $G=G^0$ and assume also that $G$ is reduced, i.e., smooth.
By Lemma   \ref{unip} we are reduced to the semiabelian case. This is main Theorem~\ref{MainThm}.

Compatiblity with the Hopf object structure follows by the same reductions
from the semiabelian case and the finite case because
$(G\times G)^0=G^0\times G^0$.
\end{proof}
The following statement must be well-known to experts but we have not found a reference.
\begin{lemma}\label{groupsplit}
Let $k$ be algebraically closed and $G$ a smooth commutative group scheme over $k$. Then
\[ G\isom G^0\times F\]
with $G^0$ connected and $F$ finite.
\end{lemma}
\begin{proof}
We claim that
\[ \Ext^1_{\CGS}(F,G^0)\]
vanishes. We compute in the category of \'etale sheaves on $\Sch$.
The finite
group scheme $F$ is constant because $k$ is algebraically closed. Hence
there is a finite resolution
\[ 0\to\Z^n\xrightarrow{M}\Z^n\to F\to 0\]
where $M$ is multiplication by a diagonal matrix. The interesting part of the 
associated long exact sequence
reads
\[ G^0(k)^n\xrightarrow{M} G^0(k)^n\to \Ext^1(F,G^0)\to \Ext^1(\Z^n,G^0)\]
The last group vanishes because $k$ is algebraically closed and
$\Hom(\Z,\cdot)=\Gamma(k,\cdot)$ is exact.
The first map is surjective because the multiplication by $m$ is surjective on the connected group $G^0$.
\end{proof}

\begin{appendix}
\section{Comparison of Chern classes}\label{sectionchern}
In this section we establish a comparison of two possible ways of attaching a cohomology class to a semiabelian variety. This technical result is
one of the key inputs into the proof of our main Theorem \ref{MainThm}.
We keep the notation of the main text, in particular Section \ref{sect_constr}.

\subsection{The comparison result}
Let
\[ 1\to\Gm\to G\to H\to 1\]
be a short exact sequence of semiabelian varieties.
Note that $G$ is a $\Gm$-torsor on $H$, hence it has a cohomology class
\begin{align}\label{c1G}
c_1([G])\in &H^1_\et(H,\ul{\Gm}_\Q)\\
         &\isom\Mor_{D^-(\STEkQ)}(L(H), i(\tatt))\\
	 &=\Mor_{\DMeetkQ}(M(H), \tatt)  \ ,
	 \end{align}
	 where the last equality comes by adjunction. 
On the other hand, the induced exact sequence
\[ 0\to \Mof(\Gm)\to\Mof(G)\to\Mof(H)\to 0\]
is an element
\[ [\Mof(G)]\in\Ext^1_{\STEkQ}(\Mof(H),\one(1)[1])=\Mor_{D^-(\STEkQ)}( \Mof(H), \tatt)\ .\]
By composition with the summation map (see Definition \ref{defnall}) 
\[\alpha_H:M(H)\to \Mof(H)\]
we obtain
\begin{equation}\label{gammaMofG}
\alpha_H^*[\Mof(G)]\in \Mor_{\DMeetkQ}(M(H), \tatt) \ .
\end{equation}
\begin{proposition}\label{propc1}Let 
\[ 1\to \Gm\to G\to H\to 1\]
be a short exact sequence of semiabelian varieties. Then the elements
$c_1([G])$ and $\alpha_H^*[\Mof(G)]$ agree in $\Mor_{\DMeetkQ}(M(G),\tatt)$.
In other words, the diagram
\[\xymatrix{
M(H)\ar[rd]^{c_1([G])]}\ar[d]^\alpha\\
\Mof(H)\ar[r]_{[\Mof(G)]}&\tatt\\
}\]
commutes.
\end{proposition}
\begin{remark} We formulate everything with rational coefficients because this is the way we want to apply the statement. However, the result is already true with integral coefficients.
\end{remark}
\subsection{Proof}
The proof will be given at the end of this appendix. We introduce some
notation. 

\begin{notation}\label{notgroupsheaf}
Let $H$ be a commutative group scheme. As before we write $\ul{H}$ for the
corresponding \'etale sheaf of abelian groups. Moreover, we write
$\ul{H}^\tr$ for the corresponding \'etale sheaf with transfers.

\end{notation}

\begin{lemma}\label{lemdelt}Let $H$ be a commutative group scheme. Then there is a commutative
diagram of $\delta$-functors on $\STEkQ$:
\[\begin{CD}
\Ext^i_{\STEkQ}(\ul{H}^\tr_\Q,\cdot)@>>>\Ext^i_{\SEkQ}(\ul{H}_\Q,\cdot)\\
@V\gamma^*VV@VVY_{\SEkQ}V\\
\Ext^i_{\STEkQ}(L(H),\cdot)@>>Y_{\STEkQ}>H^i_\et(H,\cdot)
\end{CD}\]
with $\gamma:L(H)\to\ul{H}^\tr_\Q$ the summation map of Proposition \ref{propspsz} and
$Y_?$ the Yoneda map in the categories $\STEkQ$ and $\SEkQ$, respectively.
\end{lemma}
\begin{proof}
By universality in $\STEkQ$ and $\SEkQ$, it suffices to check the case $i=0$. Let $\shfA$ be an \'etale sheaf with transfers and
$f:\ul{H}^\tr_\Q\to\shfA$ a morphism in $\STEkQ$. We describe its image in
$\shfA(H)$ via the right-hand side map of the diagram. By forgetting transfers, $f$ gives rise
to a map 
\[ f:\ul{H}_\Q\to\shfA\ .\]
Evaluating on $H$, we get
\[ f(H):\ul{H}_\Q(H)\to\shfA(H)\ .\]
Then $Y_{\SEkQ}(f)$ is defined as $f(H)(\id)$. Going via the left hand side, we have
to evaluate the map in $\STEkQ$
\[ L(H)\xrightarrow{\gamma}\ul{H}^\tr_\Q\xrightarrow{f}\shfA\]
on $H$ and get
\[ \Hom_{\Smck}(H,H)=L(H)(H)\xrightarrow{\gamma(H)}\ul{H}^\tr_\Q(H)\xrightarrow{f(H)}\shfA(H)\ .\]
Then $Y_{\STEkQ}(\gamma^*(f))$ is given by
 $f(H)\gamma(H)(\id)$.
The summation map $\gamma(H)$ maps the identity in $\Smck$ to the
identity in $\ul{H}^\tr_\Q(H)=\ul{H}_\Q(H)$ hence
\[ Y_{\SEkQ}(f)=Y_{\STEkQ}(\gamma^*(f))\ .\]
\end{proof}

\begin{corollary}\label{corc1partial}Let 
\[1\to \Gm\to G\to H\to 1\]
be a short exact sequence of commutative group schemes. Then
\[ Y_{\STEkQ}\gamma^*[\ul{G}^\tr_\Q]=\partial(\id_H)\]
where $\partial$ is the connecting homomorphism
\[ \Gamma(H,\ul{H}_\Q)\to H^1_\et(H,\ul{\Gm}_\Q)\ .\]
\end{corollary}
\begin{proof}
Note that $[\ul{G}^\tr_\Q]=\partial(\id)$ where $\partial$ is the boundary map
\[ \Hom_{\STEkQ}(\ul{H}^\tr_\Q,\ul{H}^\tr_\Q)\to\Ext^1_{\STEkQ}(\ul{H}^\tr_\Q,\ul{\Gm}^\tr_\Q)\]
for the short exact sequence
\[ 0\to \ul{\Gm}^\tr_\Q\to\ul{G}^\tr_\Q\to\ul{H}^\tr_\Q\to 0\ .\]
The statement follows from evaluating
the $\delta$-homomorphism of Lemma \ref{lemdelt}
\[ \Ext^i_{\STEkQ}(\ul{H}^\tr_\Q,\cdot)\to H^i_\et(H,\cdot)\]
on this short exact sequence.
\end{proof}

\begin{proof}[Proof of Proposition \ref{propc1}.]
Let $1\to \Gm\to G\to H\to 1$ be an exact sequence of commutative group schemes.
It suffices to show that
\[ c_1([G])=Y_{\STEkQ}\gamma^*[\Mof(G)]\in H^1_\et(H,\ul{\Gm}_\Q)\ .\]
By Corollary \ref{corc1partial} this is reduced to showing
\[ c_1([G])=\partial(\id_H)\]
where $\partial$ is the connecting homomorphism
\[  \Gamma(H,\ul{H}_\Q)\to H^1_\et(H,\ul{\Gm}_\Q)\]
for the short exact sequence
\[ 0\to\ul{\Gm}_\Q\to\ul{G}_\Q\to\ul{H}_\Q\to 0\ .\]
We use the point of view of \v{C}ech cohomology in order to compute explicitly.
Choose a local trivialization of the $\Gm$-torsor $G$, i.e., a covering
$\{U_i\to H\}_{i\in I}$ and sections $s_i:U_i\to G|_{U_i}$ inducing
isomorphisms
\[ \phi_i:\Gm\times U_i\to G|_{U_i}\ .\]
By definition $c_1[G]$ is given by the cocycle $g_{ij}\in\Gm(U_i\cap U_j)$
where
\[ s_j|_{U_i\cap U_j}=s_i|_{U_i\cap U_j}g_{ij}\ .\]
On the other hand, in order to define $\partial(\id_H)$ we have to choose preimages
of $\id_H|_{U_i}$ in $G(U_i)$. We choose $s_i$. We then have to apply
the boundary map of the \v{C}ech complex of $\ul{G}_\Q$ and get the
same cocycle $g_{ij}$. 
\end{proof}

\section{Universal properties of the symmetric (co)algebra}\label{appsym}
We review the Hopf algebra structure on the symmetric algebra and its opposite. 
Throughout the appendix let $\cata$ be a $\Q$-linear pseudo-abelian symmetric 
tensor category with unit object $\one$. 
These considerations are going to be applied to the triangulated category $\DMeetkQ$.

\subsection{The symmetric (co)algebra}

Let $ V$ be an object
of $\cata$. 
We denote
\[  T( V)=\bigoplus_{i=0}^\infty V^{\tensor i}\]
the tensor algebra with multiplication 
\[ \mu:T( V)\tensor T( V)\to T( V)\]
given by the tensor product. 

We denote by
\[ \Sym^n V\] 
the image of the projector $\frac{1}{n!}\sum_{\sigma\in \mathcal{S}_n}\sigma$. Let
\[ \iota^n:\Symn V\to V^{\tensor n},\hspace{3ex}\pi^n: V^{\tensor n}\to\Sym^n V\]
be the natural inclusion and projection. For $n=0$ we put
$\Sym^0(V)=\one$.  The projection 
\[\pi^{n+m}: V^{\tensor n+m}\to\Sym^{n+m}( V)\]
factors uniquely via $\pi^{n}\tensor \pi^m$ and induces
\[\pi_{n+m}^{n,m}: \Sym^n( V)\tensor\Sym^m (V)\to  \Sym^{n+m}( V).\]
The inclusion 
\[ \iota^{n+m}:\Sym^{n+m}( V)\to V^{\tensor n}\tensor V^{\tensor m}\]
factors uniquely via $\iota^n\tensor\iota^m$ and induces
\[\iota_{n+m}^{n,m}:\Sym^{n+m}( V)\to\Sym^n( V)\tensor\Sym^m( V)\ .\]

\begin{definition}
Assume that $V$ is odd finite-dimensional in the sense of Section \ref{notsym}, i.e., there is $N$ such that $\Sym^N(V)=0$.

The {\em symmetric algebra} on $ V$ is given by
\[\Sym( V)=\bigoplus_{n\geq 0}\Sym^n V\]
with multiplication 
\[\mu:\Sym^n( V)\tensor\Sym^m( V)\to\Sym^{n+m}( V)\]
given by $\pi^{n+m}_{n,m}$.

The {\em symmetric coalgebra}  on $ V$ is
given by 
\[\OSym( V)=\prod_{n\geq 0}\Sym^n V\]
with comultiplication 
\[\Delta:\Sym^{n+m} V\to \Sym^n( V)\tensor\Sym^m( V)\]
given by $\iota_{n+m}^{n,m}$.
\end{definition}

\begin{remark}The finiteness assumption ensures that all direct sums and products are finite. If we drop the assumption,  the definition of the algebra $\Sym(V)$ needs existence of the direct sum and that infinite direct sums commute with $\tensor$. 
The definition of the coalgebra $\OSym(V)$ needs existence of infinite products and that they
commute with $\tensor$. The latter is not satisfied in standard abelian categories like modules. Alternatively, one may work in the ind-category and the pro-category, respectively. For simplicity, we make the finiteness assumption. This is enough for our application.
\end{remark}

\subsection{Universal properties}

Note that $T( V)\to\Sym (V)$ is a morphism of algebras.
\begin{lemma}\label{universalalgebra}
Let $V$ be odd finite-dimensional.
Let $A$ be a unital algebra object in $\cata$ and let $\alpha: V\to A$ a morphism. Then
there is a unique morphism of unital algebras
\[ T( V)\to A\]
extending $\alpha$.
If $A$ is commutative, the map factors through a unique map of algebras
\[ \Sym( V)\to A\ .\]
Let $B$ be an augmented cocommutative coalgebra object in $\cata$ and $\alpha:B\to V$ a morphism. Then there is a unique morphism of augmented coalgebras
\[B\to\OSym( V)\ .\]
\end{lemma}
\begin{proof}The argument is the same as for vector spaces, where it is well-known.\end{proof}
We apply this principle to the diagonal map
\[ \Delta: V\to V\tensor\one\oplus \one\tensor V\subset T( V)\tensor T( V)\]
and obtain a {\em comultiplication}
\[\Delta: T( V)\to T( V)\tensor T( V)\ .\]
It turns $T( V)$ into a bialgebra. The same argument also turns $\Sym( V)$ into a bialgebra and
\[ T( V)\to \Sym( V)\]
is a morphism of bialgebras. Dually, the summation map
\[\OSym( V)\tensor\OSym( V)\to V\tensor\one\oplus\one\tensor V\xrightarrow{+} V\]
gives rise to a multiplication on $\OSym( V)$ making it a bialgebra.
Finally, multiplication by $-1$ defines
a map
\[  V\xrightarrow{-1}  V\]
which induces an {\em antipodal map} on $\Sym( V)$  and on $\OSym( V)$. It turns
$\Sym( V)$ and $\OSym( V)$ into  {\em Hopf algebras}.

\begin{remark}If $\cata$ is the category of $\Q$-vector spaces and $ V$ a finite-dimensional vector space, then $\Sym( V)$ is a polynomial ring, whereas $\OSym( V)$ is the algebra of distributions. The two are isomorphic but not equal.
\end{remark}

\begin{lemma}\label{lemcritbialgebra}
Let $V$ be odd finite-dimensional.
\begin{enumerate}
\item
Let $A$ be a unital  commutative bialgebra  in $\cata$ and $\alpha: V\to A$ a morphism such that
\[\begin{CD}
 V@>>>A\\
@V\Delta VV@VV\Delta V\\
 V\tensor \one\oplus\one\tensor V@>>>A\tensor A
\end{CD}\]
commutes.
Then  the universal morphism $\Sym(V)\to A$ is a morphism of bialgebras.
\item 
Let $B$ be an augmented  cocommutative bialgebra  in $\cata$ and $\beta: B\to V$ a morphism such that
\[\begin{CD}
 B\tensor B@>>>V\tensor \one\oplus\one\tensor V\\
@V\mu VV@VV+ V\\
 B@>>>V
\end{CD}\]
commutes.
Then  the universal morphism $\OSym(V)\to B$ is a morphism of bialgebras.
\end{enumerate}
\end{lemma}
\begin{proof}This is an assertion about algebra morphisms. It follows from
the universal property of $\Sym( V)$. The second part follows by the analogous argument.
\end{proof}

\begin{lemma}\label{prodexplicit}
\begin{enumerate}
\item
The component
\[ \Delta_{n+m}^{n,m}:\Sym^{n+m} V\to\Sym^n V\tensor\Sym^m V\]
of comultiplication on $\Sym( V)$ is equal to
\[ \Delta_{n+m}^{n,m}={n+m \choose n}\iota_{n+m}^{n,m}\ .\]
\item The component of multiplication on $\OSym(V)$
\[\mu_{n,m}^{n+m}:\Sym^n( V)\tensor\Sym^m( V)\to\Sym^{n+m}( V)\]
is equal to
\[ \mu_{n,m}^{n+m}={n+ m \choose n}\pi_{n,m}^{n+m}\ .\]
\item Let $V$ be odd finite-dimensional.
The universal map of bialgebras
\[\Sym( V)\to\OSym( V)\]
induced from $\Sym(V)\to  V$ is given by multiplication
by $n!$ in degree $n$. It is an isomorphism. 
\end{enumerate}
\end{lemma}
In particular, the two bialgebras are {\em not} identical.
\begin{proof}
The statement on $\Delta$ is elementary from the definitions.
We now consider the map
\[\Sym( V)\to\OSym( V)\]
given by multiplication by $n!$ in degree $n$. We see from the explicit formula that it is compatible with comultiplication, i.e., it is the canonical one.
It satisfies the criterion Lemma \ref{lemcritbialgebra}, hence it is an isomorphism of bialgebras.
The formula for multiplication on $\OSym( V)$ follows from this isomorphism.
\end{proof}

\begin{corollary}\label{universalboth}
Assume that $V$ is odd finite-dimensional. Then
the bialgebra $\Sym(V)$ has the universal properties
of Lemma \ref{universalalgebra} and Lemma \ref{lemcritbialgebra} with respect to comultiplication. The bialgebra $\OSym(V)$ has the universal properties
of Lemma \ref{universalalgebra} and Lemma \ref{lemcritbialgebra} with respect to multiplication. 
\end{corollary}

\subsection{Direct sums}
\begin{lemma}\label{opsymsum}
Let $V=U\oplus W$ in $\cata$ with $U$ and $W$ odd finite-dimensional.
The natural map
\[ \OSym(U)\tensor \OSym(W)\to U\tensor \one\oplus\one\tensor W\isom U\oplus W\]
gives rise to an isomorphism of bialgebras
\[\OSym(U)\tensor\OSym(W)\to\OSym(U\oplus W)\]
with inverse given by
\[ \OSym(U\oplus W)\xrightarrow{ \Delta}\OSym(U\oplus W)\tensor\OSym(U\oplus W)\to \OSym(U)\tensor\OSym(W)\ .\]
The analogous statement for $\Sym(U\oplus W)$ holds as well.
\end{lemma}
\begin{proof}The isomorphism is well-known for vector spaces. The case
of an additive category is a special case of \cite[Propostion 1.8]{deltens}. 
By construction, the isomorphism is the one compatible with the inclusion
into the tensor algebra, i.e., the one for the symmetric coalgebra.
The case of the symmetric algebra follows because the map is a rational multiple of the one for the coalgebra.\end{proof}

\begin{remark}The isomorphism 
\[\Sym^N(U\oplus W)\to\bigoplus_{n+m=N}\Sym^n(U)\tensor\Sym^m( V)\]
via the symmetric coalgebra
is {\em not} the same as the one defined on the symmetric algebra!
\end{remark}

\section{Filtrations on the graded symmetric (co)algebra}\label{appfilt}
The aim of this appendix is to establish a certain exact triangle for the symmetric coalgebra, see Theorem \ref{corcup}. The basic construction behind it
already appears in a paper by Deligne \cite{deltens}, more precisely in the Proof of Proposition 1.19 of loc. cit.
We wanted to understand the details of the argument and, in particular, keep precise control of the morphisms and the Hopf algebra structure. Hence we decided to give the argument in full detail.

Throughout the appendix, let $\cata$ be a $\Q$-linear abelian symmetric 
tensor category with unit object $\one$. We assume that $\tensor$ is exact.  
These considerations are going to be applied to the abelian category $\STEkQ$ of etale sheaves of $\Q$-vector spaces with transfers. This is possible by 
Remark \ref{tensorexact}.

\subsection{The graded symmetric (co)algebra}
\begin{definition}Let $V$ be an object of $\cata$. We denote by
\[\Sym^*(V)=\bigoplus_{n\geq 0}\Sym^n(V)\]
the {\em graded symmetric algebra} with $\Sym^n(V)$ in degree $n$.
It is a graded Hopf algebra with structure morphisms as in Appendix \ref{appsym}.

We denote
\[\OSym^*(V)=\bigoplus_{n\geq 0}\Sym^n(V)\ \]
the {\em graded symmetric coalgebra} with $\Sym^n(V)$ in degree $n$. It
is a graded Hopf algebra with structure morphisms as in Appendix \ref{appsym}.
\end{definition}
The explicit formulae are given in Lemma \ref{prodexplicit}.
\begin{remark}Note that we are considering {\em graded objects}.
In contrast to Appendix \ref{appsym} we do {\em not} assume that
$V$ is odd finite-dimensional. Neither do we need existence of infinite
direct sums. The notation $\bigoplus_{n\geq 0}$ is just convention. A graded
object really is given by a sequence of objects. 
Switching to graded objects and graded morphisms allows all considerations of Appendix \ref{appsym} {\em without} finiteness assumptions of any kind.
\end{remark}

\begin{lemma}\label{gradeduniversal}
The universal properties of Lemma \ref{universalalgebra} and Lemma \ref{lemcritbialgebra} are satisfied for $\Sym^*(V)$ and $\OSym^*(V)$ with respect to \emph{graded} morphisms. 

The identity on $V$ induces an isomorphism of graded Hopf algebras
\[ \Sym^*(V)\to \OSym^*(V)\ .\]
In degree $n$ it is given by multiplication by $n!$.
\end{lemma}
\begin{proof}Same arguments as in Appendix \ref{appsym}. Note that the finiteness assumptions are not needed because all morphisms respect the grading.
\end{proof}

\subsection{The filtration on the graded tensor algebra}\label{sectfilttensor}

\begin{definition}Let $U\subset V$ be a subobject in $\cata$. We define
a descending filtration on $ V$ by
\[\Fil^U_{i} V=\begin{cases} V&i=0,\\
                               U&i=1,\\
			       0&i>1.
	          \end{cases}\]
For $n\geq 0$ let 
\[ \Fil^U_i V^{\tensor n}\subset  V^{\tensor n}\]
the product filtration on $ V^{\tensor n}$ and
\[ \Gr^U_i V^{\tensor n}=\Fil^U_i V^{\tensor n}/\Fil^U_{i+1} V^{\tensor n}\]
the associated graded objec object.
\end{definition}
This means that an elementary tensor
is in $i$th step of the filtration step, if at least $i$ factors are in $U$.
\begin{lemma}Let $U\subset V$ be a subobject in $\cata$. 
The filtration $\Fil^U_iT( V)$ turns the graded tensor algebra into a filtered graded bialgebra, i.e.,
\begin{gather*}
\mu:\Fil_i^U V^{\tensor n}\tensor \Fil_j^U V^{\tensor m}\to\Fil_{i+j}^U V^{\tensor n+m}\ ,\\
\Delta: \Fil_{I}^U V^{\tensor N}\to\sum_{i+i'=I,n+n'=N}\Fil^U_i V^{\tensor n}\tensor\Fil^U_{i'} V^{\tensor n'}\ .
\end{gather*}
In particular, $\Gr^U_\bullet T^*( V)$ is a bigraded bialgebra.
\end{lemma}
\begin{proof}
The statement on multiplication is obvious. The statement on comultiplication is reduced by the universal property to the basic case $N=1$.
\end{proof}

The aim of this section is to compute the associated graded bialgebra.
\begin{proposition}\label{propfilt3}
Let
\[ 0\to U\to  V\to  W\to 0\]
be a short exact sequence in $\cata$. Then there is a canonical isomorphism of
bigraded bialgebras
\[ \Gr^U_\bullet T^*( V)\isom \Gr^U_\bullet T^*(U\oplus  W)=T^\bullet(U)\tensor T^{*-\bullet}(W)\ .\]
\end{proposition}
The proof will be given at the end of the section.
\begin{remark} Whenever the sequence has a splitting (e.g. in the case when $\cata$ is a category of vector spaces), this formula is easy to check. 
The purpose of the following arguments is to verify that they work for more general $\cata$. 
\end{remark}

\begin{lemma}\label{lemtensorintersect} Let  $A'\subset A$
and $B'\subset B$ be subobjects in $\cata$. Then 
\[ (A'\tensor B)\cap (A\tensor B')=A'\tensor B'\]
where the intersection is taken in $A\tensor B$.
\end{lemma}
\begin{proof}Consider the monomorphism of short exact sequences
\[\xymatrix{
0\ar[r]& A\tensor B'\ar[r]& A\tensor B\ar[r]&A\tensor (B/B')\ar[r]& 0\\
0\ar[r]&A'\tensor B'\ar[r]\ar@{^{(}->}[u]&A'\tensor B\ar@{^{(}->}[u]\ar[r]& A'\tensor (B/B')\ar[r]\ar@{^{(}->}[u]& 0
}\]
The assertion follows from diagram chasing.
\end{proof}

\begin{lemma}\label{lemfilt1}
Let $i\geq 0$ and $n\geq 1$ be  integers, and let $U \subset V$ be a subobject in $\cata$. 
Then the sequence
\[ 0\to U\tensor\Fil^U_i V^{\tensor (n-1)}\to 
    \left(U\tensor\Fil^{U}_{i-1} V^{\tensor (n-1)}\right)\oplus\left( V\tensor\Fil^{U}_i V^{\tensor (n-1)}\right)
  \to \Fil^U_i V^{\tensor n}\to 0\]
  is exact.
\end{lemma}
\begin{proof}
Obviously,
\[ (U\tensor\Fil^{U}_{i-1} V^{\tensor n-1})+( V\tensor\Fil^{U}_{i} V^{\tensor n-1})=\Fil^U_i V^{\tensor n}\ .\]
It remains to check that
\[  (U\tensor\Fil^{U}_{i-1} V^{\tensor n-1})\cap( V\tensor\Fil^{U}_{i} V^{\tensor n-1})=U\tensor\Fil^U_i V^{\tensor n-1}\ .\]
This is true by Lemma \ref{lemtensorintersect}.
\end{proof}

\begin{lemma}\label{lemfilt2}Let $n\geq 1, i\geq 0$ and
\[ 0\to U\to  V\to  W\to 0\]
be a short exact sequence in $\cata$. Then there is a natural isomorphism
\[ \Gr^U_i V^{\tensor n}\to \Gr^U_i(U\oplus W)^{\tensor n}\ .\]
\end{lemma}
\begin{remark}
The object $\Gr^{U}_i(U\oplus W)$ is a direct sum of tensor products
of $i$ copies of $U$ and $n-i$ copies of $ W$, running through all
possible choices. E.g.
\[ \Gr^{U}_1(U\oplus W)^{\tensor 3}=(U\tensor W\tensor  W)\oplus ( W\tensor U\tensor W)\oplus( W\tensor W\tensor U)\ .\]
\end{remark}
\begin{proof}[Proof of Lemma \ref{lemfilt2}.]
We argue by induction on $n$. The case $n=1$ holds by definition. For $n>0$ we
consider the commutative diagram of short exact sequences

\[\xymatrix{
0\ar[r]&U\tensor\Fil^U_{i+1} V^{\tensor n}\ar[r]\ar@{_{(}->}[d]& 
    \left(U\tensor\Fil^U_i V^{\tensor n}\right)\oplus\left( V\tensor\Fil^U_{i+1} V^{\tensor n}\right)\ar[r]\ar@{_{(}->}[d]&
  \Fil^U_{i+1} V^{\tensor n+1}\ar[r]\ar@{_{(}->}[d]&0
\\
 0\ar[r]&U\tensor\Fil^U_i V^{\tensor n}\ar[r]&
    \left(U\tensor\Fil^U_{i-1} V^{\tensor n}\right)\oplus\left( V\tensor\Fil^U_i V^{\tensor n}\right)
  \ar[r]& \Fil^U_i V^{\tensor n+1}\ar[r]&0
}\]
By the snake lemma we get a short exact sequence of cokernels. By induction it reads
 \[0\to U\tensor\Gr^U_i(U\oplus W)^{\tensor n}\to 
 U\tensor\Gr^U_{i-1}(U\oplus W)^{\tensor n}\oplus  V\tensor \Gr^U_i(U\oplus W)^{\tensor n}
 \to \Gr^U_i V^{\tensor n+1}   \to 0\ .\]
Note that the map
\[ U\tensor\Gr^U_i(U\oplus W)^{\tensor n}\to 
 U\tensor\Gr^U_{i-1}(U\oplus W)^{\tensor n}\]
 vanishes, hence
 \[ \Gr^U_i V^{\tensor n+1} \isom U\tensor\Gr^U_{i-1}(U\oplus W)^{\tensor n}\oplus W\tensor\Gr^U_{i-1}(U\oplus W)^{\tensor n}
 =\Gr^U_i(U\oplus W)^{\tensor n+1}\ .\]
\end{proof}
\begin{proof}[Proof of Proposition \ref{propfilt3}.]
We apply Lemma \ref{lemfilt2} for all $i$ and $n$. The compatibility
with multiplication and comultiplication follows from the 
construction or more abstractly from naturality. 
\end{proof}

\subsection{The filtration on the graded symmetric (co)algebra}
\begin{definition}\label{symninot}
Let $U\subset  V$ be a subobject in $\cata$ and $n\geq 0$. We define
a descending filtration on $\Symn( V)$ by
\[
\Fil^U_{i}\Symn( V)=
\pi^{n}\left(\Fil^U_i V^{\tensor n}\right)
\]
and
\[ \Gr^U_i\Symn( V)=\Fil^U_i\Symn( V)/\Fil^U_{i+1}\Symn( V)\]
the associated graded object.
\end{definition}
\begin{remark}We have the simpler presentation
\[ \Fil^U_i\Symn( V)=\pi^n\left(U^{\tensor i}\tensor V^{\tensor n-i}\right)\ .\]
\end{remark}

\begin{lemma}Let $U\subset V$ be a subobject in $\cata$. 
The filtration 
\[ \Fil^U_i\OSym^*( V)=\bigoplus_{n\geq 0}\Fil^U_i\Symn(V)\]
 turns the graded symmetric coalgebra 
$\OSym^*(V)$ into a filtered graded bialgebra.
In particular, $\Gr^U_\bullet\OSym^*( V)$ is a bigraded bialgebra.

The analogous statement for the symmetric algebra is also true.
\end{lemma}
\begin{proof}
The assertion for the symmetric algebra follows from the case of the tensor algebra. The assertion for the symmetric coalgebra follows because
multiplication and comultiplication on $\OSym^*( V)$ are degreewise rational multiples of
multiplication and comultiplication on $\Sym^*( V)$.
\end{proof}


\begin{proposition}\label{qniisomo2}
Let 
\[ 0\to U\to  V\to  W\to 0\]
be a short exact sequence in $\cata$. Then there is a natural isomorphism
of bigraded bialgebras
\[ \Gr^U_\bullet\OSym^*( V)\isom \OSym^\bullet(U)\tensor\OSym^{*-\bullet}( W)\ .\]
In particular for $0\leq i\leq n$, there are natural short exact sequences
\[ 0\to\Fil^U_{i+1}\Sym^n( V)\to\Fil^U_i\Symn( V)\xrightarrow{g^{n,i}}\Sym^i(U)\tensor\Sym^{n-i} W\to 0\ .\]
The map $g^{n,i}$ is obtained by factoring
\[ U^{\tensor i}\tensor V^{\tensor n-i}\to U^{\tensor i}\tensor W^{\tensor n-i}\xrightarrow{\mu_{\OSym}} \Symi(U)\tensor\Sym^{n-i}( W)
\]
uniquely through $\Fil^U_i\Symn( V)$.

The analogous statement for $\Sym^*(V)$ is also true.
\end{proposition}
\begin{proof}
We first consider the case of the symmetric algebra.
As a projector, the symmetrization map $\pi^n$ preserves short exact sequences. Hence Proposition \ref{propfilt3} implies
\[ \Gr^U_\bullet\Sym^*( V)\isom\Gr^U_\bullet\Sym^*(U\oplus W)\isom\Sym^\bullet(U)\tensor\Sym^{*-\bullet}( W)\ .\]
Specializing to $\Gr^U_i\Sym^n( V)$ provides the required short exact sequence.
It remains to verify the explicit description of $g^{n,i}$.
Everything is determined on the level of the graded tensor algebra, where the map
to the associated graded object is induced from the projection $ V\to W$ in the
appropriate factors.
As
$\Fil^U_i\Sym^n( V)$ is the image of 
$U^{\tensor i}\tensor V^{n-i}\subset\Fil^U_i V^{\tensor n}$, it suffices
to describe the map on this object.

In the case of the symmetric coalgebra everything agrees up to rational factors.
Exactness of the sequence follows from the first case.
\end{proof}
\begin{remark}This agrees with the map denoted $g^{n,i}$ in second author's thesis \cite{ward} (see  loc. cit.  Notation 5.3.10). 
\end{remark}

\subsection{Cup-product}
Let $ W$ be an object of $\cata$. The coalgebra structure on $\OSym^*( W)$
allows to define cup-products.

\begin{definition}Let $ W$ be an object of $\cata$, $c: W\to K$ a morphism
in the derived category $D^b(\cata)$. Then we define
\[ \cdot\cup c:\OSym^*( W)\to\OSym^{*-1}( W)\tensor K\]
as the composition
\begin{align*}
 \OSym^*( W)&\xrightarrow{\Delta}\OSym^*( W)\tensor\OSym^*( W)\\
           &\to \OSym^*( W) \tensor  W\\
	   &\xrightarrow{\id\tensor c}\OSym^*( W) \tensor K\ .
\end{align*}
\end{definition}
We apply this to the morphism $[ V]: W\to U[1]$ in $D^n(\cata)$ represented by 
a short exact sequence
\[ 0\to U\to V\to W\to 0\]
in $\cata$. 

\begin{proposition}\label{propcup}Let
\[0\to U\to V\to W\to 0\]
be a short exact sequence in $\cata$. Then
\[ \cdot\cup[ V]:\OSym^*( W)\to\OSym^{*-1}( W)\tensor U[1]\]
is equal to the extension class
\[ [\OSym^*( V)/\Fil^U_2\OSym^*( V)]\]
under the identifications of Proposition \ref{qniisomo2} $\Gr^U_0\OSym^*( V)=\OSym^*( W)$ and
$\Gr^U_1( V)\isom\OSym^*( W)\tensor U$.
\end{proposition}
\begin{remark}This corresponds to the crucial computation \cite[Lemma 5.3.13]{ward} .
\end{remark}
\begin{proof}[Proof of Proposition \ref{propcup}.]
We view the morphisms in $D(\cata)$ as Yoneda extensions.
We need to identify the extension class $[\Sym^n( V)/\Fil^U_2\Sym^n( V)]$, i.e.,
\[ 0\to \Sym^{n-1}(V)\tensor U\to \Sym^n( V)/\Fil^U_2\Sym^n( V)\to \Sym^n W\to 0\]
with the pull-back of
\[ 0\to \Sym^{n-1} W\tensor U\to\Sym^{n-1} W\tensor V\to\Sym^{n-1} W\tensor W\to 0\]
via the component
\[\Delta^{n-1,1}:\Sym^n W\to\Sym^{n-1} W\tensor W\]
of the comultiplication.
The same component of the comultiplication but for $\OSym^*( V)$ gives rise to a map
\[ \bar{\Delta}^{n-1,1}_n:\Sym^n V\xrightarrow{\Delta^{n-1,1}}\Sym^{n-1} V\tensor V\to\Sym^{n-1} W\tensor V\ .\]
This gives rise to a morphism of short exact sequences
\[\begin{CD}
0@>>>\Gr^U_1\Sym^n V@>>> \Sym^n V/\Fil^U_2\Sym^n V@>>>\Sym^n W@>>> 0\\
@.@VVV@VV\bar{\Delta}^{n-1,1}_nV@VV\Delta^{n-1,1}_n V\\
0@>>>\Sym^{n-1} W\tensor U@>>>\Sym^{n-1} W\tensor V@>>>\Sym^{n-1} W\tensor W@>>>0
\end{CD}\]
It remains to check that the induced map on kernels equals 
$g^{n,1}$, i.e., it is induced by multiplication.
By the explicit description, it suffices to check that the diagram
\[\xymatrix{
 V^{\tensor n-1}\tensor U\ar[r]^{\mu^n}\ar[d]_{\mu^{n-1}\tensor\id}&\Sym^n V\ar[d]^{\bar{\Delta}^{n,1}_n}\\
\Sym^{n-1} V\tensor V\ar[r]&\Sym^{n-1} W\tensor V
}\]
commutes.
Recall that multiplication and comultiplication are the ones of the symmetric coalgebra, and hence
\[ \mu^n= n!\pi^n\ ,\hspace{3ex}\mu^{n-1}=(n-1)!\pi^{n-1}\]
and 
\[ \Delta^{n-1,1}=\iota^{n-1,1}=\frac{1}{n}\sum_{i=1}^n\sigma_i\]
where
\[\sigma_i: V^{\tensor n}\to V^{\tensor n-1}\tensor  V\]
is the permutation that swaps the $i$-th factor into the last place and leaves the order otherwise intact. 
 Hence we can equivalently check the commutativity of the following diagram
\[\begin{CD}
 V^{\tensor n-1}\tensor U@>{\pi^n}>>\Sym^n V\\
@V{\pi^{n-1}\tensor\id} VV@VV{\sum_{i=1}^n\sigma_i}V\\
\Sym^{n-1} V\tensor V@>>>\Sym^{n-1} W\tensor V
\end{CD}\]

It suffices to check the same identity on the level of tensor algebras. 
By abuse
of notation
\[ \sigma_i(g_1\tensor g_2\tensor\dots\tensor g_n)=(g_1\tensor\dots\tensor\hat{g_i} \tensor\dots\tensor g_n)\tensor g_i\]
where $\hat{g_i}$ means that the factor is omitted.
The composition
\[  V^{\tensor n-1}\tensor U\subset V^{\tensor n}\xrightarrow{\sigma_i} V^{\tensor n-1}\tensor V\to W^{\tensor n-1}\tensor  V\]
vanishes for $i\neq n$ because it involves a factor $U\to W$.
Hence only $\sigma_n=\id$ contributes to
\[ V^{\tensor n-1}\tensor U\subset V^{\tensor n}\xrightarrow{\Delta^{n-1,1}} V^{\tensor n-1}\tensor V\to W^{\tensor n-1}\tensor  V\ .\]
This finishes the proof.
\end{proof}

\begin{theorem}\label{corcup}Let $\cata$ be a $\Q$-linear abelian symmetric tensor category. Moreover, let $T$ be a $\Q$-linear, tensor, symmetric, triangulated category and 
$q: D^b(A) \to T$ be a $\Q$-linear, tensor, symmetric and triangulated functor. 
Let
\[0\to U\to V\to W\to 0\]
be a short exact sequence in $\cata$. Suppose that 
$\Sym^2(q(U))=0$. 

Then there is  a canonical triangle in $T$ 
\[ \OSym^n(q( V)) \to \OSym^n(q( W))\xrightarrow{\cdot\cup[ V]}\OSym^{n-1}(q( W))\tensor q(U)[1] \ .\]
\end{theorem}
\begin{proof} By Proposition \ref{propcup} the short exact sequence in $\cata$
\[ 0\to \OSym^{*-1}(W)\tensor U\to \OSym^*(V)/\Fil^U_2\OSym^*(V)\to\OSym^*(W)\to 0\]
gives rise to the exact triangle
\[ \OSym^*(V)/\Fil^U_2\OSym^*(V)\to\OSym^*(W)\xrightarrow{\cdot\cup[V]}\OSym^{*-1}(q( W))\tensor q(U)[1] \ .\]
We apply $q$. It remains to show that
\[q(\Fil^U_2\OSym^*(V)=0\ .\]
This follows by descending induction from the system of triangles of Proposition \ref{qniisomo2}
\[ q(\Fil^U_{i+1}\Symn(V))\to q(\Fil^U_i\Sym^n(V))\to \Sym^i(q(U))\tensor\Sym^{n-i}q(W)\] 
and the vanishing of  $\Sym^i(q(U))$ for $i\geq 2$.
\end{proof}

\end{appendix}
\newpage

\bibliographystyle{alpha}	
\bibliography{masterbib}

\def\cprime{$'$}
\begin{thebibliography}{MVW06}

\bibitem[AHPL]{relative}
Giuseppe Ancona, Annette Huber, and Simon Pepin~Lehalleur.
\newblock On the relative motive of a commutative group scheme.
\newblock {\em To appear in Algebraic Geometry}.

\bibitem[AK02]{AK}
Yves Andr{\'e} and Bruno Kahn.
\newblock Nilpotence, radicaux et structures mono\"\i dales.
\newblock {\em Rend. Sem. Mat. Univ. Padova}, 108:107--291, 2002.
\newblock With an appendix by Peter O'Sullivan.

\bibitem[Bar55]{Bars}
Iacopo Barsotti.
\newblock Un teorema di struttura per le variet\`a gruppali.
\newblock {\em Atti Accad. Naz. Lincei. Rend. Cl. Sci. Fis. Mat. Nat. (8)},
  18:43--50, 1955.

\bibitem[Bea86]{Beau}
Arnaud Beauville.
\newblock Sur l'anneau de {C}how d'une vari\'et\'e ab\'elienne.
\newblock {\em Math. Ann.}, 273(4):647--651, 1986.

\bibitem[BG76]{BG}
A.~K. Bousfield and V.~K. A.~M. Gugenheim.
\newblock On {${\rm PL}$} de {R}ham theory and rational homotopy type.
\newblock {\em Mem. Amer. Math. Soc.}, 8(179):ix+94, 1976.

\bibitem[Blo76]{Bl1}
Spencer Bloch.
\newblock Some elementary theorems about algebraic cycles on abelian varieties.
\newblock {\em Inventiones Math.}, 37:215--228, 1976.

\bibitem[Bon10]{Bon}
M.~V. Bondarko.
\newblock Weight structures vs. {$t$}-structures; weight filtrations, spectral
  sequences, and complexes (for motives and in general).
\newblock {\em J. K-Theory}, 6(3):387--504, 2010.

\bibitem[BS13]{BSz}
Michel Brion and Tam{\'a}s Szamuely.
\newblock Prime-to-{$p$} \'etale covers of algebraic groups and homogeneous
  spaces.
\newblock {\em Bull. Lond. Math. Soc.}, 45(3):602--612, 2013.

\bibitem[BVK]{BVK}
Luca Barbieri-Viale and Bruno Kahn.
\newblock On the derived category of $1$-motives.
\newblock {\em To appear in {A}st\'erisque 216. Available on
  http://webusers.imj-prg.fr/~bruno.kahn/preprints/der1motast.pdf}.

\bibitem[CD09]{CD}
Denis-Charles Cisinski and Fr{\'e}d{\'e}ric D{\'e}glise.
\newblock Triangulated categories of mixed motives.
\newblock {\em Preprint: http://arxiv.org/abs/0912.2110}, 2009.

\bibitem[CD12]{CD2}
Denis-Charles Cisinski and Fr{\'e}d{\'e}ric D{\'e}glise.
\newblock Mixed {W}eil cohomologies.
\newblock {\em Adv. Math.}, 230(1):55--130, 2012.

\bibitem[Che60]{Chev}
C.~Chevalley.
\newblock Une d\'emonstration d'un th\'eor\`eme sur les groupes alg\'ebriques.
\newblock {\em J. Math. Pures Appl. (9)}, 39:307--317, 1960.

\bibitem[Con02]{Con}
Brian Conrad.
\newblock A modern proof of {C}hevalley's theorem on algebraic groups.
\newblock {\em J. Ramanujan Math. Soc.}, 17(1):1--18, 2002.

\bibitem[Del74]{DeH3}
Pierre Deligne.
\newblock Th\'eorie de {H}odge. {III}.
\newblock {\em Inst. Hautes \'Etudes Sci. Publ. Math.}, (44):5--77, 1974.

\bibitem[Del02]{deltens}
P.~Deligne.
\newblock Cat\'egories tensorielles.
\newblock {\em Mosc. Math. J.}, 2(2):227--248, 2002.
\newblock Dedicated to Yuri I. Manin on the occasion of his 65th birthday.

\bibitem[DM91]{DeMu}
Christopher Deninger and Jacob Murre.
\newblock Motivic decomposition of abelian schemes and the {F}ourier transform.
\newblock {\em J. Reine Angew. Math.}, 422:201--219, 1991.

\bibitem[EW13]{ward}
Stephen Enright-Ward.
\newblock {\em The Voevodsky motive of a rank one semiabelian variety}.
\newblock PhD thesis, Universit\"at Freiburg, Germany, 2013.

\bibitem[Gul06]{Guletskii}
Vladimir Guletski{\u\i}.
\newblock Finite-dimensional objects in distinguished triangles.
\newblock {\em J. Number Theory}, 119(1):99--127, 2006.

\bibitem[HK06]{HK}
Annette Huber and Bruno Kahn.
\newblock The slice filtration and mixed {T}ate motives.
\newblock {\em Compos. Math.}, 142(4):907--936, 2006.

\bibitem[Ivo07]{Ivo}
Florian Ivorra.
\newblock R\'ealisation {$l$}-adique des motifs triangul\'es g\'eom\'etriques.
  {I}.
\newblock {\em Doc. Math.}, 12:607--671, 2007.

\bibitem[Kim05]{Kim}
Shun-Ichi Kimura.
\newblock Chow groups are finite dimensional, in some sense.
\newblock {\em Math. Ann.}, 331(1):173--201, 2005.

\bibitem[Kin98]{Ki}
Guido Kings.
\newblock Higher regulators, {H}ilbert modular surfaces, and special values of
  {$L$}-functions.
\newblock {\em Duke Math. J.}, 92(1):61--127, 1998.

\bibitem[K{\"u}n94]{Ku1}
K.~K{\"u}nnemann.
\newblock On the {C}how motive of an abelian scheme.
\newblock In {\em Motives}, volume 55.1 of {\em Proceedings of Symposia in pure
  mathematics}, pages 189--205. American mathematical society, 1994.

\bibitem[Lod92]{Loday}
Jean-Louis Loday.
\newblock {\em Cyclic homology}, volume 301 of {\em Grundlehren der
  Mathematischen Wissenschaften [Fundamental Principles of Mathematical
  Sciences]}.
\newblock Springer-Verlag, Berlin, 1992.
\newblock Appendix E by Mar{\'{\i}}a O. Ronco.

\bibitem[Maz04]{Mazza}
Carlo Mazza.
\newblock Schur functors and motives.
\newblock {\em $K$-Theory}, 33(2):89--106, 2004.

\bibitem[MVW06]{MVW}
C.~Mazza, V.~Voevodsky, and C.~Weibel.
\newblock {\em Lecture Notes on Motivic Cohomology}.
\newblock Clay Mathematics Monographs. American Mathematical Society, 2006.

\bibitem[Org04]{Org}
Fabrice Orgogozo.
\newblock Isomotifs de dimension inf\'erieure ou \'egale \`a un.
\newblock {\em Manuscripta Math.}, 115(3):339--360, 2004.

\bibitem[O'S05]{OS}
Peter O'Sullivan.
\newblock The structure of certain rigid tensor categories.
\newblock {\em C. R. Math. Acad. Sci. Paris}, 340(8):557--562, 2005.

\bibitem[Sch94]{Scholl}
A.~J. Scholl.
\newblock Classical motives.
\newblock In {\em Motives ({S}eattle, {WA}, 1991)}, volume~55 of {\em Proc.
  Sympos. Pure Math.}, pages 163--187. Amer. Math. Soc., Providence, RI, 1994.

\bibitem[{\v{S}}er74]{Sher}
A.~M. {\v{S}}ermenev.
\newblock Motif of an {A}belian variety.
\newblock {\em Funckcional. Anal. i Prilo\v zen.}, 8(1):55--61, 1974.

\bibitem[SS03]{SS}
Michael Spie{\ss} and Tam{\'a}s Szamuely.
\newblock On the {A}lbanese map for smooth quasi-projective varieties.
\newblock {\em Math. Ann.}, 325(1):1--17, 2003.

\bibitem[Sug14]{Sug}
Rin Sugiyama.
\newblock Motivic homology of a semiabelian variety over a perfect field.
\newblock {\em Doc. Math. (19)}, pages 1061--1084, 2014.

\bibitem[SV96]{SV1}
Andrei Suslin and Vladimir Voevodsky.
\newblock Singular homology of abstract algebraic varieties.
\newblock {\em Invent. Math.}, 123(1):61--94, 1996.

\bibitem[Voe00]{TMF}
Vladimir Voevodsky.
\newblock Triangulated categories of motives over a field.
\newblock In {\em Cycles, transfers, and motivic homology theories}, volume 143
  of {\em Ann. of Math. Stud.}, pages 188--238. Princeton Univ. Press,
  Princeton, NJ, 2000.

\bibitem[Voe02]{VoeChow}
Vladimir Voevodsky.
\newblock Motivic cohomology groups are isomorphic to higher {C}how groups in
  any characteristic.
\newblock {\em Int. Math. Res. Not.}, (7):351--355, 2002.

\bibitem[Voe10]{Canc}
Vladimir Voevodsky.
\newblock Cancellation theorem.
\newblock {\em Doc. Math.}, (Extra volume: Andrei A. Suslin sixtieth
  birthday):671--685, 2010.

\bibitem[Wil09]{WildChow}
J.~Wildeshaus.
\newblock Chow motives without projectivity.
\newblock {\em Compos. Math.}, 145(5):1196--1226, 2009.

\end{thebibliography}
\end{document}